\documentclass[11pt]{amsart}
\usepackage{amssymb,amsmath}
\usepackage{eucal}
\usepackage{epstopdf}
\usepackage{mathrsfs}
\usepackage{enumerate}
\usepackage{fullpage} 
\usepackage{setspace}
\usepackage{xcolor}
\usepackage{dsfont}

\usepackage{todonotes}
\usepackage{hyperref}

\usepackage{mathtools}
\DeclarePairedDelimiter{\ceil}{\lceil}{\rceil}

\newcommand{\al}{\alpha}
\newcommand{\be}{\beta}
\newcommand{\mi}{\mathrm{i}}
\newcommand{\p}{\mathrm{p}}
\newcommand{\e}{\varepsilon}

\newcommand{\N}{\mathbb{N}}
\newcommand{\R}{\mathbb{R}}

\newcommand{\I}{\mathcal{I}}
\newcommand{\HH}{\mathcal{H}}
\newcommand{\J}{\mathcal{J}}

\newcommand{\oh}{\mbox{$\frac{1}{2}$}}
\newcommand{\chiq}{ \sideset{}{^*}\sum_{d \, (\mathrm{mod} \, q)}}
\newcommand{\chiqs}{ \sum^*_{d \, (\mathrm{mod} \, q)}}

\DeclareFontFamily{OT1}{rsfs}{}
\DeclareFontShape{OT1}{rsfs}{n}{it}{<-> rsfs10}{}
\DeclareMathAlphabet{\mathscr}{OT1}{rsfs}{n}{it}

\newtheorem{prop}{Proposition}[section]
\newtheorem{thm}[prop]{Theorem}

\newtheorem{lem}[prop]{Lemma}
\newtheorem{defn}{Definition}

\newtheorem*{defn*}{Definition}
\newtheorem{conj}{Conjecture}

\numberwithin{equation}{section}

\renewcommand{\Re}{{\mathfrak{Re}}}
\renewcommand{\Im}{{\mathfrak{Im}}}

\renewcommand{\imath}{i}

\DeclareMathOperator{\GL}{GL}

\allowdisplaybreaks

\begin{document}

\title[]{The eighth moment of the Riemann zeta function}
\author{Nathan Ng}

\address{Department of Mathematics and
Computer Science, University of Lethbridge, Lethbridge, AB Canada T1K 3M4}
\email{nathan.ng@uleth.ca}

\author{Quanli Shen}

\address{Department of Mathematics and
Computer Science, University of Lethbridge, Lethbridge, AB Canada T1K 3M4}
\email{quanli.shen@uleth.ca}

\author{Peng-Jie Wong}

\address{National Center for Theoretical Sciences\\
No. 1, Sec. 4, Roosevelt Rd., Taipei City, Taiwan}
\email{pengjie.wong@ncts.tw}


\dedicatory{In memory of Aleksandar Ivi\'{c}.}

\date{\today}


\begin{abstract}
In this article, we establish an asymptotic formula for the eighth moment of the Riemann zeta function, assuming the Riemann hypothesis and a quaternary additive divisor conjecture.   This builds on the work of the first author \cite{Ng} on the sixth moment of the Riemann zeta function and the works of Conrey-Gonek \cite{CG} and Ivi\'{c}  \cite{Iv2,Iv3}.
A key input is a sharp bound for a certain shifted moment of the Riemann zeta function, assuming the Riemann hypothesis.  
\end{abstract}

\maketitle


\section{Introduction}
This article concerns the eighth moment $I_4(T)$ of the Riemann zeta function $\zeta(s)$, where 
\begin{equation}
  \label{IkT}
  I_k(T) = \int_{0}^{T} |\zeta( \oh +\mi t)|^{2k} \, dt
\end{equation}
denotes the $2k$-th moment of the Riemann zeta function.  There is a long and extensive history of research on the moments \eqref{IkT}.
Fundamental results may be found in the books \cite{Ti}, \cite{Iv}, \cite{Mo}, and \cite{Iv2}.  Regarding its size, it was proven by Heath-Brown \cite{HB81} and Ramachandra \cite{Ram78,Ram95} 
that under the Riemann hypothesis, for any real $k\ge 0$,
\begin{equation}
  \label{lb}
 I_k(T) \gg T (\log T)^{k^2}.
\end{equation}
Recently, {Radziwi\l\l} and Soundararajan \cite{RaSo} showed that \eqref{lb} holds unconditionally for $k\ge 1$.

Under the Riemann hypothesis, Harper \cite{Ha} established that 
\begin{equation}
   \label{ub}
  I_k(T) \ll T(\log T)^{k^2}.
\end{equation}
This improved  an earlier result of Soundararajan \cite{So} in which the right of \eqref{ub} was larger by a factor of $(\log T)^{\e}$. 
Unconditionally, it is  known that 
\begin{equation*}
  I_k(T) \ll T^{M+\varepsilon}\text{ with }
   M\le\begin{cases}
 1 +  \frac{k-2}{4} & \text{ if $2\le k \le 6$;} \\
 2 +  \frac{3(k-6)}{11} & \text{ if $6\le k \le \frac{178}{26}$;} \\ 
1 +   \frac{35(k-3)}{108} & \text{ if $k \ge \frac{178}{26}$}.
\end{cases}
\end{equation*}
(see \cite[Theorem 8.3]{Iv}).

Keating and Snaith \cite{KS}, using a random matrix model, conjectured that for $k \in \mathbb{N}$, 
\begin{equation}
  \label{ksasymp}
I_k(T) \sim \frac{g_k a_k}{(k^2)!} T (\log T)^{k^2},
\end{equation}
as $T\rightarrow \infty$,
where 
\begin{equation}
  \label{gkak}
  g_k =  k^2! \prod_{j=0}^{k-1} \frac{ j!}{(k+j)!}  
\quad \text{ and } \quad 
  a_k = \prod_{p} \left( 1-\frac{1}{p} \right)^{k^2} \sum_{m=0}^{\infty} \binom{k+m-1}{m}^2 
   p^{-m}.
\end{equation}
In 2005, Conrey \emph{et al.} \cite{CFKRS}, using a heuristic argument with the approximate functional equation, conjectured that for $k \in \mathbb{N}$, 
\begin{equation}
  \label{fullasymptotic}
  I_k(T) =  T \mathcal{P}_{k^2}(\log T)  +o(T),
\end{equation}
where $\mathcal{P}_{k^2}$ is a certain polynomial of degree $k^2$.
In 1918, Hardy and Littlewood \cite{HL}   established the asymptotic \eqref{ksasymp} for the case $k=1$, and in 1926, Ingham \cite{In} established the case $k=2$. 
The asymptotic for the case $k=3$, was first conjectured by Conrey and Ghosh \cite{CGh}. Conrey and Gonek \cite{CG} 
gave a heuristic  argument, based on the conjectural asymptotics for additive divisor sums, which suggests that \eqref{ksasymp} holds for the cases $k=3,4$.
Recently, building on the works of Conrey-Gonek \cite{CG} and Hughes-Young \cite{HY}, the first author \cite{Ng} showed that a certain conjecture for smoothed ternary additive divisor sums implies that \eqref{ksasymp} and \eqref{fullasymptotic} are true in the case $k=3$.   In this article, we shall extend the ideas in \cite{Ng} to the case $k=4$.

A general approach in evaluating \eqref{IkT} is to write $|\zeta( \oh +\mi t)|^{2k} = \zeta(\oh + \mi t)^k  \zeta(\oh - \mi t)^k$.  This leads naturally to the 
consideration of the $k$-th divisor functions $\tau_k$ via the Dirichlet series $\zeta^k(s) = \sum_{n=1}^{\infty} \tau_k(n) n^{-s}$.  (A more elementary definition
is $\tau_k(n) = \# \{ (d_1, \ldots, d_k) \in \mathbb{N}^k \ | \ d_1 \cdots d_k =n \} \text{ for } n \in \mathbb{N}$.)
It turns out that the moments $I_k(T)$, for $k \in \mathbb{N}$,  are intimately related to the additive divisor sums 
\begin{equation}
  \label{Dkxr}
  D_{k}(x;r) = \sum_{n \le x} \tau_k(n) \tau_k(n+r)
\end{equation}
for $ x >0$. 
In \cite{In}, Ingham required an upper bound for $D_2(x;r)$ in order to establish the asymptotic \eqref{ksasymp} for $k=2$.
Heath-Brown \cite{HB} pushed this much further and established \eqref{fullasymptotic} in the case $k=2$, by using a more precise asymptotic 
formula for $D_2(x;r)$.   Deshouilliers and Iwaniec \cite{DI} and then Motohashi \cite{Mo} improved the error term in  \eqref{ksasymp}, for $k=2$,
by making use of the spectral theory of automorphic forms, in particular, Kuznetsov's formula.  For $k >2$,  Conrey and Gonek \cite{CG} and Ivi\'{c} \cite{Iv} studied
the connections between $D_k(x;r)$ and asymptotic and bounds for $I_k(T)$.  
It turns out that it is more convenient to consider more general sums of the shape 
\begin{equation}
  \label{DfIJ}
  D_{f;k,\ell}(r) = \sum_{m-n=r} \tau_k(m) \tau_{\ell}(n) f(m,n),
\end{equation}
where $k, \ell \in \N$, $r \in \mathbb{Z} \setminus \{ 0 \}$, and $f$ is an arbitrary function of two variables.  Observe that the sum in  \eqref{Dkxr} is a special case of \eqref{DfIJ}.
In our work, we shall assume 
$f$ is a smooth function which satisfies conditions \eqref{fsupport} and \eqref{fcond} below.
 Duke, Friedlander, and Iwaniec \cite{DFI} introduced these sums when $k=\ell=2$ (with the more general   summation condition $am-bn=r$)  to study the subconvexity problem for $\GL_2$. An advantage of the sums  $ D_{f;k,\ell}(r)$  is that due to the presence of the smooth function $f$,
 the Poisson summation formula or Voronoi summation formula can readily be applied. The sums \eqref{DfIJ} provide much greater flexibility and are 
 more useful in moment problems than the classical unsmoothed sums \eqref{Dkxr}.

The function $f$ in \eqref{DfIJ} satisfies a number of properties as follows. The partial derivatives of $f$ satisfy growth conditions. 
That is,  there must exist positive $X,Y,$ and $P$ such that 
\begin{equation}
  \label{fsupport}
  \text{support}(f) \subset [X,2X] \times [Y,2Y],
\end{equation}
and the partial derivatives of $f$ satisfy
\begin{equation} \label{fcond} 
 x^{m} y^{n} f^{(m,n)}(x,y) \ll_{m,n}
P^{m+n}. 
\end{equation}  

We now describe a conjectural formula for the sums $D_{f;k,\ell}(r)$.  We first need to introduce an arithmetic function that 
appears in the conjecture. 
\begin{defn}
Let $k \in \N$. 
 The multiplicative function $n \mapsto g_k(s,n)$ is given by
 \begin{equation*}
   g_k(s,n)
=
\prod_{p^{\alpha} \mid \mid n} 
  \frac{\sum_{j=0}^{\infty} \frac{\tau_k(p^{j+\alpha})}{p^{js}}}{ \sum_{j=0}^{\infty} \frac{\tau_k(p^j)}{p^{js}}}.
\end{equation*}
In other words, for $n \in \mathbb{N}$, we have 
$
  \sum_{m=1}^{\infty} \frac{\tau_k(nm) }{m^s}
  =  g_k(s,n) \zeta(s)^k$.
  
The multiplicative function $n \mapsto G_k(s,n)$ is given by 
\begin{equation}
  \label{Gkdefn}
  G_k(s,n) =  \sum_{d \mid n} \frac{\mu(d) d^s}{\phi(d)}    \sum_{e \mid d} \frac{\mu(e)}{e^s} g_k \left(s,\frac{ne}{d} \right).
\end{equation}
\end{defn}
It can be shown that for $s \approx 1$, that $G_k(s,p^j) \approx  \tau_k(p^j) -p^{s-1} \tau_k(p^{j-1})$ (see Lemma \ref{G4formulae} below and \cite[Lemma 5.4, p. 521]{Ng0} for precise statements).   

\begin{conj}[Quaternary additive divisor conjecture] \label{divconj}
There exists $C >0$ for which 
the following  holds.  Let $\varepsilon_0$ and $\varepsilon'$ be arbitrarily small positive constants. 
Let $P > 1$, and let $X,Y > \frac{1}{2}$ satisfy $Y \asymp X$.  Let $f$
be a smooth function satisfying \eqref{fsupport} and \eqref{fcond}.
Then, in those cases where $X$ is sufficiently large (in absolute terms), one has
  \begin{align}  \label{Dfr}
 \begin{split}
      D_{f;4,4}(r)
        & =     \frac{1}{(2\pi \mi)^2} \int_{\mathcal{B}_1} \int_{\mathcal{B}_2}  \zeta^4 (z_1)\zeta^4 (z_2) 
\sum_{q=1}^{\infty}  \frac{c_q (r) G_4 (z_1,q)   G_4 (z_2,q)}{ q^{z_1+z_2}}   \\
&\times  \int_0^{\infty} f(x,x-r) x^{z_1 -1 } (x-r)^{z_2 -1} \, dx  \,dz_2 \,d z_1  
 + O ( P^{C} X^{\frac{1}{2}+\varepsilon_0} ), 
  \end{split}
  \end{align}
 uniformly for $1 \le |r| \ll X^{1-\varepsilon'}$, where for $i=1,2$, $\mathcal{B}_i = \{ z_i \in \mathbb{C} \ | \ |z_i-1| = r_i \} \subset \mathbb{C}$ 
 are circles, centred at $1$, of radii $r_i \in (0, \frac{1}{10})$,  and $c_q(r) =  \chiqs e ( \tfrac{-dr}{q} )$ is the Ramanujan sum. 
\end{conj}

This form of the additive divisor conjecture is worked out in Appendix  1 (Section \ref{ADC}) below.  The derivation is based on Duke, Friedlander, and Iwaniec's $\delta$-method which 
is a form of the circle method.  The Ramanujan sums $c_q(r)$ arise as one is detecting the additive condition $m=n+r$.   The arithmetic functions $G_k(z,n)$ also arise
in this derivation.  For full details, see Appendix 1 and \cite[pp. 589-591]{CG}.   The error term is expected based on Aryan's work \cite{Ar} in the case $k=\ell=2$,
Blomer's work \cite{Bl} on shifted convolutions of modular forms, and a conjecture for $D_{k,\ell}(x,r)$ (see Conjecture \ref{kldivconj} in Appendix 1). 
Oddly, there are  currently no numerical confirmations of this conjecture.

Our main theorem is the following asymptotic formula for the eighth moment of the Riemann zeta function, assuming the Riemann hypothesis and the quaternary additive divisor conjecture.

\begin{thm} \label{8thmoment}
Assume the Riemann hypothesis and Conjecture \ref{divconj} (the quaternary additive divisor conjecture) are true. Then we have
\begin{equation}
  \label{I4Tasymptotic}
  I_4(T) \sim \frac{24024 a_4}{16!}T(\log T)^{16}, 
\end{equation}
as $T \to \infty$, where $a_4$ is defined in \eqref{gkak}.
\end{thm}

Our proof of Theorem \ref{8thmoment} relies in a fundamental way on the following result on shifted moments of the Riemann zeta function established by the authors in the companion article \cite{NSW}.
\begin{thm}\label{refined_Harper}
Let $k\ge \frac{1}{2}$ and assume the Riemann hypothesis. 
We define for $T > 1$ and $t_0 \in \R$, 
\begin{equation*}
\mathcal{G}(T,t_0) =\begin{cases}
 \min \left\{\frac{1}{|2t_0|}, \log T \right\}  & \text{ if $|t_0|\le \frac{1}{200}$;} \\
\log(2+ |2t_0|) & \text{ if $|t_0|>\frac{1}{200}$}. 
\end{cases}
\end{equation*}
Then for $T$ sufficiently large and $|t_0|\le 0.5T$, we have
\begin{equation}
  \label{shiftedmoment}
\int_0^{T} |\zeta(\tfrac{1}{2}+\mi t+\mi t_0)|^{k}|\zeta(\tfrac{1}{2}+\mi t-\mi t_0)|^{k} \,dt \ll_k T(\log T)^{\frac{k^2}{2}}\mathcal{G}(T,t_0)^{\frac{k^2}{2}}. \footnote{Note $\mathcal{G}(T,t_0) =  \mathcal{F}(T,t_0,-t_0)$ where  $\mathcal{F}(T,\alpha_1,\alpha_2)$ is defined in \cite{NSW}.}
\end{equation}
\end{thm} 
There are a number of key differences in establishing Theorem \ref{8thmoment} with the argument in \cite{Ng}, which deals with the sixth moment.  \\
\noindent {\bf Remarks.} 
\begin{enumerate}
\item  In \cite{Iv3} Ivi\'{c}  showed that an averaged form of the quaternary additive divisor conjecture implies that $I_4(T) \ll T^{1+\e}$ for any $\e >0$
(See Theorem 1, Corollary 1, and Corollary 2 of \cite{Iv3}).
Conrey and Gonek \cite{CG} have given a heuristic argument which shows that the unsmoothed quaternary additive divisor conjecture 
implies
$I_4(T) \sim  \frac{24024 a_4}{16!}T(\log T)^{16}$.
The work in \cite{CG} is heuristic since in many steps error terms are not bounded. 
There is no smoothing in any of the arguments in \cite{CG} including in the approximate functional equation and in the interval of integration of $I_4(T)$. 
In practice it is usually harder to obtain  good error terms without the  use of smooth weights.

\item In this argument, we only ``win by an $\e$" and no lower order terms in the main term are obtained.  In contrast, in \cite{Ng}, the full main term 
with a power savings error term is obtained for $I_3(T)$ assuming the ternary additive divisor conjecture.
\item Harper's result \eqref{ub} shows that the Riemann hypothesis implies $I_4(T) \ll T (\log T)^{16}$.  
Our result shows that the additional assumption of the quaternary additive divisor conjecture upgrades this to an asymptotic 
formula. 
 It seems likely that the methods of this article would allow one to prove that the quaternary additive divisor conjecture implies 
\begin{equation*}
    I_4(T) \ge \frac{24024 a_4}{16!}T(\log T)^{16} (1+o(1)). 
\end{equation*}
That is, if we do not assume the Riemann hypothesis, then we can obtain a lower bound of the correct order of magnitude.
The proof of such a result would follow ideas from \cite[Theorem 1.1]{SY} and also \cite[Theorem 1.2]{Sh}, where sharp lower bounds for the second moment of quadratic twists of a cusp form and for the fourth moment of quadratic Dirichlet $L$-functions are obtained unconditionally. 
\item  Shifted moments of the Riemann zeta function were studied by Chandee, and she formulated a conjecture on them in \cite[Conjecture 1]{Ch}.
In an accompanying article \cite{NSW}, assuming the Riemann hypothesis, we prove Chandee's conjecture in the case of two shifts. 
This is encapsulated in Theorem \ref{refined_Harper} above.  Our theorem is based on important ideas of Harper \cite{Ha} and includes the main theorem in \cite{Ha} as a special case. 
Crucially, we show that Harper's method still works with the shifting parameter $t_0$ allowed to be a function of $T$.  Incidentally, we note that Theorem \ref{refined_Harper}
includes Harper's theorem (the bound \eqref{ub}) as a special case by taking $t_0=0$. 

\item  The approximate functional equation for $|\zeta(\tfrac{1}{2}+\mi t)|^8$ expresses this in terms of a product of  two Dirichlet polynomials  each of length essentially
$T^{2}$.  
A technical lemma  is proven that reduces  the length of this Dirichlet polynomial by a factor of $T^{\e}$ for an arbitrarily small $\e$.
This will allow us to deal with a polynomial of length $T^{2-\e}$ and  the key point is that square root cancellation in additive divisor sums gives an error of size $O((T^{2-\e})^{\frac{1}{2}})=O(T^{1-0.5 \e})$ which is an acceptable error term.
In order to make this work, we need a bound on  shifted moments of the Riemann zeta function of the type in \eqref{shiftedmoment}.  
(The upper bound method of  Soundararajan \cite{So} does not allow us to obtain the asymptotic \eqref{I4Tasymptotic}.)
This idea of reducing the length of a Dirichlet polynomial
was used first in \cite{SY} and later in the articles of \cite{CL} and \cite{Sh}. Though in those articles the length of the Dirichlet polynomial is only reduced
by a factor of $(\log T)^{A}$ for an arbitrarily large $A$ whereas we reduce the length by a factor of $T^{\varepsilon}$ for an arbitrarily small $\e$. 

\item The other main difference with \cite{Ng} is that we deal directly with the eighth moment instead of the shifted eighth moment.  
The reason we did this is that we  encountered a technical difficulty in  trying to generalize the argument of \cite{Ng} to the case $|\I|=|\J|=4$. 
Here $\I =\{ a_1, \ldots, a_k \}$ and $\J = \{b_1, \ldots, b_k \}$ are the shifting parameters (see just below equation (1.10) in \cite{Ng}). We hope to revisit this in the future. 
The approach using shifted moments can be advantageous since the main term arises from many polar terms which arise from the 
simple poles of the Dirichlet series under consideration.  Fortunately, these polar terms can be shown to cancel by Lemma 2.5.1 of \cite{CFKRS}. 
When no shifts are introduced, the main term arises from multiple poles of the Dirichlet series under consideration.  In this case, multiple residues 
need to be computed.

\end{enumerate}


\noindent {\bf Conventions and notation.} 
Given two functions $f(x)$ and $g(x)$, we shall interchangeably use the notation  $f(x)=O(g(x))$, $f(x) \ll g(x)$, and $g(x) \gg f(x)$  to mean there exists $M >0$ such that $|f(x)| \le M |g(x)|$ for all sufficiently large $x$. 
We write $f(x) \asymp g(x)$ to mean that the estimates $f(x) \ll g(x)$ and $g(x) \ll f(x)$ simultaneously hold.  
If we write $f(x)=O_{a_1, \ldots, a_\ell}(g(x))$, $f(x) \ll_{a_1, \ldots, a_\ell} g(x)$, or $f(x) \asymp_{a_1, \ldots, a_\ell} g(x)$ for real numbers $a_1, \ldots, a_\ell$, then we mean that the 
corresponding implied constants depend on $a_1, \ldots, a_\ell$.  
In addition, $B$ shall denote a positive constant, which may be taken arbitrarily large and which may change from line to line. 
The letter $p$ will always be used to denote a prime number.
For a function $\varphi : \mathbb{R}^{+} \times \mathbb{R}^{+} \to \mathbb{C}$, $\varphi^{(m,n)}(x,y) = \frac{\partial^m}{\partial x^m}
\frac{\partial^n}{\partial y^n} \varphi(x,y)$.  
The integral notation $\int_{(c)} f(s) \,ds$ for a complex function $f(s)$ and $c \in \mathbb{R}$
 will be used frequently and is defined by  the following contour integral
\begin{equation*}
   \label{intc}
  \int_{(c)}f(s) \,ds = \int_{c-\mi\infty}^{c + \mi \infty} f(s) \, ds. 
\end{equation*}
In this article, we shall consider $s \in \mathbb{C}$ and  usually write its real part as $\sigma =\Re(s)$. 
Throughout this article, we will often use the fact that $\omega(t)$ has support in
$[c_1 T, c_2T]$ so that $t \asymp T$.   Throughout this article we use the notation $\mi$ to denote the imaginary number which satisfies $\mi^2 =-1$
and $i$ shall be used as an integer variable.

\section{Preliminary lemmas}

 The following proposition is a straightforward generalization of \cite[Proposition 2.1, p. 209]{HY} 
 and \cite{Ng}, which 
 gives approximate functional equations for  $|\zeta(\oh+ \mi t)|^k$ 
 with $k=4,6$
 (see also \cite[Lemma 1]{HB}).

\begin{prop} \label{afe}
Let $G(s)$ be an entire function satisfying 
\begin{align}
|G(s)| \ll \exp(-c |t|)  
\label{bdG}
\end{align}
for some (explicit) $c$ and $\Re(s)$ in an appropriate interval.\footnote{For instance, one may choose $G(s)= \exp(s^2)$.}  We define
\begin{equation*}
   V_{t}(x) = \frac{1}{2 \pi \mi} \int_{(1)} \frac{G(s)}{s} g(s,t) x^{-s} \,ds,
\end{equation*}
where 
\begin{equation*}
  g(s,t) = 
  \left( 
  \frac{  \Gamma \left( \frac{1}{2} (\tfrac{1}{2}+s+ \mi t) \right) }{ \Gamma  \left( \frac{1}{2} (\tfrac{1}{2}+ \mi t) \right)}
     \frac{   \Gamma \left(\frac{1}{2}( \tfrac{1}{2}+s- \mi t) \right)}{\Gamma \left(  \frac{1}{2} (\tfrac{1}{2}- \mi t) \right)}
  \right)^4.
\end{equation*}
Then for any constant $B>0$, we have 
\begin{align}\label{smoothafe}
  |\zeta(\oh+ \mi t)|^8  
  &=2 \sum_{m,n=1}^{\infty} \frac{\tau_4(m) \tau_4(n)}{(mn)^{\frac{1}{2}}}
\left( \frac{m}{n} \right)^{- \mi t}
V_{t}( \pi^4 mn)  
 + O( ( 1+|t|)^{-B} ).
 \end{align}
\end{prop}
For $t \asymp T$, one can see the terms $(m,n)$ that contribute to this sum satisfy $mn \ll T^{4+\e}$. 
In this article we shall use various facts about $g(s,t)$.  In addition, we shall encounter the function $\widetilde{V}_{t}(x)$ defined by 
\begin{equation}
  \label{tildeV}
  \widetilde{V}_{t}(x) = \frac{1}{2 \pi \mi} \int_{(1)} \frac{G(s)}{s} g(s,t) \left( \frac{U}{t } \right)^{4s} x^{-s} \,ds,
\end{equation}
where 
\begin{equation}  \label{def:U}
  U =T^{1-\e}.
\end{equation}
 We have the following lemma providing bounds for $ g(s,t) $, its partial derivatives, and $\widetilde{V}_{t}(x)$.

\begin{lem} 
\label{lem:bd-g}
Let $A$ be a positive constant. \\
(i)  For $0 \le \Re(s) \le A$, we have 
\[
 g(s,t) =  \left( \frac{t}{2} \right)^{4s}   \left(1 + O \left( \frac{|s|^2+1}{t}  \right)\right).
 \] 
(ii)  Let $\varepsilon_0 >0$ and $i \ge 0$. For $\Re(s)=\varepsilon_0$, $|\Im(s)| \le \sqrt{T}$, and $c_1 T \le t \le c_2 T$, with $0 < c_1 < c_2$,
\begin{equation}
  \label{gstiderivatives}
   \frac{d^{i}}{dt^{i}} g(s,t) \ll_{i,\varepsilon}  |s|^{i} T^{4\varepsilon_0-i}.
\end{equation}
(iii) For $x > U^4$,
$\displaystyle \widetilde{V}_{t}(x) = O \left( \left( \frac{U^4}{x} \right)^{A} \right) \,$.
\end{lem}
\begin{proof} As the proofs of Lemma \ref{lem:bd-g}(i) and (ii) are very similar to Lemma 2.2(ii) and (iii) in \cite{Ng}, we shall omit them.  Now, we prove Lemma 2.2(iii).
Moving the line of the integration in $\widetilde{V}_{t}(x) $ to $\Re(s) = A$,  we see that $\widetilde{V}_{t}(x)$ becomes
\begin{align*}
 \frac{1}{2 \pi \mi} \int_{(A)} \frac{G(s)}{s} g(s,t) \left( \frac{U}{t } \right)^{4s} x^{-s} \,ds
\ll \int_{A-\mi \infty}^{A+\mi\infty}  \frac{|G(s)|}{|s|} \left( \frac{U^4}{16x}\right)^A  \left( 1 + \frac{|s|^2+1}{t}\right) \,|ds|
\ll \left( \frac{U^4}{x}\right)^A.
\end{align*}
\end{proof}

\section{Truncating the sum}

As there are many small parameters involved in our argument, for the convenience of the reader, we track of them here. 
\begin{itemize}
\item $\varepsilon$ is used in the definition $U=T^{1-\varepsilon}$ in \eqref{def:U}.
\item $\e_0$ is an arbitrary small parameter that differs line by line. It appears in Conjecture  \ref{divconj}, \eqref{gstiderivatives}, Lemma \ref{lem:Z1}, and the integral above \eqref{equ:diagonal02}.

\item  $\varepsilon_1$ appears in \eqref{fstardefn}, the definition of $f^{*}$, where we integrate over $\Re(s) =\e_1$. Furthermore, $\e_1$ is used in the condition $|\log(\tfrac{m}{n})| \ll T_{0}^{-1+\varepsilon_1}$ (see \eqref{IM-2nd-exp}). In addition, we choose  $P=T^{1+ \e_1} T_{0}^{-1}$ (to apply Conjecture \ref{divconj}) in Section \ref{apply-divconj}. 
\item $\varepsilon_2$ is introduced for cutting the sum at $mn \ll U^{1+\varepsilon_2}$ in Section \ref{smooth-partition}. 
\item $\e_3$ is for the line $\Re(s) = \e_3$. In Section \ref{expandingr}, we move the integral in $ \tilde{I}_{M,N}
$ from $\Re(s) = \e_1$ to $\Re(s) =  \e_3$. In addition, we move the integral in \eqref{IONO} from $\Re(s) =  1$ to $\Re(s) =  \e_3$, in Section \ref{e3-line}. We assume $0 < \delta< \e_3 < 0.15$, where $\delta$ is a parameter as described below.
\item $\e'$ is related to range of $r$ in Conjecture \ref{divconj}; more precisely, $|r| \le M^{1-\e'}$. 
\item $\eta$ is a parameter used to define $\omega^{+}(t)$ and $\omega^{-}(t)$ in the proof of Theorem \ref{8thmoment} located below \eqref{IomegadecompB}. 
\item $\eta_0$ is a parameter introduced below \eqref{upbd-2} and \eqref{upbd-3} in the proof of Theorem \ref{8thmoment} to set $j$.

\item  $\delta$ appears in Lemma \ref{Stirling} for the condition $|a_{i_1}|, |b_{i_2}| \leq \delta$ and in the proof of Lemma \ref{FancyG} as the radius of the circle $C(0,\delta)$ centred at $0$. We assume $0<\delta<\frac{1}{10}$.
\item $\delta'$ is used for $|u_i| < \delta'$ in Proposition \ref{Hlemma}. 
\item $r_1, r_2$ are the radii for $\mathcal{B}_1$, $\mathcal{B}_2$, respectively, where $|r_1|, |r_2| \leq \delta< \frac{1}{10}$ (see Conjecture \ref{divconj}).

\end{itemize}

We begin by evaluating the smoothed eighth moment 
\begin{equation*}
 I_{\omega} := \int_{-\infty}^{\infty}  \omega(t)  |\zeta(\oh+ \mi t)|^8  
 \, dt
\end{equation*}
where 
$\omega$ is a function from $\mathbb{R}$ to $\mathbb{C}$  that satisfies the following three properties:
\begin{align}
 &
 \text{(i) }\omega \text{ is smooth}; \nonumber\\
&\text{(ii) }\label{cond2} \text{the support of }  \omega  \text{ lies in } [c_1T,c_2T]  \text{ where } 0 < c_1 < c_2
\text{ are positive absolute constants}; \\
& \text{(iii) }\label{cond3} \text{there exists } T_0  
\text{ such that } T^{\frac{1}{2}} \le T_0 \ll T \text{ and }
  \omega^{(j)}(t) \ll T_{0}^{-j} \text{ for } j \in \mathbb{N} \cup \{ 0 \} \text{ and all }  t \in \R. 
\end{align}
From \eqref{smoothafe} with $B=2$, it follows that 
\begin{align}
I_{\omega}  &= 2\int_{-\infty}^{\infty} \omega(t) \left( \sum_{m,n=1}^{\infty} \frac{\tau_4(m) \tau_4(n)}{(mn)^{\frac{1}{2}}}
\left( \frac{m}{n} \right)^{- \mi t}
V_{t}( \pi^4 mn)   \right) \, dt
 + O( T^{-1}) \label{Iomega} \\
 &= 2 \int_{-\infty}^{\infty} \omega(t) \sum_{m,n=1}^{\infty} \frac{\tau_4(m) \tau_4(n)}{(mn)^{\frac{1}{2}}}
\left( \frac{m}{n} \right)^{-\mi t}
\frac{1}{2\pi \mi} \int_{(1)} \frac{G(s)}{s} g(s,t) t^{-4s} \left( \frac{t^4}{\pi^4 mn }\right)^s \,ds \,dt +O(T^{-1}). \nonumber
\end{align}
One problem with this expression is that the Dirichlet polynomials are too long and they need to be 
made slightly shorter. 
This will be done by introducing the expression
\begin{equation*}
\begin{split}
\tilde{I}_{\omega} := \tilde{I}_{\omega}(U)= 2 \int_{-\infty}^{\infty} \omega(t)  \sum_{m,n=1}^{\infty} \frac{\tau_4(m) \tau_4(n)}{(mn)^{\frac{1}{2}}}
\left( \frac{m}{n} \right)^{- \mi t}
\frac{1}{2\pi \mi} \int_{(1)} \frac{G(s)}{s} g(s,t) t^{-4s} \left( \frac{U^4}{\pi^4 mn }\right)^s \, ds \,dt
\end{split}
\end{equation*}
where  $U =T^{1-\e}$ chosen as in \eqref{def:U}.
In the following key proposition we show $I_{\omega}$ is closely approximated by $\tilde{I}_{\omega}$ for $U =T^{1-\e}$.
This truncation argument has previously been used in the works of Soundararajan-Young \cite{SY} on the second moment of modular $L$-functions, 
Chandee-Li \cite{CL} on the eighth moment of Dirichlet $L$-functions, and Shen \cite{Sh} on the fourth moment of quadratic Dirichlet $L$-functions.

\begin{prop}\label{replacing-I}
Under the Riemann hypothesis, for any $\varepsilon>0$, there exists $T_\varepsilon$ such that
$$
I_{\omega} - \tilde{I}_{\omega} = O(\varepsilon T (\log T)^{16}),
$$ 
for $T\ge T_{\varepsilon}$, where the implied constant is independent of $\varepsilon$.
\end{prop}

\begin{proof}
We start with the observation
\begin{align*}
&I_{\omega} - \tilde{I}_{\omega} \\
 &= 2 \int_{-\infty}^{\infty} \omega(t)  \sum_{m,n=1}^{\infty} \frac{\tau_4(m) \tau_4(n)}{(mn)^{\frac{1}{2}}}
\left( \frac{m}{n} \right)^{-\mi t}
\frac{1}{2\pi \mi} \int_{(1)} G(s) g(s,t) t^{-4s}  \frac{1}{(\pi^4 mn)^s } \left( \frac{t^{4s} - U^{4s}}{s}  \right) \,ds\,dt +O(1) \\
& = \int_{-\infty}^{\infty} \omega(t) \frac{1}{\pi \mi } \int_{(1)} \zeta (\tfrac{1}{2} +s +\mi t)^4 \zeta (\tfrac{1}{2} +s -\mi t)^4 
     G(s) g(s,t) t^{-4s}  \pi^{-4s}  \left( \frac{t^{4s} - U^{4s}}{s} \right) \, ds\,dt +O(1).
\end{align*}
As $\frac{t^{4s} -U^{4s}}{s} $ is entire, if we move the last integration to the line $\Re(s) = 0$, we will encounter poles at $\frac{1}{2}  \pm \mi t$. Since $G^{(j)}(\frac{1}{2}  \pm \mi t)$ decays rapidly, the contribution of these poles is $O(1)$. Hence, moving the integration, by Lemma \ref{lem:bd-g}(i), we deduce
\begin{align*}
&I_{\omega} - \tilde{I}_{\omega} \\
 &= \int_{-\infty}^{\infty} \omega(t) \frac{1}{\pi \mi } \int_{(0)} \zeta (\tfrac{1}{2} +s +\mi t)^4 \zeta (\tfrac{1}{2} +s -\mi t)^4 
     G(s)   (2\pi)^{-4s}  \left( \frac{t^{4s} - U^{4s}}{s}  \right) \,ds\,dt \\
     &+ O \left( \int_{-\infty}^{\infty} \omega(t)  \int_{(0)} \left|\zeta (\tfrac{1}{2} +s +\mi t)^4 \zeta (\tfrac{1}{2} +s -\mi t)^4 
     G(s) \frac{|s|^2+1}{t}  \left(  \frac{t^{4s} - U^{4s}}{s}  \right) \right| \,d|s|\,dt\right) +O(1).
\end{align*}
Recalling that   $U =T^{1-\e}$, we have  $\frac{t^{4 \mi t_0  } -U^{4 \mi t_0  }}{\mi t_0}  \ll \log \frac{t^4}{U^4} \ll \varepsilon \log T$. Thus, the first big-O term in the above expression  is 
\begin{align}
&\ll \varepsilon \log T \int_{-\infty}^{\infty} \omega(t)    \int_{-\infty}^{\infty} |\zeta (\tfrac{1}{2} +\mi t_0 +\mi t)|^4 |\zeta (\tfrac{1}{2} +\mi t_0 -\mi t)|^4 
     |G(\mi t_0)| \frac{|t_0|^2+1}{t}   \,dt_0 \,dt\nonumber \\
    & \ll \frac{\varepsilon \log T}{T}   \int_{-\infty}^{\infty}
     |G(\mi t_0)| ({|t_0|^2+1})   \int_{-\infty}^{\infty} \omega(t) |\zeta (\tfrac{1}{2} +\mi t +\mi t_0)|^4 |\zeta (\tfrac{1}{2} +\mi t -\mi t_0)|^4  \,dt \,dt_0. 
     \label{equ:O-trunc}
\end{align}
Since  $|\zeta (\tfrac{1}{2} +\mi t \pm \mi t_0)| \ll |t_0| + |t|$ and \eqref{bdG}, the contribution of  $|t_0| >  \sqrt{T}$ to \eqref{equ:O-trunc}  is $\ll 1$. For $ |t_0|\le \sqrt{T} $, applying Theorem \ref{refined_Harper},  we derive
\begin{equation} \label{equ:bd-t0<sqrtT}
\int_0^T |\zeta(\oh+\mi t+\mi t_0)|^{4}|\zeta(\oh+\mi t -\mi t_0)|^{4} \,dt \ll T(\log T)^{16}.
\end{equation}
Hence, we obtain
$$
 \int_{|t_0| \le \sqrt{T}}  |G(\mi t_0)|  ({|t_0|^2+1})  \int_{-\infty}^{\infty} \omega(t)  |\zeta (\tfrac{1}{2} +\mi t +\mi t_0)|^4 |\zeta (\tfrac{1}{2} +\mi t- \mi t_0)|^4  \,dt  \,d t_0 \ll T(\log T)^{16}.
$$
Therefore,  \eqref{equ:O-trunc} is $\ll \varepsilon( \log T)^{17}$, and we arrive at
\begin{align*}
I_{\omega} - \tilde{I}_{\omega}  = \int_{-\infty}^{\infty} \omega(t) dt\frac{1}{\pi \mi } \int_{(0)} \zeta (\tfrac{1}{2} +s +\mi t)^4 \zeta (\tfrac{1}{2} +s -\mi t)^4 
     G(s)   (2\pi)^{-4s} \frac{t^{4s} - U^{4s}}{s} \,ds      + O( \varepsilon (\log T)^{17}).
\end{align*}
The double integral above is 
\begin{align}
\ll   \varepsilon \log T \int_{-\infty}^{\infty}  |G(\mi t_0)|    \int_{-\infty}^{\infty} \omega(t)  |\zeta (\tfrac{1}{2} +\mi t +\mi t_0)|^4 |\zeta (\tfrac{1}{2} +\mi t- \mi t_0)|^4 \, dt \, d t_0 .
\label{equ:trunc-1}
\end{align}
As argued above, the contribution of $|t_0| > \sqrt{T}$ to \eqref{equ:trunc-1} is $\ll 1$. So, it remains to consider the contribution of $|t_0| \le \sqrt{T}$. It follows from \eqref{equ:bd-t0<sqrtT} that
$$
 \varepsilon \log T  \int_{\log T< |t_0| \leq \sqrt{T} }  |G(\mi t_0) |\int_{-\infty}^{\infty} \omega(t)  |\zeta (\tfrac{1}{2} +\mi t +\mi t_0)|^4 |\zeta (\tfrac{1}{2} +\mi t- \mi t_0)|^4  \,dt  \,d t_0 \ll \varepsilon T.
$$
For $\frac{1}{200} <|t_0| \le \log T  $, it follows from Theorem \ref{refined_Harper} that  
$$
\int_0^T |\zeta(\oh+\mi t+\mi t_0)|^{4}|\zeta(\oh+\mi t -\mi t_0)|^{4} \,dt \ll T(\log T)^{8} (\log\log T)^8 \ll T  (\log T)^{9}
$$
and so
$$
 \e \log T \int_{\frac{1}{200} <|t_0| \le \log T}  |G(\mi t_0) |    \int_{-\infty}^{\infty} \omega(t)  |\zeta (\tfrac{1}{2} +\mi t +\mi t_0)|^4 |\zeta (\tfrac{1}{2} +\mi t- \mi t_0)|^4  \,dt  \,d t_0 \ll \e T  (\log T)^{10}.
$$
Finally, by Theorem \ref{refined_Harper}, for $|t_0| \leq \frac{1}{200}$, we have
\begin{align*}
\int_{-\infty}^{\infty} \omega(t)  |\zeta (\tfrac{1}{2} +\mi t +\mi t_0)|^4 |\zeta (\tfrac{1}{2} +\mi t- \mi t_0)|^4 \,dt  \ll  T  (\log T)^{8} \min \left\{ \frac{1}{|2t_0|^8}, (\log T)^8 \right\}.
\end{align*}
Thus, the contribution of  $|t_0| \leq \frac{1}{200}$ to \eqref{equ:trunc-1} is 
\begin{align*}
&\ll   \varepsilon \log T  \int_{ |t_0| \leq \frac{1}{\log T} }  |G(\mi t_0) |  T (\log T)^{16} \,dt_0 +  \varepsilon \log T
 \int_{ \frac{1}{\log T}\le |t_0| \leq \frac{1}{200}}  |G(\mi t_0)| T (\log T)^{8} \frac{1}{|t_0|^8} \,dt_0
 \\
 &\ll  \varepsilon \log T \cdot T (\log T)^{15} +  \varepsilon \log T \cdot  T(\log T)^{8} \cdot  (\log T)^{7} \\
&  \ll \varepsilon T( \log T)^{16}.
\end{align*}
Hence, gathering everything together, we obtain
\[
I_{\omega} - \tilde{I}_{\omega}  \ll \varepsilon T( \log T)^{16}
\]
as desired.
\end{proof}

We now have $
  I_{\omega} = \tilde{I}_{\omega} + O(\varepsilon T(\log T)^{16})$.
We shall divide $\tilde{I}_{\omega}$ into diagonals parts  $m=n$ and off-diagonal parts $m \ne n$ as follows.
Setting
\begin{align}
 \label{ID}
  I_{D} & := 2 \int_{-\infty}^{\infty} \omega(t)  \sum_{m=1}^{\infty} \frac{\tau_4(m)^2}{m}
\frac{1}{2\pi \mi} \int_{(1)} \frac{G(s)}{s} g(s,t) t^{-4s} \left( \frac{U^4}{\pi^4 m^2 }\right)^s  \,ds \,dt, \\
 \label{IO}
  I_{O} & :=  2 \int_{-\infty}^{\infty} \omega(t)  \sum_{\substack{ m,n=1 \\ m \ne n}}^{\infty} \frac{\tau_4(m) \tau_4(n)}{(mn)^{\frac{1}{2}}}
\left( \frac{m}{n} \right)^{-\mi t}
\frac{1}{2\pi \mi} \int_{(1)} \frac{G(s)}{s} g(s,t) t^{-4s} \left( \frac{U^4}{\pi^4 mn }\right)^s  \,ds \,dt,
\end{align}
we then obtain
\begin{equation*}
   I_{\omega}  = I_D + I_O + O(\varepsilon T(\log T)^{16}). 
\end{equation*}
We shall evaluate $I_D$ and $I_O$ in the following propositions and prove them in the later sections.  
\begin{prop} \label{diagonal}
Unconditionally, we have
\begin{equation*}
  I_D = \frac{4 a_4( \log U)^{16}}{638512875} \int_{-\infty}^{\infty} \omega(t) 
 \, dt + O(T (\log T)^{15}). 
\end{equation*}
\end{prop}

\begin{prop} \label{offdiagonal} 
Assuming Conjecture \ref{divconj}, we have
\begin{equation*}
  I_O =   -\frac{13381a_4}{2615348736000} \int_{-\infty}^{\infty} \omega(t)  (\log T)^{16} \, dt  + O(\varepsilon T (\log T)^{16})
  + O \left(T \left( \frac{T}{T_0} \right)^{1+C} \right),
\end{equation*}
where $C$ is defined as in the statement of Conjecture \ref{divconj}. 
\end{prop}

Combining Propositions \ref{diagonal} and \ref{offdiagonal}, we find 
\begin{equation}
  \label{IomegadecompB}
 I_{\omega} 
 = a_4 \int_{-\infty}^{\infty} \omega(t) \frac{24024}{16!}    (\log T)^{16} \, dt + O(\varepsilon T (\log T)^{16})  
 +O \left(T \left( \frac{T}{T_0} \right)^{1+C} \right).
\end{equation}
From this formula, we shall deduce Theorem \ref{8thmoment} by removing the smooth weight $\omega$. 

\begin{proof}[Proof of Theorem \ref{8thmoment}]
Let $\eta\in (0,1)$. Let $\omega^{+}(t)$ be a majorant for the indicator function $\mathds{1}_{[T,2T]}(t)$ supported on $[(1-\eta)T, (2+\eta)T]$, and  $\omega^{+}(t) \equiv 1$ when $t \in[T,2T]$.
Then we have 
\begin{equation*}
   I_4(2T)-I_4(T) \le I_{\omega^{+}}.  
\end{equation*}
Let $\varepsilon \in(0,1)$.  
It follows from \eqref{IomegadecompB} and the choice  $T_0=\eta T$ that there is $T_{\varepsilon}$ such that for $T \ge \max\{ \frac{1}{\eta^2}, T_{\varepsilon} \}$, 
\begin{equation}\label{upbd-2}
\begin{split}
 I_4(2T)-I_4(T) \le  I_{\omega^{+}} & =  \frac{g_4 a_4}{16!} \int_{-\infty}^{\infty} \omega^{+}(t)  (\log T)^{16} \, dt + O( \varepsilon T(\log T)^{16} )+ O(T \eta^{-1-C}) \\
 & =   \frac{g_4 a_4}{16!} T(\log T)^{16} + O( \varepsilon  T(\log T)^{16}) + O( \eta T (\log T)^{16}) + O(T \eta^{-1-C}).
\end{split}
\end{equation}
Let $\eta_0 \in(0,1)$. Set $j = \ceil{ \log_2 (1/\eta_0)}$  so that $\eta_0 T/2 \le  T/2^j  \le \eta_0 T$.\footnote{Here $\log_2$ denotes the base 2 logarithm.} Substituting the values $T/2, \ldots, T/2^j$ into \eqref{upbd-2}, we obtain
\[
     I_4(T) - I_4(T/2^j) \le  \frac{g_4 a_4}{16!} T(\log T)^{16}
     + O( \varepsilon T (\log T)^{16}) + O( \eta T (\log T)^{16}) + O(T \eta^{-1-C})
\]
whenever $\eta_0 T/2 \ge \max\{ \frac{1}{\eta^2}, T_{\varepsilon} \}$. (Here, we used the bound $T/2  + \cdots + T/2^j  \le T$.)
As $T/2^j \le \eta_0 T$,  by Harper's result \cite{Ha}, on the Riemann hypothesis, there is $\widetilde{T}$ such that for $\eta_0 T\ge \widetilde{T}$, we have 
\[
  | I_4(T/2^j)| \le I_4(\eta_0 T) \ll (\eta_0 T) (\log(\eta_0 T))^{16} \ll \eta_0  T (\log T)^{16}. 
\]
Combining everything, we find that for $T\ge \frac{1}{\eta_0} \max\{\frac{2}{\eta^2}, 2T_{\varepsilon}, \widetilde{T}\}$,
\[
  I_4(T) \le  \frac{g_4 a_4}{16!} T(\log T)^{16} + C_1 \varepsilon  T (\log T)^{16} + C_2    \eta T (\log T)^{16} + C_3 T    \eta^{-1-C} + C_4 \eta_0  T (\log T)^{16} 
\]
for certain absolute explicit constants $C_1, C_2,C_3,C_4$.  Now,  dividing both sides of the above inequality by $T (\log T)^{16}  $ and taking  $\limsup$,  we obtain
$$
\limsup_{T\rightarrow \infty } \frac{I_4(T)}{ T (\log T)^{16} }\le \frac{g_4 a_4}{16!} +  C_1 \varepsilon   + C_2   \eta   + C_4 \eta_0. 
$$ 
As $\varepsilon, \eta, \eta_0 \in(0,1)$ are arbitrary, we arrive at
$$
\limsup_{T\rightarrow \infty } \frac{I_4(T)}{ T (\log T)^{16} }\le \frac{g_4 a_4}{16!}.
$$

The lower bound is proven similarly and is based on the identity $ I_4(2T)-I_4(T) \ge I_{\omega^{-}}$,
where $\omega^{-}(t)$ is a minorant for the function $\mathds{1}_{[T,2T]}(t)$ supported on $[(1+\eta)T, (2-\eta)T]$, giving 
\begin{equation}\label{upbd-3}
 I_4(2T)-I_4(T)  \ge   \frac{g_4 a_4}{16!} T(\log T)^{16} + O( \varepsilon  T(\log T)^{16}) + O( \eta T (\log T)^{16}) + O(T \eta^{-1-C}).
\end{equation}
For $\eta_0 \in(0,1)$, we again set $j = \ceil{ \log_2 (1/\eta_0)}$ and substitute the values $T/2, \ldots, T/2^j$ into \eqref{upbd-3}, which yields
\begin{equation}\label{upbd-4}
     I_4(T) - I_4(T/2^j) \ge  \frac{g_4 a_4}{16!}  \sum_{i=1}^{j} \frac{T}{2^i} \log^{16} \left( \frac{T}{2^i} \right)
     + O( \varepsilon T (\log T)^{16}) + O( \eta T (\log T)^{16}) + O(T \eta^{-1-C}).
\end{equation}
Note that 
\begin{align}  \label{logsum}
 \begin{split}
  \sum_{i=1}^{j} \frac{T}{2^i} \log^{16} \left( \frac{T}{2^i} \right)
&  = \sum_{i=1}^{j}   \frac{T}{2^i}   \sum_{u+v=16} \binom{16}{u}   (\log T)^{u} (-i \log 2)^{v} \\
&  =  T  \sum_{u+v=16} \binom{16}{u}   (\log T)^{u} (-\log 2)^{v} \sum_{i=1}^{j} \frac{i^v}{2^i}. 
  \end{split}
\end{align}
Now,  observe that 
\[
    \sum_{i=1}^{j} \frac{i^v}{2^i} = C_{v}' + O\left( \sum_{i=j+1}^{\infty}  \frac{i^v}{2^i} \right)
    = C_{v}' + O ( j^v 2^{-j}) 
    =C_{v}' + O \left( \eta_0 \log_2^{v} \left(\frac{1}{\eta_0} \right) \right)
\]
for some constants $C_{v}'$ such that $C_{0}'=1$. Plugging this last estimate in \eqref{logsum}, we find that 
\[
     \sum_{i=1}^{j} \frac{T}{2^i} \log^{16} \left( \frac{T}{2^i} \right) = T (\log T)^{16} + O( \eta_0 T (\log T)^{16}) + O \left( \left( 1+ \eta_0 \log^{15} \left(\frac{1}{\eta_0} \right)\right) T(\log T)^{15} \right).
\]
From which, together with \eqref{upbd-4},  we can establish
$$
\liminf_{T\rightarrow \infty } \frac{I_4(T)}{ T (\log T)^{16} }\ge \frac{g_4 a_4}{16!},
$$
which completes the proof of the theorem. 
\end{proof}

\section{The diagonal terms: Proof of Proposition \ref{diagonal}}

From the definition \eqref{ID} of $I_{D}$, we see 
\begin{align}
 I_{D} = 2 \int_{-\infty}^{\infty} \omega(t)  \frac{1}{2\pi \mi} \int_{(1)} \sum_{m=1}^{\infty}  \frac{\tau_4(m)^2  }{m^{1+2s}} \pi^{-4s} \frac{G(s)}{s} g(s,t)   t^{-4s}U^{4s} \, ds\, dt. 
 \label{1stID}
\end{align}
\begin{lem}
\label{lem:Z1}
For $\Re(s)> \varepsilon_0>0$, we have
\begin{align}
\sum_{m=1}^{\infty}  \frac{\tau_4(m)^2  }{m^{1+2s}} 
= \zeta^{16} (1+2s) Z_1(s),
\label{iden-dia}
\end{align}
where 
\begin{equation}
 \label{Z1}
 Z_1(s) =  \prod_{p} \Big( 1-\frac{1}{p^{1+2s}} \Big)^{16}
 \sum_{m=1}^{\infty}  \frac{\tau_4(m)^2  }{m^{1+2s}} .
\end{equation}
In addition, $Z_1(s)$ is holomorphic and absolutely convergent on $\Re(s) > -\frac{1}{4}+ \varepsilon_0$. Also, $Z_1(0) =a_4$, where $a_4$ is defined in \eqref{gkak}.
\end{lem}

\begin{proof}
For $\Re(s)>\varepsilon_0$,  the identity \eqref{iden-dia} is trivial by comparing both sides.  Now we shall show that  $Z_1(s)$ is holomorphic and absolutely convergent on $\Re(s) > -\frac{1}{4} + \e_0$.  Note that for $\Re(s)>\varepsilon_0$, 
\begin{align}
\sum_{m=1}^{\infty}  \frac{\tau_4(m)^2  }{m^{1+2s}} 
&= 
\prod_{p} \sum_{r=0}^{\infty} \frac{\tau_4(p^r)^2}{p^{r(1+2s)}}
=\prod_{p} \sum_{r=0}^{\infty} \frac{{r+3 \choose r}^2}{p^{r(1+2s)}}
=\prod_{p}\left(  1 +  \frac{16}{p^{1+2s}} + B_p(s)\right),
\label{Lem4.1-1}
\end{align}
where $B_p(s):= \sum_{r=2}^{\infty} \frac{{r+3 \choose r}^2}{p^{r(1+2s)}}$.
By \eqref{Z1},
\begin{align}
Z_1(s)
  &=  \prod_{p} \left( \sum_{r=0}^{16} {16 \choose r} (-1)^r \frac{1}{p^{r+2rs}} \right)
 \prod_{p}\left(  1 +  \frac{16}{p^{1+2s}} + B_p(s)\right)\nonumber \\
 &= \prod_{p} \left(1 - \frac{16}{p^{1+2s}} +\sum_{r=2}^{16} {16 \choose r} (-1)^r \frac{1}{p^{r+2rs}} \right)
 \prod_{p}\left(  1 +  \frac{16}{p^{1+2s}} + B_p(s) \right)\nonumber \\
 &=  \prod_{p} \left(1 + A_p(s) + \frac{16}{p^{1+2s}} A_p(s) + A_p(s)B_p(s)+ B_p(s)-\frac{16^2}{p^{2+4s}} - \frac{16}{p^{1+2s}}B_p(s) \right),
 \label{eulerP}
\end{align}
where $A_p(s) := \sum_{r=2}^{16} {16 \choose r} (-1)^r \frac{1}{p^{r+2rs}}$.
Clearly, for $\Re(s)> -\frac{1}{4} + \varepsilon_0$,
\begin{align*}
|A_p(s)| \ll \sum_{r=2}^{16} \frac{1}{p^{(\frac{1}{2} + 2\varepsilon_0)r}} \ll \frac{1}{p^{1 + 4\varepsilon_0}}.
\end{align*}
In addition, for $\Re(s)> -\frac{1}{4} + \varepsilon_0$,
we have
\begin{align*}
|B_p(s)| \ll \sum_{r=2}^{\infty} \frac{{r+3 \choose r}^2}{p^{(\frac{1}{2} + 2\varepsilon_0)r}}
\ll \sum_{r=2}^{\infty} \frac{r^3}{p^{(\frac{1}{2} + 2\varepsilon_0)r}}
\ll  \sum_{r=2}^{\infty} \frac{1}{p^{(\frac{1}{2} + \varepsilon_0)r}}
\ll  \frac{1}{p^{1 + 2\varepsilon_0}}.
\end{align*}
The last two bounds imply \eqref{eulerP} is absolutely and  uniformly convergent when $\Re(s)> -\frac{1}{4} + \varepsilon_0$, and thus, is holomorphic in this region. 
Finally, the identity $Z_1(0) =a_4$ follows from \eqref{Z1} and \eqref{Lem4.1-1} by setting $s=0$.
\end{proof}

By \eqref{1stID} and Lemma \ref{lem:Z1}, together with  Lemma \ref{lem:bd-g}(i), we get 
\begin{align*}
 I_{D} 
&= 2 \int_{-\infty}^{\infty} \omega(t)  \frac{1}{2\pi \mi} \int_{(1)} \zeta^{16} (1+2s) Z_1(s)\pi^{-4s} \frac{G(s)}{s} g(s,t)   t^{-4s}U^{4s} \, ds\, dt \\
&= 2 \int_{-\infty}^{\infty} \omega(t)  \frac{1}{2\pi \mi} \int_{(1)}   \zeta^{16} (1+2s) Z_1(s)  (2\pi)^{-4s} \frac{G(s)}{s} 
U^{4s}
\left( 1 + O \left( \frac{|s|^2+1}{t} \right)\right)
 \, ds\, dt.
\end{align*}
One can prove that  the big-O term contributes an error at most $O(T^{4\varepsilon_0})$ for any small $\varepsilon_0>0$.  Indeed, moving to the line of the integration to $\Re(s) = \varepsilon_0$, without encountering any poles, gives
\begin{align*}
2 \int_{-\infty}^{\infty} \omega(t)  \frac{1}{2\pi \mi} \int_{(\varepsilon_0)}   \zeta^{16} (1+2s) Z_1(s)  (2\pi)^{-4s} \frac{G(s)}{s} 
U^{4s}
O \left( \frac{|s|^2+1}{t} \right)
 \, ds\, dt
 \ll \int_{-\infty}^{\infty} \omega(t) \frac{U^{4\varepsilon_0}}{t} \,dt ,
\end{align*}
which is $\ll T^{4\varepsilon_0}$.
Thus, we have
\begin{align}
I_D  
&= 2 \int_{-\infty}^{\infty} \omega(t) \frac{1}{2\pi \mi} \int_{(1)}   \zeta^{16} (1+2s) Z_1(s)  (2\pi)^{-4s} \frac{G(s)}{s} U^{4s}
 \, ds\, dt + O(T^{4\varepsilon_0}).
 \label{equ:diagonal02}
\end{align}
Moving the line of the integration to $\Re(s) = -\frac{1}{4} + \varepsilon_0$, we encounter a pole of order $17$ at $s=0$ and a new integral 
\begin{align}
2 \int_{-\infty}^{\infty} \omega(t) \frac{1}{2\pi \mi} \int_{(-\frac{1}{4}+ \varepsilon_0)}   \zeta^{16} (1+2s) Z_1(s)  (2\pi)^{-4s} \frac{G(s)}{s} U^{4s}
 \, ds\, dt \ll T U^{-1+ 4\varepsilon_0} \ll T^{4\varepsilon_0 +\e}. 
 \label{equ:diag-new}
\end{align}
With the help of Maple,  we see the residue of this pole is 
\begin{align}
\frac{2Z_1(0) (\log U)^{16}}{638512875} + O ((\log U)^{15}).
 \label{equ:diag-pole}
\end{align}
Substituting \eqref{equ:diag-new}  and \eqref{equ:diag-pole} in \eqref{equ:diagonal02}, we obtain
\begin{align*}
I_D 
&= 2 \int_{-\infty}^{\infty} \omega(t) \,dt \frac{2Z_1(0) (\log U)^{16}}{638512875}
 + O(T (\log T)^{15}).
\end{align*}
This, together with the fact that $Z_1(0)=a_4$ (see Lemma \ref{lem:Z1}), completes the proof of Proposition \ref{diagonal}.

\section{The off-diagonal terms: sketch of proof}

The most difficult part in evaluating $I_4(T)$ is the off-diagonal term  $I_{O}$. 
As the asymptotic evaluation of $I_{O}$ is very involved,  we provide a high level summary of the key steps.
This argument  can be viewed as a descendant of the arguments in the articles \cite{In}, \cite{HB}, \cite{GG}, \cite{HY}, and \cite{Ng}. 
The key idea is that the off-diagonal terms can be evaluated by rewriting them in terms of additive divisor sums and by inserting in the main
term from the additive divisor conjecture.  The one difference between this article and the earlier articles \cite{In}, \cite{HB}, and \cite{GG} is that
the off-diagonals are related to the smoothed sums \eqref{DfIJ} rather than the classical unsmoothed sums \eqref{Dkxr}. This approach was first used \cite{HY}
and then in \cite{Ng}. 
Let $\e_1 >0$. 
We move the $s$-contour in \eqref{IO} to $\Re(s) =\e_1$ to obtain
\[
  I_{O} = 2T  \sum_{m \ne n} 
  \frac{\tau_{4}(m) \tau_{4}(n)}{\sqrt{mn}}
  f^{*}(m,n) 
\]
where 
\begin{equation*}
    f^{*}(x,y) := \frac{1}{2 \pi \mi} 
    \int_{(\varepsilon_1)} \frac{G(s)}{s} 
    \left(  \frac{1}{\pi^4 xy} \right)^s 
    \frac{1}{T} 
    \int_{-\infty}^{\infty} 
    \left( \frac{x}{y} \right)^{-\mi t}
    g(s,t)  \left( \frac{U}{t} \right)^{4s}  \omega(t) \,dt \, ds.
\end{equation*}
First, a smooth partition of unity is introduced so that 
\[
  I_{O} =    \sum_{M,N} I_{M,N}
\]
where the sum is over $M,N \in \{ 2^{\frac{k}{2}} \ | \ k \ge -1 \}$, and
\begin{equation*}
   I_{M,N} = \frac{2T}{\sqrt{MN}} 
\sum_{m \ne n} \tau_{4}(m) \tau_{4}(n) W \left(\frac{m}{M} \right) W \left(\frac{n}{N} \right) f^{*}(m,n).
\end{equation*}
The point of introducing the smooth partition of unity is so that we can apply Conjecture \ref{divconj} (the quaternary additive divisor conjecture). 
The parameters $M,N$ can be truncated to the region $MN \ll U^{4 +\e_2}$.  Note that since $U=T^{1-\e}$, the function $f^{*}(m,n)$ is very small if $mn \gg U^{4 +\e_2}$ and thus the contribution of $I_{M,N}$ when $MN \gg U^{4 +\e_2}$
is negligible. Thus, 
\[
    I_{O}  \sim    \sum_{\substack{M,N \\  MN \ll U^{4 +\e_2} }} I_{M,N}.
\]
Note that 
\begin{equation}
  \label{Usize}
   MN \ll U^{4 +\e_2}   \le T^{4-3 \e} \text{ for } \e_2 \le \e.
\end{equation}

Next, we claim that $m$ and $n$ must be close together. Namely, we show that if $|\log(\tfrac{m}{n})| \gg T_{0}^{-1+\e_1}$, then $f^{*}(m,n)$ is very small. 
Thus, we may insert the condition in  
\begin{equation}
 \label{logcond}
|\log(\tfrac{m}{n})| \ll T_{0}^{-1+\e_1}
\end{equation}
 with an negligible error.  Note that this last condition forces $M$ and $N$ to be within a constant multiple of each other. 
Namely, we have the condition
\begin{equation*}
    N/3 \le M \le 3N. 
\end{equation*}
Combining this with \eqref{Usize}, it follows that 
\begin{equation*}
  M,N \le T^{2 -\frac{3 \e}{2}}. 
\end{equation*}

The idea now is to manipulate $I_{M,N}$ so that instead of summing over $m,n$, the sum is over $r, n$ where we the variable change $m=n+r$ is made. Thus
\[
     I_{M,N} = \frac{2T}{\sqrt{MN}}   \sum_{r \ne 0} \Bigg( \sum_{ m-n=r  } \tau_{4}(m) \tau_{4}(n) W \left(\frac{m}{M} \right) W \left(\frac{n}{N} \right) f^{*}(m,n) \Bigg) +O(T^{-10}).
\]
Further,  the conditions \eqref{Usize} and \eqref{logcond} imply
we can insert conditions on $r$ and $n$ so that 
\[
     I_{M,N} = \frac{2T}{\sqrt{MN}}   \sum_{\substack{ r \ne 0 \\  |r| \ll \frac{M}{T_0} T_{0}^{\varepsilon_1} } } \Bigg( \sum_{\substack{ m-n=r \\ |\log(\frac{m}{n})| \ll T_{0}^{-1+\e_1}} } \tau_{4}(m) \tau_{4}(n) W \left(\frac{m}{M} \right) W \left(\frac{n}{N} \right) f^{*}(m,n) \Bigg) +O(T^{-10}).
\]
Applying the quaternary additive divisor conjecture we obtain
\[
   I_{M,N} = \frac{2T}{\sqrt{MN}} 
\sum_{0 < |r| \ll \frac{M}{T_0} T_{0}^{\varepsilon_1}}  D_{f_r}(r) +  O \left( \left( \frac{T}{T_0} \right)^{1+C} T^{1-\frac{\e}{2}} \right)
 + O(T^{-10})
 \]
 for a function $f_r$ defined in \eqref{fMN} below and $D_{f_r}(r) := D_{f_r;4,4}(r)$ where we recall the definition \eqref{DfIJ}.  Observe that since $M \le T^{2-\frac{3\e}{2}}$, square root cancellation gives $\sqrt{M} \ll T^{1-\frac{3 \e}{4}}$
 and thus we obtain an error term  $O ( ( \frac{T}{T_0} )^{1+C} T^{1-\frac{\e}{2}} )$ from the quaternary additive divisor conjecture (for full details see Lemma \ref{adddiverror} below). 
 It follows that 
 \begin{equation*}
  I_{O} = \sum_{M,N}  \frac{2T}{\sqrt{MN}} \sum_{0<|r| \le R_0} \tilde{D}_{f_r} (r)+  O \left( \left( \frac{T}{T_0} \right)^{1+C} T^{1-\frac{\e}{2}} \right),
\end{equation*} 
where $ \tilde{D}_{f_r} (r)$ is the main term in the additive divisor conjecture, given by the main term in \eqref{Dfr} associated to $f_r$.  
The next step is to sum back over the partition of unity.  Doing this, we arrive at
\begin{align*}
 \begin{split}
I_O& =2 T  \sum_{1\leq |r| \leq R_0} 
\frac{1}{(2\pi \mi)^2} \int_{\mathcal{B}_1} \int_{\mathcal{B}_2} \zeta^4 (z_1) \zeta^4 (z_2)
\sum_{q=1}^{\infty} \frac{c_q(r) G_4 (z_1,q) G_4(z_2,q)}{q^{z_1 + z_2}}  \\
&\times \int_{\max(0,r)}^{\infty} f^* (x,x-r) x^{z_1 -3/2} (x-r)^{z_2 -3/2} \, dx \, d {z_2} \, d {z_1} +  O \left( \left( \frac{T}{T_0} \right)^{1+C} T^{1-\frac{\e}{2}} \right),
 \end{split}
\end{align*}
where 
\begin{equation*}
   f^*(x,x-r) = \frac{1}{2 \pi \mi} \int_{(\varepsilon_1)} 
\frac{G(s)}{s} \left( \frac{1}{\pi^{4} x(x-r)} \right)^s 
\frac{1}{T} \int_{-\infty}^{\infty} \left( 
1+\frac{r}{x-r} \right)^{-\mi t} g(s,t) \left( \frac{U}{t} \right)^{4s} \omega(t) \,dt \,ds.
\end{equation*}
Observe that the expression on the  right hand side depends on quantities that arise from the additive divisor conjecture, such as  the Ramanujan sum $c_q(r)$
and the divisor-type functions $G_4(z,r)$.  These quantities will then become part of the Dirichlet series
\[
    \HH_s(z_1,z_2)=   \sum_{r=1}^{\infty} 
\sum_{q=1}^{\infty} \frac{c_q(r) G_4 (z_1,q) G_4(z_2,q)}{q^{z_1 + z_2} r^{2s+ 2-z_1-z_2}}.
\]
After summing separately over $0 < r \le R_0$ and $-R_0 \ge r < 0$ and then recombining terms by using the symmetry $c_q(r)=c_q(-r)$, we obtain
 \begin{align*} 
 \begin{split}
  & I_O \\
   & =  \frac{2}{(2\pi \mi)^3} \int_{\mathcal{B}_1} \int_{\mathcal{B}_2} \int_{(\varepsilon_3)}  \zeta^4 (z_1) \zeta^4 (z_2)
 \sum_{1\leq r \leq R_0}
\sum_{q=1}^{\infty} \frac{c_q(r) G_4 (z_1,q) G_4(z_2,q)}{q^{z_1 + z_2} r^{2s+ 2-z_1-z_2}} 
\frac{G(s)}{s  \pi^{4s} }   
 \int_{-\infty}^{\infty} g(s,t)  \left( \frac{U}{t} \right)^{4s} \omega(t)  \\ 
 & 
\times  \Gamma( -z_1-z_2+2s+2)\left( 
\frac{\Gamma(z_2-s-\frac{1}{2}+\mi t )}{ \Gamma(s-z_1+\frac{3}{2}+\mi t )  } 
+ 
\frac{\Gamma(z_1-s-\frac{1}{2}-\mi t) }{\Gamma(s-z_2+\frac{3}{2}-\mi t) }
   \right)  \, dt \, ds\,  d {z_2} \, d {z_1}\\
   &+ O \left( \left( \frac{T}{T_0} \right)^{1+C} T^{1-\frac{\e}{2}} \right). 
 \end{split}  
\end{align*}
Note that the $x$-integrals have disappeared   as they   can be computed exactly in terms of the $\text{B}$ function, which  then leads to the 5 Gamma factors in the last equation.
The next step is to move the $s$-integral to the right to $\Re(s)=1$ and to let $R_0 \to \infty$.   The point of this is that the  double sums $\sum_{r,q}$ are absolutely convergent in 
this region and by extending $R_0 \to \infty$ we obtain $\HH_s(z_1,z_2)$.  We then make use of the meromorphic continuation (see Proposition \ref{Hlemma} below):
\[
   \HH_s(z_1,z_2) = \zeta(2s+2-z_1-z_2) \frac{\zeta(1+2s)^{16} \zeta( 1+2s-z_1 -z_2 + 2) }{ \zeta(1+2s - z_1 + 1)^4 \zeta(1+2s- z_2 +1)^4 }
       \tilde{ \mathscr{I}}(z_1,z_2,s)
\] 
where $  \tilde{ \mathscr{I}}(z_1,z_2,s)$ is holomorphic for $\Re(s) > -\frac{1}{4}+ 2\delta$ for some $\delta >0$.
Furthermore, we know from Lemma \ref{lem:bd-g} that $g(s,t)$ can be approximated  by $ \left( \frac{t}{2} \right)^{4s} $, and thus we see that 
\begin{align*} 
 \begin{split}
   I_O & =    \int_{-\infty}^{\infty}  \omega(t) \Bigg( 
\frac{2}{(2\pi \mi)^3} \int_{\mathcal{B}_1} \int_{\mathcal{B}_2} \int_{(\varepsilon_3)}  \zeta^4 (z_1) \zeta^4 (z_2) \\
&\times \zeta(2s+2-z_1-z_2) \frac{\zeta(1+2s)^{16} \zeta( 1+2s-z_1 -z_2 + 2) }{ \zeta(1+2s - z_1 + 1)^4 \zeta(1+2s- z_2 +1)^4 }
       \tilde{ \mathscr{I}}(z_1,z_2,s) \\
       &\times
\frac{G(s)}{s  \pi^{4s} }   
  \left( \frac{U}{2} \right)^{4s}  
\Gamma( -z_1-z_2+2s+2)\left( 
\frac{\Gamma(z_2-s-\frac{1}{2}+\mi t )}{ \Gamma(s-z_1+\frac{3}{2}+\mi t )  } 
+ 
\frac{\Gamma(z_1-s-\frac{1}{2}-\mi t) }{\Gamma(s-z_2+\frac{3}{2}-\mi t) }
   \right)   \, ds\,  d {z_2} \, d {z_1} \Bigg) \, dt \\
   &+ O \left( \left( \frac{T}{T_0} \right)^{1+C} T^{1-\frac{\e}{2}} \right). 
 \end{split}  
\end{align*}
A calculation with Stirling's formula (see \eqref{minus-t} and \eqref{Stirling2} below) shows that for $|\Im(s)| \le t+1$ 
\[
  \frac{\Gamma(z_2-s-\frac{1}{2}+\mi t )}{ \Gamma(s-z_1+\frac{3}{2}+\mi t )  } 
+ 
\frac{\Gamma(z_1-s-\frac{1}{2}-\mi t) }{\Gamma(s-z_2+\frac{3}{2}-\mi t) }
 \sim 2 t^{-(2s+2-z_1-z_2)} \cos \left( \frac{\pi}{2} (2s+2-z_1-z_2) \right) ,
\]
and it follows that 
\begin{align*} 
 \begin{split}
   I_O & \sim     \int_{-\infty}^{\infty}  \omega(t) \Bigg( 
\frac{4}{(2\pi \mi)^3} \int_{\mathcal{B}_1} \int_{\mathcal{B}_2} \int_{(\varepsilon_3)}  \zeta^4 (z_1) \zeta^4 (z_2) \\
& \times \zeta(2s+2-z_1-z_2) \frac{\zeta(1+2s)^{16} \zeta( 1+2s-z_1 -z_2 + 2) }{ \zeta(1+2s - z_1 + 1)^4 \zeta(1+2s- z_2 +1)^4 }
       \tilde{ \mathscr{I}}(z_1,z_2,s) \\
       & \times
\frac{G(s)}{s  \pi^{4s} }   
  \left( \frac{U}{2} \right)^{4s}  
\Gamma( -z_1-z_2+2s+2)
t^{-(2s+2-z_1-z_2)} \cos \Big( \frac{\pi}{2} (2s+2-z_1-z_2) \Big) 
   \, ds\,  d {z_2} \, d {z_1} \Bigg) \, dt 
 \end{split}  
\end{align*}
up to an error term $O ( ( \frac{T}{T_0} )^{1+C} T^{1-\frac{\e}{2}} )$.   These integrals are standard and can be evaluated by shifting the contour of $\Re(s)$
past $s=0$ and then applying the residue theorem.  Note that we $\zeta(1+u) = \frac{1}{u} + O(1)$  as $u \to 1$ and $\tilde{ \mathscr{I}}(0,0,0)=a_4$
so that approximately we have 
\begin{align*} 
 \begin{split}
   I_O & \sim    a_4  \int_{-\infty}^{\infty}  \omega(t) \Bigg( 
\frac{4}{(2\pi \mi)^3} \int_{\mathcal{B}_1} \int_{\mathcal{B}_2} \int_{(\varepsilon_3)} 
 \frac{ (2s-z_1+1)^4 (2s-z_2+1)^4 }{ (z_1-1)^4(z_2-1)^4 (2s+1-z_1-z_2)(2s)^{16} (2s-z_1-z_2+2) }  \\
       & \times 
\frac{G(s)}{s  \pi^{4s} }   
  \left( \frac{U}{2} \right)^{4s}  
\Gamma( -z_1-z_2+2s+2)
t^{-(2s+2-z_1-z_2)} \cos \Big( \frac{\pi}{2} (2s+2-z_1-z_2) \Big) 
   \, ds\,  d {z_2} \, d {z_1} \Bigg) \, dt. 
 \end{split}  
\end{align*}
up to an error term $O ( T (\log T)^{15}+( \frac{T}{T_0} )^{1+C} T^{1-\frac{\e}{2}} )$. 
For full details of the calculation, see 
Section \ref{feval-off} below.  It should be noted that the last multiple integral is a multivariable version of the types of integrals that appear in standard applications of Perron's formula.

\section{The off-diagonal terms: the full details}

\subsection{Using a smooth partition of unity and  restricting $M$ and $N$}\label{smooth-partition}
Let $\varepsilon_1  >0$ and 
\begin{equation}
  \label{fstardefn}
    f^{*}(x,y) := \frac{1}{2 \pi \mi} 
    \int_{(\varepsilon_1)} \frac{G(s)}{s} 
    \left(  \frac{1}{\pi^4 xy} \right)^s 
    \frac{1}{T} 
    \int_{-\infty}^{\infty} 
    \left( \frac{x}{y} \right)^{-\mi t}
    g(s,t)  \left( \frac{U}{t} \right)^{4s}  \omega(t) \,dt \, ds.
\end{equation}
We move the $s$-contour in \eqref{IO} to $\Re(s) =\e_1$ to obtain
\[
  I_{O} = 2T  \sum_{m \ne n} 
  \frac{\tau_{4}(m) \tau_{4}(n)}{\sqrt{mn}}
  f^{*}(m,n) .
\]
The next step is to introduce a smooth partition of unity so that the variables $m$ and $n$ lie in dyadic boxes of the shape 
$[M,2M]$ and $[N,2N]$. This will allow us to apply the smoothed additive divisor conjecture for $\tau_4$.  

Recall that there exists a smooth function $W_0$ supported in $[1,2]$ such that  $\sum_{k \in \mathbb{Z}} W_0 ( x/2^{\frac{k}{2}} ) =1$ for $x > 0$. (For an explicit construction of a smooth partition of unity, see \cite[p. 360]{H}.) Note that for $x \ge 1$, we have 
\begin{equation*}
   \sum_{ \substack{ M= 2^{\frac{k}{2}} \\ k \ge -1} } W_0 \left( \frac{x}{M} \right) =1. 
\end{equation*}
Hence, setting  $W(x) =x^{-\frac{1}{2}} W_0(x)$, we can write
\begin{equation*}
      I_{O} =  \sum_{M,N} I_{M,N},
\end{equation*}
where the sum is over $M,N \in \{ 2^{\frac{k}{2}} \ | \ k \ge -1 \}$, and
\begin{equation*}
  \label{IMN}
   I_{M,N} = \frac{2T}{\sqrt{MN}} 
\sum_{m \ne n} \tau_{4}(m) \tau_{4}(n) W \left(\frac{m}{M} \right) W \left(\frac{n}{N} \right) f^{*}(m,n).
\end{equation*}

We now restrict the size of $M$ and $N$ so that $MN \ll U^{4+\e_2}$. To do so, we shall show that the contribution 
of $M,N$ with $MN \gg U^{4+\e_2}$ is negligibly small. We first write
\begin{equation}
  \label{fstarV}
   f^{*}(x,y) = \frac{1}{T} \int_{-\infty}^{\infty}  \left( \frac{x}{y} \right)^{-\mi t} \omega(t) \widetilde{V}_{t}(\pi^4 xy) \, dt,
\end{equation}
where  $\widetilde{V}_{t}(u)$ is defined in \eqref{tildeV}.
By Lemma \ref{lem:bd-g}(iii),  when $mn \gg U^{4 +\varepsilon_2}$ and $c_1 T \le t \le c_2 T$, we have
\[
\widetilde{V}_{t}(\pi^4 mn) \ll \left( \frac{U^4}{mn}\right)^A
\]
for any $A>0$. Then 
\[
 f^{*}(m,n)  \ll \frac{1}{T} \int_{c_1T}^{c_2T} \omega(t) \left( \frac{U^4}{mn}\right)^A\, dt 
 \ll \left( \frac{U^4}{mn}\right)^A.
\]
From this, it follows that  
\begin{equation*}
  \begin{split}
  \sum_{\substack{M,N \\ MN \gg U^{4+\varepsilon_2}}} I_{M,N}  
  &\ll T U^{4A}\sum_{\substack{M,N \\ MN \gg U^{4+\varepsilon_2}}} \frac{1}{\sqrt{MN}} 
\sum_{m \ne n} \frac{\tau_{4}(m) \tau_{4}(n)}{(mn)^A} W \left(\frac{m}{M} \right) W \left(\frac{n}{N} \right) \\
&\ll T U^{4A}\sum_{\substack{M,N \\ MN \gg U^{4+\varepsilon_2}}}
\sum_{m \ne n} \frac{\tau_{4}(m) \tau_{4}(n)}{(mn)^{A+\frac{1}{2}}} W_0 \left(\frac{m}{M} \right) W_0 \left(\frac{n}{N} \right) \\
&\ll T U^{4A}\sum_{\substack{M,N \\ MN \gg U^{4+\varepsilon_2}}}
\sum_{\substack{M \leq m \leq 2M \\ N \leq n \leq 2N}} \frac{1}{(mn)^A}\\
&\ll T U^{4A}\sum_{\substack{M,N \\ MN \gg U^{4+\varepsilon_2}}}\frac{1}{(MN)^{A-1}}\\
  &\ll T U^{4A} U^{-(A-2)(4+ \varepsilon_2)}.
  \end{split}
\end{equation*}
By taking $A \geq  \frac{(1- \varepsilon)(8+2\varepsilon_2)+ 1 + B}{\varepsilon_2 (1-\varepsilon)}$, we have 
\begin{align*}
\sum_{\substack{M,N \\ MN \gg U^{4+\varepsilon_2}}} I_{M,N} \ll T^{-B}
\end{align*}
for any large constant $B>0$,
which allows us to assume $MN \ll U^{4 +\e_2}$ for $ I_{M,N}$ in the remaining discussion. 

Our next step is to show that $m$ and $n$ must be close together. This is since $f^{*}(x,y)$ is small unless $x$ and $y$ are close together, which  we shall prove  rigorously as follows.  
We first control the inner integral in \eqref{fstardefn} for $s$ such that $\Re(s) =\varepsilon_1$ and $|s| \le \sqrt{T}$. Using the integration by parts $j$ times, we derive  
\[
\int_{-\infty}^{\infty} 
\left( \frac{x}{y} \right)^{-\mi t} g(s,t)  \left( \frac{U}{t} \right)^{4s}  \omega(t) \, dt 
\ll \frac{U^{4\e_1}}{|\log(\tfrac{x}{y})|^j}
\int_{-\infty}^{\infty} 
\Big|
\frac{\partial^j}{\partial t^j} (g(s,t) t^{-4s} \omega(t)) 
\Big| \, dt.
\]
To bound the above $j$-th partial derivative, we apply the generalized product rule to deduce
\begin{align*}
 \frac{\partial^j}{\partial t^j} ( g(s,t) t^{-4s} \omega(t))   = \sum_{a+b+c=j}  { j \choose a,b,c}
  \frac{\partial^a}{\partial t^a} g(s,t) \frac{\partial^b}{\partial t^b} t^{-4s}  \omega^{(c)}(t), 
\end{align*} 
which by \eqref{cond3} and Lemma \ref{lem:bd-g}(iii), is 
$$
  \ll \sum_{a+b+c=j} { j \choose a,b,c} |s|^a T^{4 \varepsilon_1-a} \frac{(|s|+j)^b}{T^{4\e_1 +b}} T_{0}^{-c}
  \ll \left( \frac{|s|}{T} + \frac{|s|+j}{T} + \frac{1}{T_0} \right)^j \ll  \frac{\max(|s|,1)^j }{T_0^j}.
$$
Therefore, we  arrive at
\[
\int_{-\infty}^{\infty} 
\left( \frac{x}{y} \right)^{-\mi t} g(s,t)  \left( \frac{U}{t} \right)^{4s}  \omega(t) \, dt 
\ll   T^{4 \varepsilon_1}  \frac{\max(|s|,1)^j }{|\log(\tfrac{x}{y})|^jT_0^j}
\]
whenever $|s| \le \sqrt{T}$. 

Secondly, we study the inner integral in \eqref{fstardefn} for $s$ such that $\Re(s) =\varepsilon_1$ and $|s| > \sqrt{T}$. By Lemma \ref{lem:bd-g}(i), as $\Re(s) =\varepsilon_1 >0$, we know  
\begin{align*}
 \int_{-\infty}^{\infty}  \left(\frac{x}{y} \right)^{-\mi t}   g(s,t) \left( \frac{U}{t} \right)^{4s}  \omega(t) \,dt
\ll \int_{c_1T}^{c_2T}  |g(s,t) | \,dt 
\ll  \int_{c_1T}^{c_2T}  \left( \frac{t}{2} \right)^{4 \varepsilon_1} \left( 1 + O \left(\frac{|s|^2}{t}\right)
\right) \,dt, 
\end{align*}
which is $\ll (T+|s|^2) T^{4\varepsilon_1}$. 
As $|s| > \sqrt{T}$, it is clear that $(T+|s|^2) T^{4\varepsilon_1}
   \ll T^{4 \varepsilon_1}  |s|^2 = T^{4 \varepsilon_1} \max(|s|,1)^2 $ and thus
\begin{equation*}
\begin{split}
 (T+|s|^2) T^{4\varepsilon_1} \ll  T^{4 \varepsilon_1} \max(|s|,1)^2 
   \left( \frac{\max(|s|,1)}{\sqrt{T}} \right)^{2j}
   = T^{1+4 \varepsilon_1}    \frac{\max(|s|,1)^{2j+2}}{T^{j+1}} .
\end{split}
\end{equation*}
Assuming $x,y\gg 1 $ and $xy \ll U^{4+\varepsilon_2}$, we see 
$$
|\log(\tfrac{x}{y})| \le |\log x| + |\log y|
 \ll \log U \le  \log T \ll_{j} T^{1/j},
 $$ 
which implies that  $|\log(\tfrac{x}{y})|^j T_0^j \ll_j T T_{0}^j \le T^{1+j}$. 

To summarise, we have shown that  for any $s$ with $\Re(s) = \varepsilon_1$,
\[
\int_{-\infty}^{\infty} 
\left( \frac{x}{y} \right)^{-\mi t} g(s,t)  \left( \frac{U}{t} \right)^{4s}  \omega(t) \,dt 
\ll   T^{1+4 \varepsilon_1}  \frac{\max(|s|,1)^{2j+2} }{|\log(\tfrac{x}{y})|^jT_0^j}.
\]
Plugging this bound into \eqref{fstardefn} then yields
\begin{equation}
\begin{split}
  \label{f*bound}
  f^{*}(x,y) \ll 
  \frac{T^{4 \varepsilon_1}}{(xy)^{\varepsilon_1} |\log(\tfrac{x}{y})|^j T_{0}^j}
    \int_{(\varepsilon_1)} \frac{|G(s)|}{|s|} \max(|s|,1)^{2j+2}
    \,|ds|  
 \ll_{j}   \frac{T^{4 \varepsilon_1}}{|\log(\tfrac{x}{y})|^j T_{0}^j}
\end{split}
\end{equation} 
if $x,y \gg 1$ and $xy \ll U^{4 + \e_2}$. 
Now, assume 
\begin{equation*}
   \label{logcondition}
|\log(\tfrac{x}{y})| \gg T_{0}^{-1+\varepsilon_1}.
\end{equation*}
By   \eqref{cond3} and \eqref{f*bound}, for any constant $B>0$, we have 
\begin{equation*}
  \label{f*bound2}
  f^{*}(x,y) \ll \frac{T^{4 \varepsilon_1}}{T_{0}^{j \varepsilon_1}} 
  \le T^{\e_1(4-\frac{j}{2} )} \ll T^{-B}
\end{equation*}
if we choose $j \ge \frac{2B}{\e_1} + 8$. 

Now, setting $m-n=r$, we have obtained
\begin{equation}
\label{IM-2nd-exp}
  I_{M,N} 
= \frac{2T}{\sqrt{MN}}
\sum_{r \ne 0} \sum_{\substack{m-n=r \\ |\log(\frac{m}{n})| \ll T_{0}^{-1+\varepsilon_1}}} \tau_4(m) \tau_4(n) 
W \left(\frac{m}{M} \right) W \left(\frac{n}{N} \right)f^{*}(m,n)
+ O(T^{-B})
\end{equation}
for $MN \ll U^{4+ \e_2}$ and any constant $B>0$. 
We shall further impose several other conditions on $M$ and $N$. First of all, observe that the condition $|\log(\tfrac{m}{n})| \ll T_{0}^{-1+\varepsilon_1}$ forces the above sum to be empty unless  $N/3 \le M \le 3N$. Indeed, if $M < N/3$ or $M >3N$, then for $m,n$ satisfying $W(m/M) W(n/N) \ne 0$, we have $|\log(\tfrac{m}{n})| \ge \log(\tfrac{3}{2})$. In addition, as the conditions in the sum tell us $0 \ne r/N \ll T_{0}^{\varepsilon_1-1}$, we know that either the sum is empty, or  $3N \ge M \ge N/3 \gg |r| T_{0}^{1-\varepsilon_1} \gg T_{0}^{1-\varepsilon_1}$. In light of these, we may only consider the situation  $3N \ge M \ge N/3$ in this section. For the sake of convenience, we write $M \asymp N$ to mean 
\begin{equation}
  \label{MNcondition}
 N/3 \le M \le 3N .
\end{equation}
(We will only use this restricted meaning for the notation $\asymp$ with $M$ and $N$.) Note that the contribution of $M,N$ that do not satisfy $M \asymp N \gg T_{0}^{1-\varepsilon_1}$ to the sum $\sum_{MN \ll U^{4 +\varepsilon_2}} I_{M,N}$ is  $\ll (\log T)^2 T^{-B}$. From this, it remains to evaluate $I_{M,N}$ for $M,N$ such that $M \asymp N$ and $M,N \gg T_{0}^{1-\varepsilon_1}$. Now, fixing $x-y=r$ and using  \eqref{fstardefn}, we then have the following proposition.

\begin{prop} 
\label{IMNestimate}
Let $B >0$ be arbitrary and fixed. For $0 \ne r \in \mathbb{Z}$, we set
\begin{align}  \label{fMN}
 \begin{split}
  f_r(x,y) &:= f_{r;M,N}(x,y)\\
  &:=
   W\left( \frac{x}{M} \right)
W \left( \frac{y}{N} \right)
\frac{1}{2 \pi \mi} \int_{(\varepsilon_1)} \frac{G(s)}{s} 
\left( \frac{1}{\pi^4 xy} \right)^s 
\frac{1}{T} \int_{-\infty}^{\infty} \left( 
1+\frac{r}{y} \right)^{-\mi t} g(s,t)  \left( \frac{U}{t} \right)^{4s}  \omega(t) \,dt \, ds
 \end{split}
\end{align}
if $x,y >0$; otherwise, we set $f_r(x,y)=0$. Let $D_{f_r}(r): =D_{f_r,4,4}(r)$ be the quaternary additive divisor sum  given by \eqref{DfIJ} with $f=f_r$ and $k=\ell =4$.  Then for $M,N \gg T_{0}^{1-\varepsilon_1}$ such that $M \asymp N$ and $MN \ll U^{4+\varepsilon_2}$, we have 
\begin{equation}
  \label{IMnew}
  I_{M,N} = \frac{2T}{\sqrt{MN}} 
\sum_{0 < |r| \ll \frac{M}{T_0} T_{0}^{\varepsilon_1}}  D_{f_r}(r) 
 + O(T^{-B}).
\end{equation}
In addition, for those $M,N \gg 1$, satisfying $MN \ll U^{4 +\varepsilon_2}$, such that either 
$M \not \asymp N$ or $\min(M,N) \ll T_{0}^{1-\varepsilon_1}$, we have the bound $I_{M,N} \ll T^{-B}$.  
\end{prop}

\subsection{Applying the quaternary additive divisor conjecture} \label{apply-divconj}
 Observe that for $ \e_2 \le \e$, we have
\[
  U^{4+\e_2} \le (T^{1-\e})^{4+\e_2}= T^{4+\e_2-4\e-\e \e_2 }
  \le T^{4-3 \e}.
\]
In this section, we shall apply the additive divisor conjecture, Conjecture \ref{divconj}, to handle $I_{M,N}$ in \eqref{IMnew}. This first leads to the following lemma telling us that $f(x,y):=T^{-4 \e_1 } f_r(x,y)$ satisfies the conditions 
\eqref{fsupport} and \eqref{fcond} with $X=M$, $Y=N$, and $P=T^{1+  \e_1 }T_{0}^{-1}$,  where $f_r(x,y)$ is defined as in \eqref{fMN}. (The proof of this technical lemma will be given in Section \ref{appendix2}.)

\begin{lem} \label{fpartials}
Let $0< \e_1, \e_2 \leq \frac{1}{2}$. Then we know
$$
\mathrm{support}(f_{r;M,N}) \subseteq [M,2M] \times [N,2N].
$$
In addition, for  $M \ll U^{2+\varepsilon_2}$, $M \asymp N$, and $1 \le |r| \ll \frac{M}{T_0} T_{0}^{\varepsilon_1}$, we have 
\begin{equation*}
  \label{ffpartials}
  x^m y^n f_{r;M,N}^{(m,n)}(x,y) \ll   T^{4 \varepsilon_1} P^{n},
\end{equation*}
where $P = T^{1+ \e_1} T_{0}^{-1}$.
\end{lem}

By Lemma \ref{fpartials}, the support of $f(x,y)=T^{- 4 \e_1} f_r(x,y)$ is in $[M,2M] \times [N,2N]$, with $M \asymp N$, and $f(x,y)$ satisfies the condition \eqref{fcond}. In addition, we claim  $|r| \ll M^{1-\e_1} \le M^{1-\e'}$ for $\e'\le \e_1 < \frac{1}{6}$. Indeed, if $\e_1 < \frac{1}{6}$ and $\e_2 \leq \e$, we have   $  (MT_0)^{\e_1} \ll  T^{3 \e_1} \le T^{\frac{1}{2}} \le T_0$, which is equivalent to
\[
    \frac{M}{T_0} T_{0}^{\e_1} \ll M^{1-\e_1}.
\]
Now, applying Conjecture \ref{divconj} (with $X=M$, $Y=N$, $P=T^{1+ \e_1}T_{0}^{-1}$, and $f(x,y)=T^{- 4 \e_1} f_r(x,y)$), we derive
\begin{equation}\label{IO1dyadic2}
I_O = \sum_{\substack{M,N \\ M \asymp N, MN \ll U^{4+ \varepsilon_2} \\ M,N \gg T_0^{1-\varepsilon_1} } } \tilde{I}_{M,N}
+ \sum_{\substack{M,N \\ M \asymp N, MN \ll U^{4+ \varepsilon_2} \\ M,N \gg T_0^{1-\varepsilon_1} } }   O(\mathcal{E}_{M,N})
+O(1),
\end{equation}
where 
\begin{equation}
\label{tildeIMN}
\tilde{I}_{M,N} = \frac{2T}{\sqrt{MN}} \sum_{0<|r| \ll \frac{M}{T_0}T_{0}^{\varepsilon_1}} \tilde{D}_{f_r} (r),
\end{equation}
\begin{align}  \label{tildeDfr}
 \begin{split}
\tilde{D}_{f_r} (r)  &= \frac{1}{(2\pi \mi)^2} \int_{\mathcal{B}_1} \int_{\mathcal{B}_2} \zeta^4 (z_1) \zeta^4 (z_2)
\sum_{q=1}^{\infty} \frac{c_q(r) G_4 (z_1,q) G_4(z_2,q)}{q^{z_1 + z_2}} \\
&\times\int_{\max(0,r)}^{\infty} f_r (x,x-r) x^{z_1 -1} (x-r)^{z_2 -1} \, dx \, d {z_2} \, d {z_1},
 \end{split}
\end{align}
and 
\begin{equation}
  \label{EMN}
   \mathcal{E}_{M,N}  =
   \frac{T^{1+4\e_1}T^{C \e_1 }}{\sqrt{MN}} \sum_{0 < |r| \ll \frac{M}{T_0} T_{0}^{\varepsilon_1}}  \left(\frac{T}{T_0} \right)^{C}M^{\frac{1}{2}+\varepsilon_0}.
\end{equation} 

Therefore, assuming Conjecture \ref{divconj}, to study $I_O$, we shall estimate $\tilde{I}_{M,N}$ and  $\mathcal{E}_{M,N}$ for the sums in \eqref{IO1dyadic2}. To end this section, we bound the contribution of $\mathcal{E}_{M,N}$ to $I_O$ in \eqref{IO1dyadic2} by proving the following lemma.  

\begin{lem}    \label{adddiverror}
Let $\e, \e_0, \e_1, \e_2 >0$ be constants such that
\begin{equation*}
 \max\{ \e_0, \e_1,  C \e_1 \} \le \frac{\e}{40} \text{ and }  \e_2 \le \e. 
\end{equation*}
Let  $\mathcal{E}_{M,N}$ be defined as in \eqref{EMN}. Then for $T$ sufficiently large with respect to these parameters,  we have
\begin{equation*}
 \sum_{\substack{ M, N \\ M \asymp N, MN \ll U^{4+\varepsilon_2}}}   \mathcal{E}_{M,N}  \ll   
  \left( \frac{T}{T_0} \right)^{1+C} T^{1-\frac{\e}{2}}.
\end{equation*}
\end{lem}

\begin{proof}
As $\e_0 \max\{ C, 1 \} \le \e_1$ and $M \asymp N$, a straightforward calculation shows that $ \mathcal{E}_{M,N}$ is 
\begin{align*}
 \ll \frac{T^{1 + 4\e_1+C \e_1}}{M}  \sum_{0 < |r| \ll \frac{M}{T_0} T^{\varepsilon_1}}  
  \left( \frac{T}{T_0} \right)^{C}M^{\frac{1}{2}+\varepsilon_0}
  \ll  T^{\frac{\e}{8}} \left( \frac{T}{M}   \right)
      \left( \frac{T}{T_0} \right)^{C} M^{\frac{1}{2} +\varepsilon_0}
      \left(
       \frac{M}{T_0}  T^{\varepsilon_1} \right),
\end{align*} 
which is  $\ll 
    T^{  \frac{\e}{4}} \left( \frac{T}{T_0} \right)^{1+C} M^{\frac{1}{2}}$. 
Hence, we have
$$
\sum_{\substack{ M \asymp N\\ MN \ll U^{4+ \e_2} }}   \mathcal{E}_{M,N}
  =  \sum_{\substack{ M \asymp N\\ MN \ll U^{4+ \e_2} }}  T^{\frac{\e}{4}} \left( \frac{T}{T_0} \right)^{1+C} M^{\frac{1}{2}} 
  \ll  T^{  \frac{\e}{4}} \left( \frac{T}{T_0} \right)^{1+C}  
   \sum_{\substack{M \ll U^{2+\frac{\e_2}{2}} \\
  N \asymp M} } M^{\frac{1}{2}}, 
$$
which is
$$
 \ll T^{\frac{\e}{4}} \left( \frac{T}{T_0} \right)^{1+C} ( U^{2+\frac{\e_2}{2}})^{\frac{1}{2}}
     =  T^{ \frac{\e}{4}} \left( \frac{T}{T_0} \right)^{1+C}  U^{1+ \frac{\e_2}{4}}.
$$ 
Finally, as $\e_2 \le \e$, it is clear that
$$
  U^{1+ \frac{\e_2}{4}} = (T^{1-\e})^{1+ \frac{\e_2}{4}} = T^{1-\e +\frac{\e_2}{4}-\frac{\e \e_2}{4}} \le T^{1-\frac{3\e}{4}},
$$
and thus
\begin{equation*}
  \sum_{\substack{ M, N \\ M \asymp N, MN \ll U^{4+\varepsilon_2}}}   \mathcal{E}_{M,N} 
  \ll  T^{ \frac{\e}{4}} \left( \frac{T}{T_0} \right)^{1+C} T^{1-\frac{3\e}{4}}
  \le \left( \frac{T}{T_0} \right)^{1+C} T^{1-\frac{\e}{2}}
\end{equation*}
as desired. 
\end{proof}

\subsection{Expanding the range of $|r|$} \label{expandingr}

Recall that by Lemma \ref{adddiverror}, the equation \eqref{IO1dyadic2} yields 
\begin{align*}
I_O = \sum_{\substack{M,N \\ M \asymp N, MN \ll U^{4+ \varepsilon_2} \\ M,N \gg T_0^{1-\varepsilon_1} } } \tilde{I}_{M,N}
+   O \left( \left( \frac{T}{T_0} \right)^{1+C} T^{1-\frac{\e}{2}} \right),
\end{align*}
where $ \tilde{I}_{M,N}$ is defined in \eqref{tildeIMN}. Note that  for the $s$-integral of $ \tilde{I}_{M,N}$  (see \eqref{fMN}), we can move the integral from the line $\Re(s) = \e_1$ to $\Re(s) = \e_3$, where $0< \delta< \e_3<0.15$, without encountering any poles. 
To handle the sum above, we first control $\tilde{I}_{M,N}$, defined in \eqref{tildeIMN}, by expanding the range of the sum in \eqref{tildeIMN}  to  $0 < |r| \le R_0$ with 
\begin{equation*}
R_0 =T^{5}.
\end{equation*} 
Observe that from \eqref{fstardefn} and \eqref{fMN}, it follows that the $x$-integral in \eqref{tildeDfr} can be written as
\begin{equation}
  \label{iMNr}
 i_{M,N,r} =  \int_{\max(0,r)}^{\infty} W \left( \frac{x}{M} \right) W \left(  \frac{x-r}{N} \right) f^{*}(x,x-r) 
 x^{z_{1}-1}(x-r)^{z_{2}-1} \,  dx.
\end{equation}
Since $W$ is supported in $[1,2]$, we may only consider the situation that the variable $x$ in \eqref{iMNr} lies in $(M,2M) \cap (N+r, 2N+r)$. It is clear that if such an open interval is empty, then $ i_{M,N,r} =0$.  Without loss of generality, we consider the case $r \ge 1$. Note that the condition $x>N+r$ gives $x-r \ge N \ge 2^{-\frac{1}{2}}$, which allows us to have $x \ge r+2^{-\frac{1}{2}}$. Also, as we know  $N+r \le 2M$ (because $(M,2M) \cap (N+r, 2N+r)$ is empty otherwise).  This and assumption  \eqref{MNcondition} give 
$r \le 2M-N \le 2M- \frac{M}{3}= \frac{5M}{3}$. 
In light of these observations, applying \eqref{f*bound}, we conclude that for $r \gg \frac{M}{T_0} T_{0}^{\varepsilon_1}$  and $j$ sufficiently large, 
$$
  f^{*}(x,x-r) 
  \ll \frac{T^{4 \e_3  }}{ |\log(\frac{x}{x-r})  T_0|^j} \le \frac{T^{4 \e_3 }}{ |\log(\frac{2M}{2M-r})  T_0|^j}
$$
as $x \in [r+2^{-\frac{1}{2}},2M]$. Thus, since $|\log(1-x)| \gg x$  for $x\in (0,\frac{5}{6} ] $, we arrive at
\begin{equation}
  \label{fstarxr}
  f^{*}(x,x-r) 
  \ll   \frac{T^{4 \e_3  }}{ |\log(1-\frac{r}{2M})  T_0|^j}   \ll   \frac{T^{4 \e_3  }}{  |\frac{r}{2M} T_0|^j}  \ll T^{4 \e_3  -j \varepsilon_1}  \ll T^{-B}.
\end{equation}
Using  \eqref{iMNr} and \eqref{fstarxr},  we then deduce $i_{M,N,r} \ll M^{1+2 \delta}T^{-B}$ for $r \gg \frac{M}{T_0} T^{\varepsilon_1}$. Arguing similarly, we also have
$i_{M,N,r} \ll M^{1+2\delta} T^{-B}$ for $r$ such that $-r \gg \frac{M}{T_0}  T^{\varepsilon_1}$.

Note that we are trying to bound 
\begin{equation*}
   \sum_{\substack{M,N \\ M \asymp N, MN \ll U^{4+ \varepsilon_2} \\ M,N \gg T_0^{1-\varepsilon_1} } } 
   \sum_{\frac{M}{T_0}T_{0}^{\varepsilon_1} \ll  |r| \le R_0}  \frac{2T}{\sqrt{MN}}  |\tilde{D}_{f_r} (r)|. 
\end{equation*}
From \eqref{tildeDfr} and Lemma  \ref{G4formulae}, it follows that
\begin{align}  \label{tildeDfr-bd}
 \begin{split}
|\tilde{D}_{f_r} (r)|& \ll \int_{\mathcal{B}_1} \int_{\mathcal{B}_2}  \frac{1}{(r_1 r_2)^4} 
\sum_{q=1}^{\infty} \frac{ (q,r)  ((32)^{\omega(q)} \tau_4(q))^2 q^{r_1+r_2}}{q^{2-r_1-r_2}} \\
&\times \int_{\max(0,r)}^{2M} T^{-B} |x|^{\Re(z_1) -1} |x-r|^{\Re(z_2) -1} \, dx \, |d {z_2}| \, |d {z_1}| \\
& \ll   \frac{1}{(r_1 r_2)^4}  (r_1 r_2)  \sum_{q=1}^{\infty}  \frac{ (q,r) 32^{2\omega(q)} \tau_4(q)^2}{q^{2-2(r_1+r_2)}} T^{-B} M^{1+r_1+r_2} \\
& \ll \tau_2(r) T^{-B} M^{1+r_1+r_2}
 \end{split}
\end{align}
as long as $r_1+r_2 \le \frac{1}{4}$ and $r_1, r_2$ are bounded away from zero. (Here we used the bounds 
$\sum_{q=1}^{\infty} \frac{(q,r)}{q^c} \le \tau_2(r) \zeta(c)$, for $c >1$, and $\omega(q) \ll \frac{\log q}{\log \log q}+ 1$.) Thus, we have 
\begin{align*}
\begin{split}
  & \sum_{\substack{M,N \\ M \asymp N, MN \ll U^{4+ \varepsilon_2} }} 
   \sum_{\frac{M}{T_0}T_{0}^{\varepsilon_1} \ll  |r| \le R_0}  \frac{2T}{\sqrt{MN}} \tau_2(r) T^{-B} M^{1+r_1+r_2}  \\
  & \ll T^{1-B}   \sum_{\substack{M,N \\ M \asymp N, MN \ll U^{4+ \varepsilon_2} }}   M^{1+r_1+r_2} (R_0 \log R_0) \\
  & \ll T^{1-B} (R_0 \log R_0)   (U^{2+ \varepsilon_2/2})^{1+r_1+r_2}  \\
  & \ll T^{1-B} R_0 T^{(2+\varepsilon_2/2)(1+r_1+r_2)} \log T \\
  & \ll T^{1+5 + (2+\varepsilon_2/2)(1+r_1+r_2) -B} \log T. 
\end{split}
\end{align*}
So, if we choose $B \ge 2+5 + (2+\varepsilon_2/2)(1+r_1+r_2)$, then we have an error term that is $O((\log T)/T)$. 
Hence, we have established
\begin{equation} \label{equ:I_O-mid}
  I_{O} = \sum_{\substack{M,N \\ M \asymp N, MN \ll U^{4+ \varepsilon_2} \\ M,N \gg T_0^{1-\varepsilon_1} } }  \frac{2T}{\sqrt{MN}} \sum_{0<|r| \le R_0} \tilde{D}_{f_r} (r)+  O \left( \left( \frac{T}{T_0} \right)^{1+C} T^{1-\frac{\e}{2}} \right).
\end{equation}

\subsection{Making $M$ and $N$ unrestricted}

In this section, we shall remove the conditions on $M$ and $N$ from the sum in \eqref{equ:I_O-mid} by showing that $i_{M,N,r}$ is small for $M$ and $N$ not satisfying the conditions. Firstly, for $M \not \asymp N$, we know $N > 3M$ or $N < M/3$. Without loss of generality, we consider the case that $N > 3M$.  From the support of $W$, in \eqref{iMNr}, we only need to consider $x$ satisfying 
$M \le x \le 2M$ and $N \le x-r \le 2N$, which tell us 
$|\log ( \frac{x}{x-r} )| = \log  ( \frac{x-r}{x}  ) \ge \log(\frac{N}{2M}) \ge \log(\frac{3}{2})$. Hence,  by  \eqref{f*bound}, \eqref{iMNr} and the bound $|z_i-1|<\frac{1}{10}$, we deduce
$$
   i_{M,N,r}  \ll_j \int_{M}^{2M} \frac{T^{4\e_3  } }{T_0^j} (MN)^{\frac{1}{10}} \,dx
   \ll \frac{T^{4 \e_3 } }{T_0^j} M (MN)^{\frac{1}{10}} 
   \ll \frac{T^{4\e_3 } }{T_0^j} \frac{MN}{T_0^{1-\e_1}} (MN)^{\frac{1}{10}}
    \ll \frac{T^{4\e_3  +\frac{11}{10} (1-\e)(4+\varepsilon_2) }}{T_0^{j+1-\e_1 }},
$$
where we used the condition $N \gg T_{0}^{1-\e_1}$ (resp. $MN \ll U^{4+\varepsilon_2}= T^{(1-\e)(4+\varepsilon_2)}$) for the third (resp. last) estimate. Using \eqref{cond3},  we can choose $j$ that is large enough so that  $i_{M,N,r}  \ll T^{-B}$. Therefore, for $M$ and $N$ with $M \not \asymp N$, processing a similar argument as in the previous section, we obtain $\frac{2T}{\sqrt{MN}} \sum_{0<|r| \le R_0} \tilde{D}_{f_r} (r) \ll (\log T)^{-2}$. 

Secondly, we deal with the situation that either $M$ or $N \ll T_{0}^{1-\varepsilon_1}$. We may suppose, without loss of generality, that $M,N \ge 2^{-\frac{1}{2}}$, $M \ll
T_{0}^{1-\varepsilon_1}$, and $r\ge 1$. By the support of $W$, in \eqref{iMNr}, we can require that the variable $x$ satisfies  $x \ge N+r > r \ge 1$, which implies  $|\log(\frac{x}{x-r})| \ge \log(\frac{x}{x-1}) > x^{-1} \ge (2M)^{-1} \gg T_{0}^{\varepsilon_1-1}$.  Thus, by \eqref{f*bound}, for any $B>0$, we derive
$$
f^{*}(x,x-r) \ll T^{4\e_3 } T_{0}^{ -j (\varepsilon_1-1) -j} \ll T^{-B}
$$
when $j$ is sufficiently large. With this bound in hand, we again have  $ \frac{2T}{\sqrt{MN}} \sum_{0<|r| \le R_0} \tilde{D}_{f_r} (r)  \ll (\log T)^{-2}$.  

From the above discussion, if we include all $M$ and $N$, with
$MN \ll U^{4 +\varepsilon_2}$, such that  either $M \not \asymp N$ or $\min(M,N) \ll T_{0}^{1-\varepsilon_1}$, in the sum in \eqref{equ:I_O-mid}, we will introduce an extra error, which is at most
 $(\log T)(\log T)(\log T)^{-2}  \ll 1$ and thus negligible, as the choices for each of such $M$ and $N$ are $O(\log T)$.

Lastly, we consider the contribution $\widetilde{\Delta}$ from those pairs $(M,N)$  in \eqref{equ:I_O-mid} with $MN \gg U^{4 +\varepsilon_2}$ . Using \eqref{fstarV} and Lemma \ref{lem:bd-g}(iv), for any constant $A>0$, we obtain
$f^{*}(x,y) \ll ( \frac{U^{4}}{xy} )^{A}$ if $xy >U^{4}$. Similar to  \eqref{tildeDfr-bd}, for any arbitrary $A>0$, we have
\begin{align*}
\widetilde{\Delta}
 & \ll    \sum_{r \ne 0}  \tau_2(|r|) 
  \sum_{\substack{ M,N \\   MN \gg U^{4+\varepsilon_2}  \\ M-2N < r < 2M-N }}   \frac{T}{\sqrt{MN}}
  \int_{\max( M, N+r,r,0)}^{\min(2M,2N+r)}  \left(\frac{U^4}{x(x-r)}  \right)^{A}  
  (MN)^{\frac{1}{10}}
 \, dx \\
  & \ll  T   \sum_{r \ne 0}  \tau_2(|r|) 
  \sum_{\substack{M,N \\ MN \gg \max(U^{4+\varepsilon_2}, |r|)}  }  (MN)^{-\frac{1}{2}+\frac{1}{10}} 
  \left( \frac{U^4}{MN} \right)^{A} \min(M,N)
\end{align*}
provided that $r_1+r_2 \le \frac{1}{4}$ and $r_1, r_2$ are bounded away from zero. To estimate the last sum, we may consider $MN=H$, which forces $M$ and $N$ to satisfy $\min(M,N) \le H^{\frac{1}{2}}$.  Clearly, there is an integer $h=h(H) \ge 0$ such that $H=2^{\frac{h-2}{2}} \gg U^{4+\varepsilon_2}$. Also, we observe that for every given $H$, we can have at most $h(H)+1 \ll \log H$ pairs $M,N$ satisfying $MN=H$. Hence, we arrive at
\begin{align*}
\widetilde{\Delta}
 & \ll T  \sum_{H \gg U^{4+\varepsilon_2}}
 H^{-\frac{1}{2}+\frac{1}{10}} \left( \frac{U^4}{H} \right)^{A}
 H^{\frac{1}{2}} (\log H)   \sum_{1 \le r \ll H} \tau_2(r) 
 \ll T     \sum_{H \gg U^{4+\varepsilon_2}}  H^2   \left( \frac{U^4}{H} \right)^{A}.
\end{align*}
Observe that for $A>2$, the last sum is convergent and is $\ll U^{2(4+\e_2) -A\e_2 } = T^{(1-\e)(2(4+\e_2) -A\e_2 )} $. Therefore, we establish $\widetilde{\Delta}\ll_A 1$ for any $A \ge \frac{1}{(1-\e)\e_2} + \frac{8}{\e_2} +2$.

To summarise, we have removed all the restrictions on $M$ and $N$ appearing in \eqref{equ:I_O-mid}. More precisely, we have shown 
\begin{equation} \label{equ:I_O-1}
  I_{O} = \sum_{M,N}  \frac{2T}{\sqrt{MN}} \sum_{0<|r| \le R_0} \tilde{D}_{f_r} (r)+  O \left( \left( \frac{T}{T_0} \right)^{1+C} T^{1-\frac{\e}{2}} \right).
\end{equation}

\subsection{A smooth partition of unity, again}

Recalling $W(x) = x^{-\frac{1}{2}} W_0 (x)$,
we see 
\begin{align*}
\sum_{M,N} \frac{1}{\sqrt{MN}} W \left(\frac{x}{M} \right) W \left(\frac{x-r}{N} \right)
&= \sum_{M,N} \frac{1}{\sqrt{MN}}  \left(\frac{x}{M} \right)^{-1/2}  \left(\frac{x-r}{N} \right)^{-1/2}
W_0 \left(\frac{x}{M} \right) W_0 \left(\frac{x-r}{N} \right)\\
&= x^{-1/2} (x-r)^{-1/2}  \sum_{M,N} W_0 \left(\frac{x}{M} \right) W_0 \left(\frac{x-r}{N} \right),
\end{align*}
which is equal to $x^{-1/2} (x-r)^{-1/2}$ as $ \sum_{k\in \mathbb{Z}} W_0 (x/2^{\frac{k}{2}}) =1$.
From \eqref{fMN} and the above identity, inserting \eqref{tildeDfr} into \eqref{equ:I_O-1}, we have
\begin{align*}
 \begin{split}
I_O& =2 T  \sum_{1\leq |r| \leq R_0} 
\frac{1}{(2\pi \mi)^2} \int_{\mathcal{B}_1} \int_{\mathcal{B}_2} \zeta^4 (z_1) \zeta^4 (z_2)
\sum_{q=1}^{\infty} \frac{c_q(r) G_4 (z_1,q) G_4(z_2,q)}{q^{z_1 + z_2}}  \\
&\times \int_{\max(0,r)}^{\infty} f^* (x,x-r) x^{z_1 -3/2} (x-r)^{z_2 -3/2} \, dx \, d {z_2} \, d {z_1} +  O \left( \left( \frac{T}{T_0} \right)^{1+C} T^{1-\frac{\e}{2}} \right),
 \end{split}
\end{align*}
where 
\begin{equation*}
   f^*(x,x-r) = \frac{1}{2 \pi \mi} \int_{(\e_3 )} 
\frac{G(s)}{s} \left( \frac{1}{\pi^{4} x(x-r)} \right)^s 
\frac{1}{T} \int_{-\infty}^{\infty} \left( 
1+\frac{r}{x-r} \right)^{-\mi t} g(s,t) \left( \frac{U}{t} \right)^{4s} \omega(t) \,dt \,ds, 
\end{equation*}
Hence, we have
\begin{align*}
I_O& = 2 \sum_{1\leq |r| \leq R_0} 
\frac{1}{(2\pi \mi)^2} \int_{\mathcal{B}_1} \int_{\mathcal{B}_2} \zeta^4 (z_1) \zeta^4 (z_2)
\sum_{q=1}^{\infty} \frac{c_q(r) G_4 (z_1,q) G_4(z_2,q)}{q^{z_1 + z_2}}
\int_{\max(0,r)}^{\infty} x^{z_1 -\frac{3}{2}} (x-r)^{z_2 -\frac{3}{2}} \\
&\times  \frac{1}{2 \pi \mi} \int_{( \e_3 )} 
\frac{G(s)}{s} \left( \frac{1}{\pi^{4} x(x-r)} \right)^s 
 \int_{-\infty}^{\infty} \left( 
1+\frac{r}{(x-r)} \right)^{-\mi t} g(s,t)  \left( \frac{U}{t} \right)^{4s} \omega(t) \, dt \, ds\, 
dx \, d {z_2} \, d {z_1}\\
&+ O \left( \left( \frac{T}{T_0} \right)^{1+C} T^{1-\frac{\e}{2}} \right).
\end{align*}
We now split this sum, according to $r >0$ and $r<0$, to obtain 
\begin{align*}
I_O& =  2 \sum_{1\leq r \leq R_0} 
\frac{1}{(2\pi \mi)^2} \int_{\mathcal{B}_1} \int_{\mathcal{B}_2} \zeta^4 (z_1) \zeta^4 (z_2)
\sum_{q=1}^{\infty} \frac{c_q(r) G_4 (z_1,q) G_4(z_2,q)}{q^{z_1 + z_2}} 
 \int_{\max(0,r)}^{\infty} x^{z_1 -\frac{3}{2}} (x-r)^{z_2 -\frac{3}{2}} \\
&\times  \frac{1}{2 \pi \mi} \int_{(\e_3  )} 
\frac{G(s)}{s} \left( \frac{1}{\pi^{4} x(x-r)} \right)^s 
 \int_{-\infty}^{\infty} \left( 
1+\frac{r}{(x-r)} \right)^{-\mi t} g(s,t)  \left( \frac{U}{t} \right)^{4s} \omega(t) \, dt \, ds\, 
dx \, d {z_2} \, d {z_1} \\
& + 2 \sum_{1\leq r \leq R_0} 
\frac{1}{(2\pi \mi)^2} \int_{\mathcal{B}_1} \int_{\mathcal{B}_2} \zeta^4 (z_1) \zeta^4 (z_2)
\sum_{q=1}^{\infty} \frac{c_q(-r) G_4 (z_1,q) G_4(z_2,q)}{q^{z_1 + z_2}}
 \int_{\max(0, -r)}^{\infty} x^{z_1 -\frac{3}{2}} (x+r)^{z_2 -\frac{3}{2}} \\
&\times  \frac{1}{2 \pi \mi} \int_{(\e_3  )} 
\frac{G(s)}{s} \left( \frac{1}{\pi^{4} x(x+r)} \right)^s 
 \int_{-\infty}^{\infty} \left( 
1+\frac{ -r }{(x+r)} \right)^{-\mi t} g(s,t)  \left( \frac{U}{t} \right)^{4s} \omega(t) \, dt \, ds\, 
dx \, d {z_2} \, d {z_1}\\
&+ O \left( \left( \frac{T}{T_0} \right)^{1+C} T^{1-\frac{\e}{2}} \right).
\end{align*}
In the $x$-integrals, we make the variable change $x=ry$ to deduce   
\begin{align*}
I_O
& =  2 \sum_{1\leq r \leq R_0} 
\frac{1}{(2\pi \mi)^2} \int_{\mathcal{B}_1} \int_{\mathcal{B}_2} \zeta^4 (z_1) \zeta^4 (z_2)
\sum_{q=1}^{\infty} \frac{c_q(r) G_4 (z_1,q) G_4(z_2,q)}{q^{z_1 + z_2}}  r^{z_1+z_2-3}\\
&\times  \int_{\max(0,1) }^{\infty} y^{z_1 -\frac{3}{2}} (y-1)^{z_2 -\frac{3}{2}} \\
&\times  \frac{1}{2 \pi \mi} \int_{(\e_3  )} 
\frac{G(s)}{s}  \frac{1}{r^{2s}} \left( \frac{1}{\pi^{4} y(y-1)} \right)^s 
 \int_{-\infty}^{\infty} \left(  \frac{y-1}{y} 
 \right)^{\mi t} g(s,t)  \left( \frac{U}{t} \right)^{4s} \omega(t) \, dt \, ds\, 
(r dy) \, d {z_2} \, d {z_1} \\
& +  2\sum_{1\leq r \leq R_0} 
\frac{1}{(2\pi \mi)^2} \int_{\mathcal{B}_1} \int_{\mathcal{B}_2} \zeta^4 (z_1) \zeta^4 (z_2)
\sum_{q=1}^{\infty} \frac{c_q(r) G_4 (z_1,q) G_4(z_2,q)}{q^{z_1 + z_2}} 
r^{z_1+z_2-3}\\
&\times  \int_{\max(0,-1)}^{\infty} y^{z_1 -\frac{3}{2}} (y+1)^{z_2 -\frac{3}{2}} \\
&\times  \frac{1}{2 \pi \mi} \int_{(\e_3  )} 
\frac{G(s)}{s}  \frac{1}{r^{2s}}  \left( \frac{1}{\pi^{4} y(y+1)} \right)^s 
 \int_{-\infty}^{\infty} \left(  \frac{y+1}{y}
 \right)^{\mi t} g(s,t)  \left( \frac{U}{t} \right)^{4s} \omega(t) \, dt \, ds\, 
(r \, dy) \, d {z_2} \, d {z_1}\\
&+ O \left( \left( \frac{T}{T_0} \right)^{1+C} T^{1-\frac{\e}{2}} \right).
\end{align*}
It follows that $I_O= I^{+} + I^{-} +  O ( ( \frac{T}{T_0} )^{1+C} T^{1-\frac{\e}{2}} )$ where 
\begin{align*}
  I^{\pm}  & = \frac{2}{(2\pi \mi)^2} \sum_{1\leq r \leq R_0} 
 \int_{\mathcal{B}_1} \int_{\mathcal{B}_2} \zeta^4 (z_1) \zeta^4 (z_2)
\sum_{q=1}^{\infty} \frac{c_q(r) G_4 (z_1,q) G_4(z_2,q)}{q^{z_1 + z_2}}   r^{z_1+z_2-2}\\
 &\times \int_{\max(0,\pm 1) }^{\infty} y^{z_1 -\ \frac{3}{2}} (y \mp 1)^{z_2 - \frac{3}{2}} \\
&\times  \frac{1}{2 \pi \mi} \int_{(\e_3  )} 
\frac{G(s)}{s}  \frac{1}{r^{2s}} \left( \frac{1}{\pi^{4} y(y \mp 1)} \right)^s 
 \int_{-\infty}^{\infty} \left(  \frac{y \mp 1}{y} 
 \right)^{\mi t} g(s,t)  \left( \frac{U}{t} \right)^{4s} \omega(t) \, dt \, ds\,  dy \, d {z_2} \, d {z_1}. 
\end{align*}
We see 
\begin{align*}
  I^{\pm}  & =  
\frac{2}{(2\pi \mi)^3} \int_{\mathcal{B}_1} \int_{\mathcal{B}_2} \int_{(\e_3  )}  \zeta^4 (z_1) \zeta^4 (z_2)
 \sum_{1\leq r \leq R_0}
\sum_{q=1}^{\infty} \frac{c_q(r) G_4 (z_1,q) G_4(z_2,q)}{q^{z_1 + z_2} r^{2s+2-z_1-z_2}}\\
&\times  
\frac{G(s)}{s}   \frac{1}{  \pi^{4s} }
 \int_{-\infty}^{\infty} g(s,t)  \left( \frac{U}{t} \right)^{4s} \omega(t) 
\left(  \int_{\max(0,\pm 1) }^{\infty} y^{z_1 -s-3/2-\mi t} (y \mp 1)^{z_2-s - \frac{3}{2}+\mi t } \, dy  \right)  \, dt \, ds\,  d {z_2} \, d {z_1}. 
\end{align*} 
For $\eta =\pm 1$, we have 
$$
   \int_{\max(0,\pm 1) }^{\infty} y^{z_1 -s- \frac{3}{2}-\mi t} (y \mp 1)^{z_2-s - \frac{3}{2}+\mi t }
\,dy =  \int_{\max(0,\eta) }^{\infty} y^{z_1 -s- \frac{3}{2}-\mi t} (y -\eta)^{z_2-s - \frac{3}{2}+\mi t } \,dy,
$$
which equals
$$ \begin{cases}
\mathrm{B}(z_2-s-\frac{1}{2}+\mi t , -z_1-z_2+2s+2) & \text{ if } \eta=1; \\
\mathrm{B}(z_1-s-\frac{1}{2}-\mi t,-z_1-z_2+2s+2) & \text{ if } \eta=-1;
\end{cases}
  =
 \begin{cases}
\frac{\Gamma(z_2-s-\frac{1}{2}+\mi t ) \Gamma( -z_1-z_2+2s+2)}{ \Gamma(s-z_1+\frac{3}{2}+\mi t )  } & \text{ if } \eta=1; \\
\frac{\Gamma(z_1-s-\frac{1}{2}-\mi t) \Gamma(-z_1-z_2+2s+2)}{\Gamma(s-z_2+\frac{3}{2}-\mi t) } & \text{ if } \eta=-1.
\end{cases} 
$$
Here, we used the following identities for  the Beta function $\mathrm{B}$:
$$
\int_{0}^{\infty}\frac{x^{u-1}}{(1+x)^{v}}\,dx=\mathrm{B}(u,v-u)=\frac{\Gamma(u)\Gamma(v-u)}{\Gamma(v)}
$$
for $\Re(u)>0$ and $\Re(v)>0$. Therefore, it follows that 
\begin{align}\label{IO-mid-expression}
 \begin{split}
 &  I_O \\ & =  
\frac{2}{(2\pi \mi)^3} \int_{\mathcal{B}_1} \int_{\mathcal{B}_2} \int_{(\e_3)}  \zeta^4 (z_1) \zeta^4 (z_2)
 \sum_{1\leq r \leq R_0}
\sum_{q=1}^{\infty} \frac{c_q(r) G_4 (z_1,q) G_4(z_2,q)}{q^{z_1 + z_2} r^{2s+ 2-z_1-z_2}} 
\frac{G(s)}{s  \pi^{4s} }   
 \int_{-\infty}^{\infty} g(s,t)  \left( \frac{U}{t} \right)^{4s} \omega(t)  \\ 
 & 
\times  \Gamma( -z_1-z_2+2s+2)\left( 
\frac{\Gamma(z_2-s-\frac{1}{2}+\mi t )}{ \Gamma(s-z_1+\frac{3}{2}+\mi t )  } 
+ 
\frac{\Gamma(z_1-s-\frac{1}{2}-\mi t) }{\Gamma(s-z_2+\frac{3}{2}-\mi t) }
   \right)  \, dt \, ds\,   d {z_2} \, d {z_1}\\
   &+ O \left( \left( \frac{T}{T_0} \right)^{1+C} T^{1-\frac{\e}{2}} \right). 
 \end{split}  
\end{align}

\subsection{Moving the integral right to $\Re(s)=1$}
 
We now move the $s$-integral in \eqref{IO-mid-expression} right to the line $\Re(s)=1$. We do this because
we may apply Proposition \ref{Hlemma}(i) later.
Observe that the term in the rounded bracket in \eqref{IO-mid-expression}  has simple poles at 
$s=\p_1 = -\frac{1}{2}+ z_1 - \mi t$  and $s=\p_2 = -\frac{1}{2} + z_2  + \mi t$.    
Using $\Gamma(z) \sim z^{-1}$ as $z \to 0$, we find that the residue at $s=\p_1$ (for the $s$-integral) is
\begin{align*}
-\frac{G(\p_1)}{\p_1 \pi^{4\p_1} r^{2\p_1}} g(\p_1, t) \left( \frac{U}{t} \right)^{4\p_1} \frac{ \Gamma (-z_1 -z_2 + 2\p_1 + 2)}{\Gamma(\p_1 - z_2 + \frac{3}{2} -\mi t)},
\end{align*}
and the residue at $s = \p_2$ is 
\begin{align*}
-\frac{G(\p_2)}{\p_2 \pi^{4\p_2} r^{2\p_2}} g(\p_2, t) \left( \frac{U}{t} \right)^{4\p_2} \frac{ \Gamma (-z_1 -z_2 + 2\p_2 + 2)}{\Gamma(\p_2 - z_1 + \frac{3}{2} +\mi t)} .
\end{align*}
Inserting $\p_1 = -\frac{1}{2}+ z_1 - \mi t$  and $\p_2 = -\frac{1}{2} + z_2  + \mi t$ into the above residues, respectively, we see the residue at $\p_i$ is
\[
-  \frac{G(\p_i)}{\p_i \pi^{4\p_i} r^{2\p_i}} g(\p_i, t) \left( \frac{U}{t} \right)^{4\p_i} .
\]
By Lemma \ref{lem:bd-g}(i), when $c_1 T \leq t \leq c_2 T$,  $\frac{G(\p_i)}{\p_i \pi^{4\p_i} r^{2\p_i}} g(\p_i, t) \left( \frac{U}{t} \right)^{4\p_i}$ is
\[
 \ll \frac{|G(\p_i)|}{ \sqrt{|\Im(z_i) + (-1)^i t|^2 + 1}} r^{1 - 2\Re(z_i)}
 \left( \frac{U}{2} \right)^{-2 + 4\Re(z_i)}    \left( \frac{(|\Im(z_i) + (-1)^i t|)^2+1}{t}  \right).
\]
As  $G(\p_i) \ll (|\Im(z_i) + (-1)^i t|)^2+1)^{-B}$, we conclude that the  contribution of each residue  is at most
\begin{align*} 
 \begin{split}
&   \sum_{1 \le r \le R_0} \frac{ |\zeta(z_1)\zeta(z_2)|^4}{r^{2-\Re(z_1) -\Re(z_2) }}\sum_{q=1}^{\infty} \frac{|c_q(r) G_4 (z_1,q) G_4(z_2,q)|}{q^{\Re(z_1) + \Re(z_2)} } \\
&\times  \int_{c_1 T}^{c_2 T}  
 \left( \frac{T}{2} \right)^{-2 + 4\Re(z_i)}     \frac{((|\Im(z_i) + (-1)^i t|)^2+1)^{-B}}{t} r^{1-2\Re(z_i)}\,dt  .
 \end{split}
\end{align*}
%
%
%
%
The sum over $q$ can be bounded by a similar manner as in \eqref{tildeDfr}. Therefore the above  is bounded by 
\begin{align*}
    \ll   T^{3-B} \sum_{1 \le r \le R_0} \frac{\tau_2(r) }{r^{1-r_1-r_2}} \ll T^{3-B} R_0 =  T^{8-B}.
\end{align*}
Choosing  $B$ to be large, we then obtain 
\begin{align} \label{I111b}
 \begin{split}
  & I_O\\
  & = \frac{2}{(2\pi \mi)^3} \int_{\mathcal{B}_1} \int_{\mathcal{B}_2} \int_{(1)}  \zeta^4 (z_1) \zeta^4 (z_2)
 \sum_{1\leq r \leq R_0}
\sum_{q=1}^{\infty} \frac{c_q(r) G_4 (z_1,q) G_4(z_2,q)}{q^{z_1 + z_2} r^{2s+ 2-z_1-z_2}}
\frac{G(s)}{s  \pi^{4s} }   
 \int_{-\infty}^{\infty} g(s,t)  \left( \frac{U}{t} \right)^{4s} \omega(t)  \\ 
 & 
\times \Gamma( -z_1-z_2+2s+2) \left( 
\frac{\Gamma(z_2-s-\frac{1}{2}+\mi t ) }{ \Gamma(s-z_1+\frac{3}{2}+\mi t )  } 
+ 
\frac{\Gamma(z_1-s-\frac{1}{2}-\mi t) }{\Gamma(s-z_2+\frac{3}{2}-\mi t) }
   \right)  \, dt \, ds\,  d {z_2} \, d {z_1} \\
&   + O(1 ) +  O \left( \left( \frac{T}{T_0} \right)^{1+C} T^{1-\frac{\e}{2}} \right).
\end{split}
\end{align}

\subsection{Extending $R_0 \to \infty$}

In this section, we shall extend $R_0$ to $\infty$ for the sum in \eqref{I111b}. We start by observing $\Re(s - z_i + 1) \in [1-\delta, 1+ \delta] \subseteq [0.9, 1.1 ],$ $i=1,2$. It implies that $\Re(s - z_1 + 1 + s-z_2 + 1) \geq 2 -2 \delta \geq 1$. For $|\Im(s-z_i+1 )| \le t+1$, $i=1,2$, by Lemma \ref{Stirling}(iii), it follows that 
\begin{align*}
\frac{\Gamma(z_2-s-\frac{1}{2}+\mi t ) }{ \Gamma(s-z_1+\frac{3}{2}+\mi t )  } 
& = \frac{\Gamma(\frac{1}{2} - (s-z_2 + 1) + \mi t) }{ \Gamma(\frac{1}{2}  + (s - z_1 + 1) + \mi t)  } \\
&\ll \frac{(\Im(s - z_1 + 1))^2+(\Im(s-z_2 + 1))^2 }{t^{2}}e^{\frac{\pi}{2}\left(\Im(2s-z_1-z_2)\right)},
\end{align*}
and 
\begin{align*}
\frac{\Gamma(z_1-s-\frac{1}{2}-\mi t) }{\Gamma(s-z_2+\frac{3}{2}-\mi t) }
&= \frac{\Gamma(\frac{1}{2} - (s-z_1 + 1) - \mi t) }{ \Gamma(\frac{1}{2}  + (s - z_2 + 1) - \mi t)  } \\
&\ll \frac{(\Im(s - z_1 + 1))^2+(\Im(s-z_2 + 1))^2 }{t^{2}}e^{-\frac{\pi}{2}\left(\Im(2s-z_1-z_2)\right)}.
\end{align*}
Therefore, we have
\begin{align}   \label{M-gammastirlingB}
 \begin{split}
   \frac{\Gamma(z_2-s-\frac{1}{2}+\mi t ) }{ \Gamma(s-z_1+\frac{3}{2}+\mi t )  } 
   + \frac{\Gamma(z_1-s-\frac{1}{2}-\mi t) }{\Gamma(s-z_2+\frac{3}{2}-\mi t) }
    \ll 
    \frac{(\Im(s - z_1 + 1))^2+(\Im(s-z_2 + 1))^2 }{t^{2}}e^{\frac{\pi}{2}|\Im(2s-z_1-z_2)|}.
      \end{split}
\end{align}
On the other hand, assuming  either $|\Im(s-z_1+1 )| > t+1$ or $|\Im(s-z_2 )+1| > t+1$, by Lemma \ref{Stirling}(ii), we also have
\begin{align}
\label{M-gammabound}
 \begin{split}
&\frac{\Gamma(z_2-s-\frac{1}{2}+\mi t ) }{ \Gamma(s-z_1+\frac{3}{2}+\mi t )  } 
 +  \frac{\Gamma(z_1-s-\frac{1}{2}-\mi t) }{\Gamma(s-z_2+\frac{3}{2}-\mi t) } \\
&  \ll \frac{(\Im(s-z_1+1))^2 + (\Im(s-z_2+1))^2 }{t^2} e^{\frac{\pi}{2} |\Im(2s-z_1 -z_2 +2)|}.
  \end{split}
\end{align}

By Stirling's formula, for $-4 \le x \le 4$ and $|y| \ge \frac{1}{2}$, we have
\begin{equation}
  \label{M-gammazbd}
|\Gamma(x+iy)| \asymp |y|^{x-\frac{1}{2}} e^{-\frac{\pi}{2}|y|},
\end{equation}
which gives
\begin{equation}
  \label{M-gammaoneline}
 |\Gamma(-z_1 -z_2+2s+2)| \ll  (|2s- z_1 -z_2| +1)^{4- \Re(z_1) -\Re(z_2)-\frac{1}{2} } e^{-\frac{\pi}{2} |\Im(2s- z_1 -z_2)|}  .
\end{equation}
In the above, we have used the trivial bound that $|\Gamma (-z_1 -z_2 + 2s +2)| \ll 1$ for $\Im (2s -z_1- z_2 + 2) \leq 2$. Therefore, combining \eqref{M-gammastirlingB}, \eqref{M-gammabound} and \eqref{M-gammaoneline},   for $\Re(s)=1$, we have
\begin{align*} 
 \begin{split}
&\Gamma( -z_1-z_2+2s+2) \left(  \frac{\Gamma(z_2-s-\frac{1}{2}+\mi t ) }{ \Gamma(s-z_1+\frac{3}{2}+\mi t )  }  +  \frac{\Gamma(z_1-s-\frac{1}{2}-\mi t) }{\Gamma(s-z_2+\frac{3}{2}-\mi t) } \right) \\
&\ll 
\frac{|\Im(s-z_1+1)|^2 + |\Im(s-z_2+1)|^2 }{t^2}
(|2s- z_1 -z_2| +1 )^{4- \Re(z_1) -\Re(z_2)-\frac{1}{2} }.
  \end{split}
\end{align*}
Then we see that the $s$-integral in \eqref{I111b} is bounded  by
\begin{align*}
&\ll \left(\frac{U}{t} \right)^4 \int_{(1)} \frac{|G(s)|}{|s|} 
  \left( \frac{t}{2} \right)^{4}   \left(1 + O \left( \frac{|s|^2+1}{t}  \right)\right)
 t^{-2} (|s|+|z_1| + |z_2|+ 1 )^{\frac{7}{2} +  2\delta} r^{-2 }  \,|ds|\\
&\ll r^{-2} t^{2} \int_{(1)} |G(s)| |s|^{10}
\,|ds| \\
&\ll  r^{-2 } t^{2}  . 
\end{align*}
Hence, we conclude that the removal of the condition $r \le R_0=T^{5}$ adds to the right of 
\eqref{I111b} an error at most 
\begin{align*} 
\begin{split}
 \sum_{r > R_0} 
  \int_{c_1 T}^{c_2 T} O(r^{-2+2\delta} t^{2} ) \,dt  \left|\sum_{q=1}^{\infty} \frac{c_q(r) G_4 (z_1,q) G_4(z_2,q)}{q^{z_1 + z_2} } \right|
  & \ll T^3  \sum_{r > T^{5}} r^{-2+2\delta} \tau_2(r)\\
  & \ll T^3  T^{-5 + 15 \delta}\\
  & \ll 1.
\end{split}
\end{align*} 
This  means that on the right of  \eqref{I111b}, we can omit the condition $r \le R_0$. 
Now, by the absolute convergence, we may swap the order of summation and integration there to obtain 
\begin{equation}
\label{IONO}
\begin{split}
 I_O  & = \frac{2}{(2\pi \mi)^3} \int_{\mathcal{B}_1} \int_{\mathcal{B}_2}  \zeta^4 (z_1) \zeta^4 (z_2) \int_{-\infty}^{\infty}\omega(t)  \int_{(1)} 
\frac{G(s)}{s  \pi^{4s} }   
 g(s,t)  \left( \frac{U}{t} \right)^{4s}  \HH_s(z_1,z_2) \\ 
 & 
\times \Gamma( -z_1-z_2+2s+2)\left( 
\frac{\Gamma(z_2-s-\frac{1}{2}+\mi t ) }{ \Gamma(s-z_1+\frac{3}{2}+\mi t )  } 
+ 
\frac{\Gamma(z_1-s-\frac{1}{2}-\mi t) }{\Gamma(s-z_2+\frac{3}{2}-\mi t) }
   \right)  \, dt \, ds\,   d {z_2} \, d {z_1} \\
   &+ O(1 ) + O \left( \left( \frac{T}{T_0} \right)^{1+C} T^{1-\frac{\e}{2}} \right),
\end{split}
\end{equation} 
where 
\begin{equation*}
   \HH_s(z_1,z_2)=  \sum_{r=1}^{\infty}
\sum_{q=1}^{\infty} \frac{c_q(r) G_4 (z_1,q) G_4(z_2,q)}{q^{z_1 + z_2} r^{2s+2-z_1-z_2}}.
\end{equation*}
It is convenient to write 
\begin{equation}
  \label{Hidentity}
  \HH_s(z_1,z_2) = \mathscr{H}(z_1-1,z_2-1,z_1+z_2-2,s)
\end{equation}
where 
\begin{equation}
  \label{curlyH}
 \mathscr{H}(u_1,u_2,u_3,s):=
     \sum_{r=1}^{\infty}
\sum_{q=1}^{\infty} \frac{c_q(r) G_4 (1+u_1,q) G_4(1+u_2,q)}{q^{2+u_3} r^{2s-u_3}} .
\end{equation}

\subsection{Handling $\HH_s(z_1,z_2)$}
We begin with the following proposition connecting $ \mathscr{H}(u_1,u_2,u_3,s)$ and $\zeta(s)$.

\begin{prop} \label{Hlemma}
Let $\delta' \in (0,\frac{1}{5})$.  For $i=1,2,3$, we assume
\begin{equation} \label{cond-ui}
|u_i| < \delta'.
\end{equation}
(i) For $\Re(s)>\frac{1}{2} + 2\delta'$, We have  
\begin{equation*}
     \mathscr{H}(u_1,u_2,u_3,s) = \zeta(2s-u_3) \frac{\zeta(1+2s)^{16} \zeta(1+2s-u_1-u_2) }{ \zeta(1+2s-u_1)^4 \zeta(1+2s-u_2)^4 }
        \mathscr{I}(u_1,u_2,u_3,s),
\end{equation*}
where   
\begin{equation}
\mathscr{I}(u_1,u_2,u_3,s)
= \prod_{p}  \mathscr{I}_p(u_1,u_2,u_3,s),
\label{def-scrI}
\end{equation}
and $\mathscr{I}_p(u_1,u_2,u_3,s)$ is defined in \eqref{Ipdefn} below. \\
(ii) For $\Re(s) = \sigma \ge -\frac{1}{2} + 2 \delta'$, we have  
\begin{equation*}
     \mathscr{I}_p(u_1,u_2,u_3,s)=1 -  Y_p p^{-2-4s} + O( p^{7 \delta' + \vartheta(\sigma)})
\end{equation*}
where 
\begin{equation}
 \label{Yp}
 \begin{split}
 Y_p &=  -6X_2^{-2} - 6X_1^{-2} - 16(X_1 X_2)^{-1} + 4(X_2^{-2} X_1^{-1}) + 4(X_1^{-2} X_2^{-1}) - (X_1 X_2)^{-2}\\
 & - 36 + 24 X_2^{-1} + 24 X_1^{-1}, \\
 X_i &=   p^{-u_i}, \text{ for } i=1,2,3,
\end{split}
\end{equation}
and
\begin{equation*}
  \vartheta(\sigma) = 
  \begin{cases}
  -2 & \text{ if } \sigma \ge 0; \\
  -2-2\sigma & \text{ if } 0 > \sigma \ge -\frac{1}{4}; \\
  -3-6 \sigma & \text{ if } -\frac{1}{4} > \sigma \geq - \frac{1}{2}.
  \end{cases}
\end{equation*}

Moreover, it follows that $  \mathscr{I}(u_1,u_2,u_3,s)$ is holomorphic in $\Re(s) > -\frac{1}{4}+\delta'$, and 
$ \mathscr{H}(u_1,u_2,u_3,s)$ is meromorphic in $\Re(s) > -\frac{1}{4}+\delta'$ with the exception of poles at 
\begin{equation}
 \label{Hpoles}
 s= \frac{1+u_3}{2},\, s=0, \, s= \frac{u_1+u_2}{2}, \,  s =\frac{\rho-1+u_1}{2},  \, s =\frac{\rho-1+u_2}{2}, 
\end{equation} 
where $\rho$ ranges through the non-trivial zeros of $\zeta(s)$. 
\end{prop}

By \eqref{Hidentity}, we see 
\begin{equation}
   \HH_s(z_1,z_2) = \zeta(2s+2-z_1-z_2) \frac{\zeta(1+2s)^{16} \zeta(1+2s-z_1 -z_2 + 2) }{ \zeta(1+2s - z_1 + 1)^4 \zeta(1+2s- z_2 +1)^4 }
       \tilde{ \mathscr{I}}(z_1,z_2,s),
       \label{HHs1}
\end{equation} 
where 
\[
\tilde{ \mathscr{I}}(z_1,z_2,s) :=  \mathscr{I}(z_1-1,z_2-1,z_1 + z_2 -2, s)
\]
We see that $  \HH_s(z_1,z_2)$ has an analytic continuation to $\Re(s) \ge -\frac{1}{4} +\eta'$, where $\eta'=\eta'(\delta)$, with the exception of poles at 
$$
 s= \frac{z_1+z_2-1}{2}, \, 0, \, -1+\frac{z_1+z_2}{2}, \, \frac{\rho+z_1-2}{2}, \, \frac{\rho+z_2-2}{2}
$$
where $\rho$ ranges through non-trivial zeros of $\zeta(s)$. We also see in \eqref{IONO} that 
$\Gamma( -z_1-z_2+2s+2)$ has a pole at 
$$
    s= \frac{z_1+z_2-2}{2},
$$
and 
$\frac{\Gamma(z_2-s-\frac{1}{2}+\mi t ) }{ \Gamma(s-z_1+\frac{3}{2}+\mi t )  } 
+ 
\frac{\Gamma(z_1-s-\frac{1}{2}-\mi t) }{\Gamma(s-z_2+\frac{3}{2}-\mi t) }$ has poles at 
\begin{equation*}
 s=\p_1,\, \p_2. 
\end{equation*}

\subsection{Moving the integral in \eqref{IONO} left to $\Re(s) = \e_3$}\label{e3-line}
Recall \eqref{IONO}. 
We first observe that $ \HH_s(z_1,z_2) = \mathscr{H}(1-z_1,1-z_2,z_1+z_2-2,s)$ and move the contour to $\Re(s) =\e_3$ with  $\e_3 >\delta>0$, which is large enough to miss poles at $s=0$ and $s= -1 + \frac{z_1+z_2}{2}$ of $ \HH_s(z_1,z_2)$ and the pole at $ s= \frac{z_1+z_2-2}{2}$ of $\Gamma( -z_1-z_2+2s+2)$. For the pole at $s= \frac{z_1+z_2-1}{2}$ of  $  \HH_s(z_1,z_2)$, we note that $ 
\frac{\Gamma(z_2-s-\frac{1}{2}+\mi t ) }{ \Gamma(s-z_1+\frac{3}{2}+\mi t )  } 
+ 
\frac{\Gamma(z_1-s-\frac{1}{2}-\mi t) }{\Gamma(s-z_2+\frac{3}{2}-\mi t) }$
 has a zero at the same point. Also, the contribution from the pole at $s=\p_i$ in the $s$-integral of  \eqref{IONO}  is
\[
  -\frac{G(\p_i)}{\p_i} g(\p_i,t)    \left( \frac{U}{t} \right)^{4\p_i} 
   \HH_{\p_i} (z_1,z_2)
  \pi^{-4 \p_i} 
  \ll |\p_i|^{-1-B}   U ^{4\Re(\p_j)}  \frac{|\p_i|^{2}}{t} 
 \ll T^{3-B}.  
\]
It then follows from the above bound and \eqref{IONO}  that 
\begin{equation*}
\begin{split}
   I_O &= \frac{2}{(2\pi \mi)^3} \int_{\mathcal{B}_1} \int_{\mathcal{B}_2}  \zeta^4 (z_1) \zeta^4 (z_2) \int_{-\infty}^{\infty}\omega(t)  \int_{(\e_3)} 
\frac{G(s)}{s  \pi^{4s} }   
 g(s,t)  \left( \frac{U}{t} \right)^{4s}  \HH_s(z_1,z_2) \\ 
 & 
\times \Gamma( -z_1-z_2+2s+2)\left( 
\frac{\Gamma(z_2-s-\frac{1}{2}+\mi t ) }{ \Gamma(s-z_1+\frac{3}{2}+\mi t )  } 
+ 
\frac{\Gamma(z_1-s-\frac{1}{2}-\mi t) }{\Gamma(s-z_2+\frac{3}{2}-\mi t) }
   \right)  \, dt \, ds\,  d {z_2} \, d {z_1} \\
   &+  O(1 ) + O \left( \left( \frac{T}{T_0} \right)^{1+C} T^{1-\frac{\e}{2}} \right).
\end{split}
\end{equation*}  
By Lemma \ref{lem:bd-g}(i), we know 
\begin{equation}
\begin{split}
  \label{IOetatilde01}
   I_O &= \frac{2}{(2\pi \mi)^3} \int_{\mathcal{B}_1} \int_{\mathcal{B}_2}  \zeta^4 (z_1) \zeta^4 (z_2) \int_{-\infty}^{\infty}\omega(t)  \int_{(\e_3)} 
\frac{G(s)}{s  \pi^{4s} }   
 \left( \frac{t}{2} \right)^{4s}   \left( \frac{U}{t} \right)^{4s}  \HH_s(z_1,z_2) \\ 
 & 
\times \Gamma( -z_1-z_2+2s+2)\left( 
\frac{\Gamma(z_2-s-\frac{1}{2}+\mi t ) }{ \Gamma(s-z_1+\frac{3}{2}+\mi t )  } 
+ 
\frac{\Gamma(z_1-s-\frac{1}{2}-\mi t) }{\Gamma(s-z_2+\frac{3}{2}-\mi t) }
   \right)  \, dt \, ds\,   d {z_2} \, d {z_1}  \\
     &+ I_{O,E} +  O(1 ) + O \left( \left( \frac{T}{T_0} \right)^{1+C} T^{1-\frac{\e}{2}} \right),
\end{split}
\end{equation}  
where 
\begin{equation}
\begin{split}
  \label{IOetatildeE01}
   I_{O,E} &= \frac{2}{(2\pi \mi)^3} \int_{\mathcal{B}_1} \int_{\mathcal{B}_2}  \zeta^4 (z_1) \zeta^4 (z_2) \int_{-\infty}^{\infty}\omega(t)  \int_{(\e_3)} 
\frac{G(s)}{s  \pi^{4s} }   
 \left( \frac{t}{2} \right)^{4s} O \left( \frac{|s|^2+1}{t}  \right)   \left( \frac{U}{t} \right)^{4s}  \HH_s(z_1,z_2) \\ 
 & 
\times \Gamma( -z_1-z_2+2s+2)\left( 
\frac{\Gamma(z_2-s-\frac{1}{2}+\mi t ) }{ \Gamma(s-z_1+\frac{3}{2}+\mi t )  } 
+ 
\frac{\Gamma(z_1-s-\frac{1}{2}-\mi t) }{\Gamma(s-z_2+\frac{3}{2}-\mi t) }
   \right)  \, dt \, ds\,   d {z_2} \, d {z_1}. 
    \end{split}
\end{equation}  
We shall give an upper bound for $I_{O,E}$. Observing that for $0<\delta < \e_3<0.15$, then $\Re(s-z_i + 1) \in [\e_3 - \delta, \e_3 + \delta] \subseteq [0,0.4]$, $i=1,2$. It implies 
$\Re(s-z_1 + 1 + s-z_2 + 1) \leq 0.8$.  We remind the reader that in \eqref{IOetatildeE01},  $t$ does not denote the imaginary part of $s$ but a variable satisfying $c_1 T \leq t \leq c_2 T$. It follows from   Lemma \ref{Stirling}(ii) and \eqref{M-gammazbd} that the inner integral in \eqref{IOetatildeE01} on $|\Im(s-z_1 + 1)| \geq t+1$ or $|\Im(s-z_2+ 1)| \geq t+1$ is  
\begin{equation}
  \begin{split}
 \label{IOetatildeE02}
 & \ll \int_{\substack{\Re(s) = \e_3 \\   |\Im(s-z_1 + 1)| \geq t+1, \\ \text{or }  |\Im(s-z_2+ 1)| \geq t+1}  }    
 \frac{|G(s)|}{|s|} 
  O \left(U^{4 \e_3} \frac{|s|^2}{t}  \right) |\HH_s(z_1,z_2)|
 (|2s- z_1 -z_2|+1)^{2} 
 e^{-\frac{\pi}{2} |\Im(2s- z_1 -z_2)|}\\
 &\times \frac{|\Im(s-z_1+1)|^2 + |\Im(s-z_2+1)|^2 }{t^2} e^{\frac{\pi}{2} |\Im(2s-z_1 -z_2 +2)|} \,d|s| \\
 &\ll  U^{4 \e_3} 
 \int_{\substack{\Re(s) = \e_3 \\   |\Im(s-z_1 + 1)| \geq t+1, \\ \text{or }  |\Im(s-z_2+ 1)| \geq t+1}  } 
  |G(s)| |s|^5
  |\HH_s(z_1,z_2)| \,d|s|\\
  &\ll  T^{(1 - \varepsilon)4 \e_3} 
  \int_{\substack{\Re(s) = \e_3 \\   |\Im(s-z_1 + 1)| \geq t+1, \\ \text{or }  |\Im(s-z_2+ 1)| \geq t+1}  } 
   |s|^{-10} \,d|s|\\
  &\ll T^{-5}.
  \end{split}
  \end{equation}
In the second last inequality above, we used the upper bound that for $\Re(s) = \e_3$ and $0<\delta < \e_3$,
  \[
   |\HH_s(z_1,z_2)|  \ll |\zeta(2s + 2 -z_1 -z_2)| \frac{1}{\e_3^{16}} \frac{1}{2\e_3 - 2\delta} 
   \frac{1}{(2 \e_3 - \delta)^4} \frac{1}{(2 \e_3 - \delta)^4} \ll |s|+1,
  \]
   which can be deduced from \eqref{HHs1} and Proposition \ref{Hlemma}.
By Lemma \ref{Stirling}(i) and the bound $\Re(s-z_1 + 1 + s-z_2 + 1) \leq 0.8$,  the inner integral on $|\Im(s-z_i + 1)| < t+1,\, i=1,2$ in \eqref{IOetatildeE01} is
\begin{equation}
  \label{IOetatildeE03}
 \begin{split}
 &\ll \int_{\substack{\Re(s) = \e_3 \\   |\Im(s-z_i + 1)| < t+1,\, i=1,2}  }   \frac{|G(s)|}{|s|} 
 O \left(U^{4 \e_3}  \frac{|s|^2+1 }{t}  \right) |\HH_s(z_1,z_2)|
 ( |2s- z_1 -z_2|+1)^{2} 
   e^{-\frac{\pi}{2} |\Im(2s- z_1 -z_2)|}\\
 &\times \left( t^{-2s+ z_1 + z_2 - 2} \cos 
  ( \tfrac{\pi}{2} (2s -z_1 -z_2 +2)) 
  +O \left(  t^{-2\e_3+ 2\delta}   e^{\frac{\pi}{2} |\Im(2s-z_1-z_2)|} \tfrac{1 +|s-z_1 + 1|^2  +|s-z_2 + 1|^2 }{t} \right)\right) \, d|s|\\
 &\ll \int_{\substack{\Re(s) = \e_3 \\  |\Im(s-z_i + 1)| < t+1,\, i=1,2}  }    \frac{|G(s)|}{|s|} 
  O \left(U^{4 \e_3} \frac{|s|^2+1 }{t}  \right) |\HH_s(z_1,z_2)|
  ( |2s- z_1 -z_2|+1)^{2} 
   e^{-\frac{\pi}{2} |\Im(2s- z_1 -z_2)|}\\
 &\times e^{\frac{\pi}{2} |\Im(2s -z_1 -z_2 )|} t^{-2 \e_3 + 2\delta}  \left( 1 +\tfrac{1 +|s-z_1 + 1|^2  +|s-z_2 + 1|^2 }{t} \right) \, d|s| \\
 &\ll  T^{-1 + 2\e_3+ 2\delta -  4 \e_3\varepsilon }.
   \end{split}
  \end{equation}
  In the second inequality above, we used the fact that for $z \in \mathbb{C}$,
  \[
  |\cos (\tfrac{\pi}{2} z) | \ll e^{\frac{\pi}{2} |\Im (z)|}.
  \]
Thus, substituting \eqref{IOetatildeE02} and  \eqref{IOetatildeE03} in \eqref{IOetatildeE01}, we have 
$ I_{O,E} \ll  T^{ 2\e_3 + 2\delta - 4 \e_3\varepsilon } \ll T^{0.5}$, provided that $0<\e_3 <0.15$ and $0<\delta<\frac{1}{10}$, which, together with  \eqref{IOetatilde01}, yields
\begin{equation}
\begin{split}
   & I_O  \\ &=   
\frac{2}{(2\pi \mi)^2} \int_{\mathcal{B}_1} \int_{\mathcal{B}_2} \zeta^4 (z_1) \zeta^4 (z_2)   \frac{1}{2 \pi \mi} \int_{(\e_3)} 
\frac{G(s)}{s }  \HH_s(z_1,z_2) \\
&\times    
 \int_{-\infty}^{\infty}
  \left( \frac{U}{2 \pi } \right)^{4s}
 \omega(t)  
\left( \frac{\Gamma(z_2-s-\frac{1}{2}+\mi t ) }{ \Gamma(s-z_1+\frac{3}{2}+\mi t )  } 
+ \frac{\Gamma(z_1-s-\frac{1}{2}-\mi t)}{\Gamma(s-z_2+\frac{3}{2}-\mi t) }
   \right)   \Gamma( -z_1-z_2+2s+2)  \, dt \, ds\,   d {z_2} \, d {z_1} \\
   &+ O( T^{ 0.5 })  + O \left( \left( \frac{T}{T_0} \right)^{1+C} T^{1-\frac{\e}{2}} \right).
   \end{split}
   \label{IO_Mid}
\end{equation}

\section{Further evaluation of off-diagonal terms: Proof of Proposition \ref{offdiagonal}}\label{feval-off} 
In this section, we shall give a further simplification of $I_O$, and extract one main term from there.
By a rearrangement for \eqref{IO_Mid}, we write
\begin{align}
\label{IO-Mid01}
 I_O =  \int_{-\infty}^{\infty} \omega(t) 
 \frac{2}{2 \pi \mi} \int_{(\e_3)} 
\frac{G(s)}{s }  \left( \frac{U}{2 \pi } \right)^{4s} \mathscr{I}(s,t) \, ds \, dt + O( T^{ 0.5 })  + O \left( \left( \frac{T}{T_0} \right)^{1+C} T^{1-\frac{\e}{2}} \right),
\end{align}
where 
\begin{equation*}
\begin{split}
   \mathscr{J}(s,t) & :=
\frac{1}{(2\pi \mi)^2} \int_{\mathcal{B}_1} \int_{\mathcal{B}_2} \zeta^4 (z_1) \zeta^4 (z_2) \HH_s(z_1,z_2)
\left( 
\frac{\Gamma(z_2-s-\frac{1}{2}+\mi t ) }{ \Gamma(s-z_1+\frac{3}{2}+\mi t )  } 
+ 
\frac{\Gamma(z_1-s-\frac{1}{2}-\mi t)}{\Gamma(s-z_2+\frac{3}{2}-\mi t) }
   \right) \\
  & \times \Gamma( -z_1-z_2+2s+2)   \, d {z_2} \, d {z_1}  \\
  & = \frac{1}{2 \pi \mi } \int_{\mathcal{B}_2} \zeta^4(z_2) \times \text{inn}(z_2) \, dz_2 ,
\end{split}
\end{equation*}
\begin{align*}
  \text{inn}(z_2)   & := 
\frac{1}{2\pi \mi} \int_{\mathcal{B}_1}   \zeta^4 (z_1)   \mathcal{F}_{s,t}(z_1,z_2)
   \, d {z_1}
   = \frac{1}{2\pi \mi} \int_{\mathcal{B}_1}  \frac{1}{(z_1-1)^4}  ( (z_1-1) \zeta(z_1))^4   \mathcal{F}_{s,t}(z_1,z_2) \, dz_1,
\end{align*}
and
\[
   \mathcal{F}_{s,t}(z_1,z_2) := \HH_s(z_1,z_2)
\left( 
\frac{\Gamma(z_2-s-\frac{1}{2}+\mi t ) }{ \Gamma(s-z_1+\frac{3}{2}+\mi t )  } 
+ 
\frac{\Gamma(z_1-s-\frac{1}{2}-\mi t)}{\Gamma(s-z_2+\frac{3}{2}-\mi t) }
   \right)   \Gamma( -z_1-z_2+2s+2) .
\]
Setting $H(z_1) = ( (z_1-1) \zeta(z_1))^4 \mathcal{F}_{s,t}(z_1,z_2)$, we have 
\begin{align*}
     \frac{1}{2\pi \mi} \int_{\mathcal{B}_1}  \frac{1}{(z_1-1)^4}  H(z_1) \, dz_1 = \frac{H^{(3)}(1)}{3!}
     =  \sum_{i'+i=3}  \frac{b_{i'} \mathcal{F}_{s,t}^{(i,0)}(1,z_2)}{i!},
\end{align*} 
where $b_{i'} \in \R$ satisfy
\begin{equation*}
h(z) := ((z-1) \zeta(z))^4  = 1 + b_{1}(z-1) + b_{2} (z-1)^{2} + \cdots.
\end{equation*}  
Here 
\[
\mathcal{F}_{s,t}^{(i,j)}(z_1,z_2) :=  \frac{\partial^{i}}{\partial z_1^i}
       \frac{\partial^{j}}{\partial z_2^j} \mathcal{F}_{s,t}(z_1,z_2).
\]
We now compute $  \mathcal{F}_{s,t}^{(i,0)}(z_1,z_2)$ for $i \in \{0, 1,2,3\}$. By the generalized product rule, we obtain
\begin{equation}
\begin{split}
  \label{Fi01z2}
  &  \mathcal{F}_{s,t}^{(i,0)}(1,z_2) \\ & = 
    \sum_{u+v+w=i}  \binom{i}{u,v,w} 
    \HH_s^{(u,0)}(z_1,z_2) \Big|_{z_1=1}
       \frac{\partial^{v}}{\partial z_1^v}
    \left.   \left( 
\frac{\Gamma(z_2-s-\frac{1}{2}+\mi t ) }{ \Gamma(s-z_1+\frac{3}{2}+\mi t )  } 
+ 
\frac{\Gamma(z_1-s-\frac{1}{2}-\mi t)}{\Gamma(s-z_2+\frac{3}{2}-\mi t) }
   \right)   \right|_{z_1=1}  \\
   & \times (-1)^w  \Gamma^{(w)}( -z_2+2s+1) .
\end{split}
\end{equation}
Observe that by Proposition  \ref{Hlemma}(i),  for $\Re(s) > \frac{1}{2} + 4\delta $ and $u \in \mathbb{Z}_{\ge 0}$, we have
\begin{equation}
   \label{Hu0}
      \HH_s^{(u,0)}(z_1,z_2) \Big|_{z_1=1}
       = \sum_{i_1+i_2=u} \binom{u}{i_1}  
       \left(
       \sum_{r=1}^{\infty}
\sum_{q=1}^{\infty} \frac{c_q(r) G_4^{(i_1)} (1,q) G_4(z_2,q)}{q^{1 + z_2} r^{2s+1-z_2}} (\log(\tfrac{r}{q}))^{i_2}
\right).
\end{equation}
Hence, we have
\begin{align*}
&  \mathscr{J}(s,t) \\& = \frac{1}{2 \pi \mi}\int_{\mathcal{B}_2} \zeta(z_2)^4 \left(  \frac{1}{2\pi \mi} \int_{\mathcal{B}_1}  \frac{1}{(z_1-1)^4}  H(z_1) \, dz_1   \right) \, dz_2 
    =\frac{1}{2 \pi \mi}\int_{\mathcal{B}_2} \zeta(z_2)^4   \sum_{i'+i=3}  \frac{b_{i'} \mathcal{F}_{s,t}^{(i,0)}(1,z_2)}{i!} \, dz_2 \\
   & =  \sum_{i'+i=3}  \frac{b_{i'}}{i!}  \frac{1}{2 \pi \mi}  \int_{\mathcal{B}_2}   \frac{1}{(z_2-1)^4}  ( (z_2-1) \zeta(z_2))^4  \mathcal{F}_{s,t}^{(i,0)}(1,z_2) \, dz_2 
    =  \sum_{i'+i=3}  \frac{b_{i'}}{i!} 
   \sum_{j'+j=3}   \frac{b_{j'}}{j!} 
   \mathcal{F}_{s,t}^{(i,j)}(1,1).
\end{align*}
 From  \eqref{Fi01z2}, it follows that
\begin{equation*}
\begin{split}
    \mathcal{F}_{s,t}^{(i,j)}(1,1)  = 
    \sum_{u+v+w=i} & \binom{i}{u,v,w} 
    \sum_{a+b+c=j}
    \binom{j}{a,b,c}
    \HH_s^{(u,a)}(1,1)  
      \mathcal{G}_{s,t}^{(v,b)}(1,1)
     (-1)^{w+c}  \Gamma^{(w+c)}(2s), 
\end{split}
\end{equation*}
where 
\begin{equation}
  \label{Gbvs}
   \mathcal{G}_{s,t}^{(v,b)}(1,1)=  \frac{\partial^{v}}{\partial z_1^v}
   \frac{\partial^{b}}{\partial z_2^b}
  \left.  \left( 
\frac{\Gamma(z_2-s-\frac{1}{2}+\mi t ) }{ \Gamma(s-z_1+\frac{3}{2}+\mi t )  } 
+ 
\frac{\Gamma(z_1-s-\frac{1}{2}-\mi t)}{\Gamma(s-z_2+\frac{3}{2}-\mi t) }
   \right)   \right|_{z_1=z_2=1}.
\end{equation}
By \eqref{Hu0}, we also know for $\Re(s) > \frac{1}{2} + 4\delta $, 
\begin{align}
\begin{split}
   \label{Hua}
      \HH_s^{(u,a)}(1,1) 
& = \sum_{i_1+i_2=u} \binom{u}{i_1}  \sum_{j_1+j_2=a} \binom{a}{j_1}  
       \left(
       \sum_{r=1}^{\infty}
\sum_{q=1}^{\infty} \frac{c_q(r) G_4^{(i_1)} (1,q) G_4^{(j_1)}(1,q)}{q^{2} r^{2s}} 
(\log(\tfrac{r}{q}))^{i_2+j_2}
\right) \\
 & = \sum_{i_1+i_2=u} \binom{u}{i_1}   \sum_{j_1+j_2=a} \binom{a}{j_1}  
         \mathscr{H}^{(i_1,j_1,i_2+j_2,0)}(0,0,0,s),  
\end{split}
\end{align}
where $\mathscr{H}$ is defined as in \eqref{curlyH}.  Note that the above identity is also valid in the region $\Re(s) > -\frac{1}{4} +2\delta $ since $\HH_s^{(u,a)}(1,1) $ can be meromorphically   extended to $\Re(s) >- \frac{1}{4} +2\delta $ by Proposition \ref{Hlemma} and the meromorphic continuation is unique.
We then arrive at
\begin{align*}
 & \mathscr{J}(s,t) \\
 &   =   \sum_{i'+i=3}  \frac{b_{i'}}{i!} 
   \sum_{j'+j=3}   \frac{b_{j'}}{j!} 
   \mathcal{F}_{s,t}^{(i,j)}(1,1)  \\
   & =   \sum_{i'+i=3}  \frac{b_{i'}}{i!} 
   \sum_{j'+j=3}   \frac{b_{j'}}{j!} \sum_{u+v+w=i}  \binom{i}{u,v,w} 
    \sum_{a+b+c=j}
    \binom{j}{a,b,c}
    \HH_s^{(u,a)}(1,1)  
       \mathcal{G}_{s,t}^{(v,b)}(1,1)
     (-1)^{w+c}  \Gamma^{(w+c)}(2s) .
\end{align*}
Thus, combined with \eqref{IO-Mid01}, it follows that
\begin{align}   \label{I-O-g}
 \begin{split}
 I_O &=  \int_{-\infty}^{\infty} \omega(t) 
 \frac{2}{2 \pi \mi} \int_{(\e_3)} 
\frac{G(s)}{s }  \left( \frac{U}{2 \pi } \right)^{4s}  \sum_{i'+i=3}  \frac{b_{i'}}{i!} 
   \sum_{j'+j=3}   \frac{b_{j'}}{j!} \sum_{u+v+w=i}  \binom{i}{u,v,w} 
    \sum_{a+b+c=j}
    \binom{j}{a,b,c}\\
  & \times \left(  
     \HH_s^{(u,a)}(1,1)  
      \mathcal{G}_{s,t}^{(v,b)}(1,1)
     (-1)^{w+c}  \Gamma^{(w+c)}(2s) 
\right)
 \, ds \, dt  +  O( T^{ 0.5 }) +O \left( \left( \frac{T}{T_0} \right)^{1+C} T^{1-\frac{\e}{2}} \right)\\
 & = 2 \sum_{i'+i=3}  \frac{b_{i'}}{i!} 
   \sum_{j'+j=3}   \frac{b_{j'}}{j!} \sum_{u+v+w=i}  \binom{i}{u,v,w} 
    \sum_{a+b+c=j}
    \binom{j}{a,b,c} \\
   &  \times   \int_{-\infty}^{\infty} \omega(t) 
 \frac{1}{2 \pi \mi} \int_{(\e_3)} 
\frac{G(s)}{s }  \left( \frac{U}{2 \pi } \right)^{4s}  \HH_s^{(u,a)}(1,1)  
 \mathcal{G}_{s,t}^{(v,b)}(1,1)
     (-1)^{w+c}  \Gamma^{(w+c)}(2s)  \, ds \, dt   \\
     & + O( T^{ 0.5 }) + O \left( \left( \frac{T}{T_0} \right)^{1+C} T^{1-\frac{\e}{2}} \right).
  \end{split}
\end{align}
The last $s$-integral can be further evaluated by the following lemma.
\begin{lem}
\label{lem:iX}
Let $0<\delta < \frac{1}{10}$ and $0<\delta<\e_3<0.15$. Let $u,a,v,b,w,c$ be non-negative integers such that $a+b+c \leq 3$ and $u+v+w \leq 3$.  For ${\bf x} = (u,a,v,b,w,c)$, we set 
\begin{equation*}
   i_{{\bf x}} := \frac{1}{2 \pi \mi} \int_{(\e_3)} 
\frac{G(s)}{s }  \left( \frac{U}{2 \pi } \right)^{4s}  \HH_s^{(u,a)}(1,1)  
 \mathcal{G}_{s,t}^{(v,b)}(1,1)
     (-1)^{w+c}  \Gamma^{(w+c)}(2s)  \, ds .
\end{equation*}
Then we have 
\begin{align}
\label{iden-ix}
 i_{{\bf x}} = i_{{\bf x},0}  + i_{{\bf x},1}  + O\left( \left( \frac{U^2}{t} \right)^{-\frac{1}{2} + 4\delta} \right),
\end{align}
where
\begin{equation}
\label{i0}
\begin{split}
  & i_{{\bf x},0} \\
   &:=  \sum_{i_1+i_2=u} \binom{u}{i_1}  \sum_{j_1+j_2=a} \binom{a}{j_1}  
     \sum_{x_0 =0}^{i_2 + j_2} \sum_{\substack{x_1 + x_2 + x_3 =i_1 \\x_1,x_2,x_3  \geq 0}}   \sum_{\substack{y_1 + y_2 + y_3 =j_1 \\y_1,y_2,y_3 \geq 0}}
      {i_2 +j_2 \choose x_0} (-1)^{x_0 + x_1 + y_1}   
      { i_1 \choose x_1,x_2,x_3}  \\
    & \times { j_1 \choose y_1,y_2,y_3}   (-1)^{b+v}
   \sum_{\substack{r_1' \in \{0,\ldots, v\} \\  r_2' \in \{0,\ldots, b\} \\ v+b-r_1'-r_2'  \equiv 0\, (\operatorname{mod} 2) }}  (-1)^{r_1' + r_2'} {v \choose r_1'} {b \choose r_2'} (\log t)^{r_1'+r_2'}   \left(  \mi\frac{\pi}{2}\right)^{v+b-r_1'-r_2'}  \\
 & \times  \frac{1}{2^{x_1 +y_1 + x_2 +y_2 +w+c+10}} \frac{1}{(x_1 +y_1 + x_2 +y_2 +w+c+10)!}\\
 &\times \sum_{k = 0}^{x_1 +y_1 + x_2 +y_2 +w+c+10} {x_1 +y_1 + x_2 +y_2 +w+c+10 \choose k} \log^{k} \left(\frac{U^4} {t^2}\right)  J_0^{(x_1 +y_1 + x_2 +y_2 +w+c+10-k)} (0),
 \end{split}
\end{equation}
and
\begin{align*}
&   i_{{\bf x},1}\\ 
&:=  \sum_{i_1+i_2=u} \binom{u}{i_1} \sum_{j_1+j_2=a} \binom{a}{j_1}  
     \sum_{x_0 =0}^{i_2 + j_2} \sum_{\substack{x_1 + x_2 + x_3 =i_1 \\x_1,x_2,x_3  \geq 0}}   \sum_{\substack{y_1 + y_2 + y_3 =j_1 \\y_1,y_2,y_3 \geq 0}}
      {i_2 +j_2 \choose x_0} (-1)^{x_0+x_1 + y_1}   
      { i_1 \choose x_1,x_2,x_3}   \\
    &  \times {j_1 \choose y_1,y_2,y_3}   (-1)^{b+v}
   \sum_{\substack{r_1' \in \{0,\ldots, v\} \\  r_2' \in \{0,\ldots, b\} \\ v+b-r_1'-r_2'  \equiv 1\, (\operatorname{mod} 2) }}  (-1)^{r_1' + r_2'} {v \choose r_1'} {b \choose r_2'} (\log t)^{r_1'+r_2'}   \left(  \mi\frac{\pi}{2}\right)^{v+b-r_1'-r_2'}  \\
    &\times \frac{\pi}{2^{x_1 +y_1 + x_2 +y_2 +w+c+10}} \frac{1}{(x_1 +y_1 + x_2 +y_2 +w+c+9)!}\\
    &  \times 
 \sum_{k = 0}^{x_1 +y_1 + x_2 +y_2 +w+c+9} {x_1 +y_1 + x_2 +y_2 +w+c+9 \choose k} \log^{k} \left(\frac{U^4} {t^2}\right)  J_1^{(x_1 +y_1 + x_2 +y_2 +w+c+9-k)} (0).
\end{align*}
Here, $J_0(s):=J_0(s; x_0, x_1, y_1, x_2, y_2, x_3,y_3, i_2,j_2,w,c)$ \\and $J_1(s):=J_1(s; x_0, x_1, y_1, x_2, y_2, x_3,y_3, i_2,j_2,w,c)$ are defined by 
\begin{align*}
J_0(s) &:= G(s) \frac{1}{(2\pi)^{4s}}\zeta(1+2s)^2(2s)^2  g_{x_2}(s) (2s)^{x_2+3} g_{y_2}(s) (2s)^{y_2 + 3} \zeta^{(x_1+y_1)} (1+2s)  (2s)^{x_1 + y_1 +1 } \\
&\times2\cos(\pi s) (-1)^{w+c} \Gamma^{(w+c)}(2s) (2s)^{w+c+1}\mathscr{I}^{(x_3,y_3,i_2+j_2-x_0,0)}(0,0, 0,s) \zeta^{(x_0)} (2s )  ,
\end{align*}
and
\begin{align*}
J_1(s) &:= G(s) \frac{1}{(2\pi)^{4s}}  \zeta(1+2s)^2(2s)^2  g_{x_2}(s) (2s)^{x_2+3} g_{y_2}(s) (2s)^{y_2 + 3} \zeta^{(x_1+y_1)} (1+2s)  (2s)^{x_1 + y_1 +1 }\\
&\times   \frac{2\mi \sin(\pi s)}{\pi s} (-1)^{w+c} \Gamma^{(w+c)}(2s) (2s)^{w+c+1}  \mathscr{I}^{(x_3,y_3,i_2+j_2-x_0,0)}(0,0, 0,s)\zeta^{(x_0)} (2s ) ,
\end{align*}
where $g_{x_2}, g_{y_2}$ are defined in \eqref{equ:g-function-0123}, and  $\mathscr{I}(0,0, 0,s)$ is defined in \eqref{def-scrI}.

Clearly, both $J_0(s)$ and $J_1(s)$ are holomorphic for $  - \frac{1}{4} +2 \delta< \Re(s) <\frac{1}{2}$, and $J_0(s),J_1(s) \ll \frac{1}{1+|\Im (s)|^{100}}$ uniformly in this region.

In addition,
we have
\begin{equation}
i_{{\bf x},1} \ll (\log T)^{15}   .
\label{ix1}
\end{equation}
\end{lem}

\begin{proof}
From Lemmata \ref{FancyG}(ii) and \ref{HHHFinal} in the next section (which give a further computation for $\mathcal{G}_{s,t}^{(v,b)}(1,1)$ and $\HH_s^{(u,a)}(1,1)$), it follows that 
 the $|\Im(s)| > t+1$ part of $ i_{{\bf x}} $ is bounded by 
\begin{equation*}
\begin{split}
&\frac{1}{2 \pi \mi} \int_{\substack{\Re(s) = \e_3\\ |\Im(s)|>t+1}} 
\frac{G(s)}{s }  \left( \frac{U}{2 \pi } \right)^{4s}  \HH_s^{(u,a)}(1,1)  
 \mathcal{G}_{s,t}^{(v,b)}(1,1)
     (-1)^{w+c}  \Gamma^{(w+c)}(2s)  \, ds\\
  & \ll
    \sum_{i_1+i_2=u} \binom{u}{i_1}   \sum_{j_1+j_2=a} \binom{a}{j_1}  
     \sum_{x_0 =0}^{i_2 + j_2} \sum_{\substack{x_1 + x_2 + x_3 =i_1 \\x_1,x_2,x_3 \geq 0}}   \sum_{\substack{y_1 + y_2 + y_3 =j_1 \\y_1,y_2,y_3 \geq 0}}
     {i_2 +j_2 \choose x_0} (-1)^{x_0+x_1 + y_1}   { i_1 \choose x_1,x_2,x_3}  \\
    &\times  { j_1 \choose y_1,y_2,y_3} \frac{1}{2 \pi \mi} \int_{(\e_3)} 
\frac{G(s)}{s }  \left( \frac{U}{2 \pi } \right)^{4s} 
 {\zeta (1+2s) ^{2}}
  \zeta^{(x_0)} (2s )     
      g_{x_2} (s)  g_{y_2}(s)   \zeta^{(x_1+y_1)} ( 1+ 2s)
\\
&\times  \mathscr{I}^{(x_3,y_3,i_2+j_2-x_0,0)}(0,0, 0,s)
 O \left( \frac{\Im(s)^2}{t^2}e^{\pi|\Im(s)|}  \right)
     (-1)^{w+c}  \Gamma^{(w+c)}(2s)  \, ds\\
     &\ll T^{-2},
     \end{split}
\end{equation*}
 due to the rapid decay of $|G(s)|$. 

Therefore, it suffices to consider the $|\Im(s)| \le t+1$ portion of $ i_{{\bf x}} $. By Lemmata \ref{FancyG}(i) and \ref{HHHFinal}, together with the above bound,
we  can write  
\begin{align*}
 i_{{\bf x}}  =  j_{{\bf x},0}  + j_{{\bf x},1} + j_{{\bf x},\mathrm{error}} + O(T^{-2}),
\end{align*}
where 
\begin{align} \label{i-x-0}
 \begin{split}
  & j_{{\bf x},0}\\
    &:=  \sum_{i_1+i_2=u} \binom{u}{i_1} \sum_{j_1+j_2=a} \binom{a}{j_1}  
     \sum_{x_0 =0}^{i_2 + j_2} \sum_{\substack{x_1 + x_2 + x_3 =i_1 \\x_1,x_2,x_3  \geq 0}}   \sum_{\substack{y_1 + y_2 + y_3 =j_1 \\y_1,y_2,y_3 \geq 0}}
      {i_2 +j_2 \choose x_0} (-1)^{x_0+x_1 + y_1}   
       { i_1 \choose x_1,x_2,x_3} \\
    & \times  {j_1 \choose y_1,y_2,y_3} (-1)^{b+v}
   \sum_{\substack{r_1' \in \{0,\ldots, v\} \\  r_2' \in \{0,\ldots, b\} \\ v+b-r_1'-r_2'  \equiv 0\, (\operatorname{mod} 2) }}  (-1)^{r_1' + r_2'} {v \choose r_1'} {b \choose r_2'} (\log t)^{r_1'+r_2'}   \left(  \mi\frac{\pi}{2}\right)^{v+b-r_1'-r_2'}  \\
& \times \frac{1}{2 \pi \mi} \int_{\substack{\Re(s) = \e_3\\ |\Im(s)| \leq t+1}} 
\frac{1}{s }  \left( \frac{U^2}{t } \right)^{2s}    \frac{1}{(2s)^{2 + x_2 +3 +y_2+3 +x_1 +y_1+1+w+c+1}} J_0(s)\,ds,
 \end{split}
\end{align}
\begin{align}\label{i-x-1}
 \begin{split}
  & j_{{\bf x},1}\\
  &:= \sum_{i_1+i_2=u} \binom{u}{i_1}   \sum_{j_1+j_2=a} \binom{a}{j_1}  
     \sum_{x_0 =0}^{i_2 + j_2} \sum_{\substack{x_1 + x_2 + x_3 =i_1 \\x_1,x_2,x_3  \geq 0}}   \sum_{\substack{y_1 + y_2 + y_3 =j_1 \\y_1,y_2,y_3 \geq 0}}
      {i_2 +j_2 \choose x_0} (-1)^{x_0+x_1 + y_1}   
      { i_1 \choose x_1,x_2,x_3}   \\
    &  \times  { j_1 \choose y_1,y_2,y_3}  (-1)^{b+v}
   \sum_{\substack{r_1' \in \{0,\ldots, v\} \\  r_2' \in \{0,\ldots, b\} \\ v+b-r_1'-r_2'  \equiv 1\, (\operatorname{mod} 2) }}  (-1)^{r_1' + r_2'} {v \choose r_1'} {b \choose r_2'} (\log t)^{r_1'+r_2'}   \left(  \mi\frac{\pi}{2}\right)^{v+b-r_1'-r_2'}   \\ 
& \times \frac{1}{2 \pi \mi} \int_{\substack{\Re(s) = \e_3\\ |\Im(s)| \leq t+1}} 
\frac{1}{s }  \left( \frac{U^2}{t} \right)^{2s}  \frac{\pi s}{(2s)^{2 + x_2 +3 +y_2+3 +x_1 +y_1 +1+w+c+1}} J_1(s) \,ds,
 \end{split}
\end{align}
and 
\begin{align*}
&j_{{\bf x},\mathrm{error}}\\
 &:= \sum_{i_1+i_2=u} \binom{u}{i_1}   \sum_{j_1+j_2=a} \binom{a}{j_1}  
     \sum_{x_0 =0}^{i_2 + j_2} \sum_{\substack{x_1 + x_2 + x_3 =i_1 \\x_1,x_2,x_3 \geq 0}}   \sum_{\substack{y_1 + y_2 + y_3 =  j_1 \\y_1,y_2,y_3 \geq 0}}
      {i_2 +j_2 \choose x_0} (-1)^{x_0+x_1 + y_1}   { i_1 \choose x_1,x_2,x_3}\\ 
    &\times   { j_1 \choose y_1,y_2,y_3}  \frac{1}{2 \pi \mi} \int_{\substack{\Re(s) = \e_3\\ |\Im(s)| \leq t+1}} 
\frac{G(s)}{s }  \left( \frac{U}{2 \pi } \right)^{4s} 
{\zeta (1+2s) ^{2}}
  \zeta^{(x_0)} (2s )     
      g_{x_2} (s)  g_{y_2}(s)   \zeta^{(x_1+y_1)} ( 1+ 2s)\\
 &\times\mathscr{I}^{(x_3,y_3,i_2+j_2-x_0,0)}(0,0, 0,s)\\
&\times
 O \left( \delta^{-b-v}  t^{-2 \Re(s)+2\delta}    \exp\left( \pi |\Im(s)| +\pi \delta\right)  \left(\frac{1+2|s|^2}{t}\right)  \right)
     (-1)^{w+c}  \Gamma^{(w+c)}(2s)  \, ds.
\end{align*}
Clearly, $j_{{\bf x}, \mathrm{error}} \ll U^{4\e_3 } T^{-2\e_3 + 2\delta -1} \ll T^{-\frac{1}{2}}$ since $0<\delta< \frac{1}{10}$ and $0<\e_3<0.15$. Note that
the integrals in $j_{{\bf x},0}$ and $j_{{\bf x},1}$ can be extended to the contour $\Re(s) = \e_3$ at the cost of $O(T^{-2})$ since $|\Im(s) |> t+ 1$ parts of the integrals are $\ll T^{-2}$.

Observe that the integral in \eqref{i-x-0} (over the full line $\Re(s) = \e_3$) is 
\begin{align*}
\frac{1}{2 \pi \mi} \int_{(\e_3)}  \frac{1}{2^{x_1 +y_1 + x_2 +y_2 +w+c+10}} 
  \left(\frac{U^2}{t}\right)^{2s}  \frac{1}{s^{x_1 +y_1 + x_2 +y_2 +w+c+11}} J_0(s) \,   ds .
\end{align*}    
We then move this integral to $\Re(s) = -\frac{1}{4} + 2\delta$, where the integral on the new vertical line is $   \ll  \left( \frac{U^2}{t} \right)^{-\frac{1}{2} + 4\delta}$ since $J_0(s)\ll \frac{1}{1+|\Im (s)|^{100}}$, and we see that the residue of the pole at $s=0$ is 
\begin{align*}
&\frac{1}{2^{x_1 +y_1 + x_2 +y_2 +w+c+10}} \frac{1}{(x_1 +y_1 + x_2 +y_2 +w+c+10)!} \\
&\times  \left. \frac{d^{x_1 +y_1 + x_2 +y_2 +w+c+10}}{d s^{x_1 +y_1 + x_2 +y_2 +w+c+10}}  
 \left( \left(\frac{U^2}{t}\right)^{2s}  J_0(s) \right)\right|_{s=0}\\
 &=\frac{1}{2^{x_1 +y_1 + x_2 +y_2 +w+c+10}} \frac{1}{(x_1 +y_1 + x_2 +y_2 +w+c+10)!}\\
 & \times\sum_{k = 0}^{x_1 +y_1 + x_2 +y_2 +w+c+10} {x_1 +y_1 + x_2 +y_2 +w+c+10 \choose k}  \log^{k} \left(\frac{U^4} {t^2}\right) J_0^{(x_1 +y_1 + x_2 +y_2 +w+c+10-k)} (0).
\end{align*}
Similarly, the integral in \eqref{i-x-1} is 
\begin{align*}
\frac{1}{2 \pi \mi} \int_{(\e_3)}  \frac{\pi}{2^{x_1 +y_1 + x_2 +y_2 +w+c+10}} 
  \left(\frac{U^2}{t}\right)^{2s}  \frac{1}{s^{x_1 +y_1 + x_2 +y_2 +w+c+10}} J_1(s) 
      \, ds.
      \end{align*}
Moving the integral to $\Re(s) = -\frac{1}{4} + 2\delta$, we derive that the new integral on  $\Re(s) = -\frac{1}{4} + 2\delta $ is 
$   \ll  \left( \frac{U^2}{t} \right)^{-\frac{1}{2} + 4\delta}$, and
 the residue of the pole  at $s=0$ is
  \begin{align*}
 & \frac{\pi}{2^{x_1 +y_1 + x_2 +y_2 +w+c+10}}\frac{1}{(x_1 +y_1 + x_2 +y_2 +w+c+9)!} \\
& \times \sum_{k = 0}^{x_1 +y_1 + x_2 +y_2 +w+c+9} {x_1 +y_1 + x_2 +y_2 +w+c+9 \choose k} \log^{k} \left(\frac{U^4} {t^2}\right) J_1^{(x_1 +y_1 + x_2 +y_2 +w+c+9-k)} (0).
\end{align*}
Gathering everything above together, we have completed the proof for \eqref{iden-ix}.

The functions $J_0(s)$ and $J_1(s)$ are holomophic for $-\frac{1}{4} + 2\delta < \Re(s) < \frac{1}{2}$ because $\mathscr{I}(0,0, 0,s)$ is holomorphic in this region due to Proposition \ref{Hlemma}.  The upper bound for $J_0(s)$ and $J_1(s)$ is trivial since $|G(s)|$ decays.

For $ i_{{\bf x},1}$, we know 
\[
i_{{\bf x},1} \ll  \log^{r_1' +r_2'} (t) \log^{x_1 +y_1 + x_2 +y_2 +w+c+9} \left(\frac{U^4} {t^2}\right) \ll (\log T)^{r_1'+r_2'+x_1 +y_1 + x_2 +y_2 +w+c+9}.
\]
Since
 $$r_1'+r_2'+x_1 +y_1 + x_2 +y_2 +w+c+9
      \le v+b +x_1 +y_1 + x_2 +y_2 +w+c+9
      \le v+b+ (i_1+j_1) +w+c+9,
 $$     
which is $ \le (a+b+c)+ (u+v+w) +9   \le 3+3+9 =15,$      
we establish \eqref{ix1}.
\end{proof}

By \eqref{I-O-g} and Lemma \ref{lem:iX}, we have 
\begin{align*}
 I_O  
& = 2\sum_{i'+i=3}  \frac{b_{i'}}{i!} 
   \sum_{j'+j=3}   \frac{b_{j'}}{j!} \sum_{u+v+w=i}  \binom{i}{u,v,w} 
    \sum_{a+b+c=j}
    \binom{j}{a,b,c}   \int_{-\infty}^{\infty} \omega(t)  i_{{\bf x},0}  \,  dt + O(T (\log T)^{15}) \\
   &  + O \left(T \left( \frac{T}{T_0} \right)^{1+C} \right),
\end{align*}
where $ i_{{\bf x},0}$ is defined in \eqref{i0}. Note that we only need to consider the case $i =j = 3$  since otherwise the above sum is $\ll T(\log T)^{i+j+10} \ll T(\log T)^{15}$. With the help of Maple,  it gives that 
 
\begin{equation}
\label{final11}
\begin{split}
 I_O 
 &  = \int_{-\infty}^{\infty} \omega(t)  \bigg(
-\frac{\log^4(t) \log^{12}(U^4/t^2)}{535088332800} + \frac{\log^5(t)\log^{11}(U^4/t^2)}{122624409600 }- \frac{\log^6(t)\log^{10}(U^4/t^2)}{66886041600} \\
&- \frac{\log^2(t)\log^{14}(U^4/t^2)}{119027426918400} 
+ \frac{\log^3(t)\log^{13}(U^4/t^2)}{6376469299200} - \frac{\log^{16}(U^4/t^2)}{171399494762496000}\bigg)
\mathscr{I}(0, 0, 0, 0)  \, dt\\
& + O(T (\log T)^{15})
  + O \left(T \left( \frac{T}{T_0} \right)^{1+C} \right).
  \end{split}
 \end{equation}  
 
Let  $n_1,n_2 \geq 0$ be integers  satisfying  $n_1+n_2 = 16$. As we work over $c_1T\leq t \leq c_2T$, we know  $\log t = \log T + O(1)$ and hence
\begin{equation}
\label{bino}
\begin{split}
&\int_{-\infty}^{\infty} \omega(t)   \log^{n_1}(t)  \log^{n_2}(U^4/t^2) \,dt \\
&=   (\log T + O(1))^{n_1}   (4 (1-\varepsilon)\log T - 2 \log T + O(1) )^{n_2} \int_{-\infty}^{\infty} \omega(t) \,dt\\
&= ((\log T)^{n_1} + O((\log T)^{n_1-1}))  ((2 - 4\varepsilon)^{n_2} ( \log T)^{n_2} + O(( \log T)^{n_2-1}))
\int_{-\infty}^{\infty} \omega(t) \,dt \\
&= ((2 - 4\varepsilon)^{ n_2} ( \log T)^{n_1 + n_2}   + O((\log T)^{n_1+n_2 -1})) \int_{-\infty}^{\infty} \omega(t) \,dt\\
& =  ( 2 ^{ n_2} ( \log T)^{16}   + O (\varepsilon ( \log T)^{16} )+ O((\log T)^{15}))  \int_{-\infty}^{\infty} \omega(t) \,dt. 
\end{split}
\end{equation}
Note that in the second equality, the terms $O((\log T)^{n_1-1})$ and $O((\log T)^{n_2-1})$ are valid even if $n_1=0$ or $n_2=0$.  

Finally, plugging \eqref{bino} into \eqref{final11}, we derive
\begin{equation*}
\begin{split}
 I_O&  =  \bigg(
-\frac{2^{12}}{535088332800} + \frac{2^{11}}{122624409600 }- \frac{2^{10}}{66886041600} - \frac{2^{14}}{119027426918400} + \frac{2^{13}}{6376469299200} \\
&- \frac{2^{16}}{171399494762496000}\bigg)
\mathscr{I}(0, 0, 0, 0)  \int_{-\infty}^{\infty} \omega(t) \, dt + O(\varepsilon T (\log T)^{16})
  + O\left(T \left( \frac{T}{T_0} \right)^{1+C} \right)\\
  &= -\frac{13381 \mathscr{I}(0, 0, 0, 0)}{2615348736000} \int_{-\infty}^{\infty} \omega(t)  (\log T)^{16} \, dt  + O(\varepsilon T (\log T)^{16})
  + O\left(T \left( \frac{T}{T_0} \right)^{1+C} \right).
    \end{split}
   \end{equation*}
We see from \eqref{a4D} that $\mathscr{I}(0, 0, 0, 0) = a_4$. It completes the proof of  Proposition \ref{offdiagonal}.

\section{Proof of Proposition \ref{Hlemma} and Lemmata on $ \mathcal{G}_{s,t}^{(v,b)}(1,1)$ and  $\HH_s^{(u,a)}(1,1)$}
In this section, we shall give a proof for Proposition \ref{Hlemma}. Then we obtain an asymptotic formula for $\mathcal{G}_{s,t}^{(v,b)}(1,1)$ in \eqref{Gbvs}. We shall also simplify $\HH_s^{(u,a)}(1,1)$ that appeared in \eqref{Hua}.

\subsection{Proof of Proposition \ref{Hlemma}}
\begin{lem} \label{G4formulae}
Let $z \in \mathbb{C}$ and $q \in \mathbb{N}$. Then
\begin{equation*}
   |G_4(z,q)| \le 2^{\omega(q)} \tau_4(q) \prod_{p \mid q}  (1+p^{-\Re(z)})^3  (1+p^{\Re(z)-1}),
\end{equation*}
where  $\omega(q)$ is the number of distinct primes of $q$. In addition, if $1-r \le \Re(z) \le 1+r$, for some $r \in (0,1)$, then 
\begin{equation}
   \label{G4boundB}
     |G_4(z,q)| \le (32)^{\omega(q)} \tau_4(q) q^{r}. 
\end{equation}

Moreover, for any prime $p$ and $j \in \mathbb{N}$, we have 
 \begin{equation}
\begin{split}
   \label{G4zpj}
   G_{4}(z,p^j)  & =   \frac{p}{p-1} \left(  \tau_4(p^j)  \mathcal{Q}_j(p^{-z}) 
  - p^{z-1}  \tau_4(p^{j-1}) \mathcal{Q}_{j-1}(p^{-z}) \right),
 \end{split}
 \end{equation}
 where $  \mathcal{Q}_0(x)=1$ and
 \begin{equation}
  \label{Qj}
  \mathcal{Q}_j(x)  = 1-\frac{j}{j+1} 3x + \frac{j}{j+2} 3x^2 - \frac{j}{j+3} x^3.
\end{equation}
  In the cases, $j=1,2$, we obtain from \eqref{G4zpj} and \eqref{Qj} that
\begin{equation}
   \label{G4up}
   G_{4}(z,p) = \frac{p}{p-1}(4 -p^{z-1}-6p^{-z}+4p^{-2z}-p^{-3z}),
\end{equation}
and 
\begin{equation}
\begin{split}
   \label{G4up2}
   G_{4}(z,p^2)  
  & = \frac{p}{p-1} ( (10 -4p^{z-1}) + (-20 +6p^{z-1} )p^{-z} + (15-4p^{z-1})p^{-2z} + (-4+p^{z-1} )p^{-3z} ) . 
 \end{split}
 \end{equation}
\end{lem}

\begin{proof}
We begin with some facts about $g_4(z,p^j)$. 
In \cite{Ng}, it was proven that 
\begin{equation}
  \label{ngid}
  g_{4}(z,p^j) = \tau_{4}(p^j) 
  \mathcal{Q}_j(p^{-z}) ,
\end{equation}
where $\mathcal{Q}_j$ is defined in \eqref{Qj}.
Furthermore, for  $j \ge 1$,  we have the identity
\begin{equation*}
    \mathcal{Q}_j(x)
  =  j x^{-j} \int_{0}^{x} t^{j-1}(1-t)^{3} \, dt.
\end{equation*}
Note that performing the integration gives the formula in \eqref{Qj} and that
$ \mathcal{Q}_j(x)$ is a degree-three polynomial such that $ \mathcal{Q}_j(0)=1$. By the triangle inequality, we have the bound 
\begin{equation*}
|\mathcal{Q}_j(x)| \le 1 + 3|x| + 3|x|^2 +|x|^3 = (1+|x|)^3. 
\end{equation*}

We now compute $G_4(z,p^j)$. By definition \eqref{Gkdefn}, we know
\begin{align*}
  G_{4}(z,p^j) & = \sum_{d \mid p^j} \frac{\mu(d) d^z}{\phi(d)}
  \sum_{e \mid d} \frac{\mu(e)}{e^z} g_4 \left(z, \frac{p^j e}{d}\right) 
  = g_4(z,p^j)-\frac{p^z}{p-1} g_4(z,p^{j-1})+ \frac{1}{\phi(p)} g_4(z,p^j) \\
  & = \frac{p}{p-1} g_4(z,p^j)- \frac{p^z}{p-1}g_4(z,p^{j-1}).
 \end{align*}
By inserting \eqref{ngid} in the last expression, we establish \eqref{G4zpj}.
This implies that for $x=\Re(z)$,
 \begin{align*}
 | G_{4}(z,p^j)| & \le  \frac{p}{p-1}  \tau_4(p^j) \left( (1+p^{-x})^3 + p^{x-1} \frac{\tau_4(p^{j-1})}{\tau_4(p^j)}   (1+p^{-x})^3 \right) \\
 & = \frac{p}{p-1}  \tau_4(p^j) \left( (1+p^{-x})^3 + p^{x-1} \frac{j}{j+3}   (1+p^{-x})^3 \right)  \\
 & \le   \frac{p}{p-1}  \tau_4(p^j)  (1+p^{-x})^3  (1+p^{x-1}). 
 \end{align*}
From multiplicativity, it follows that 
 \[
   |G_4(z,q)| \le 2^{\omega(q)} \tau_4(q) \prod_{p \mid q}  (1+p^{-\Re(z)})^3  (1+p^{\Re(z)-1}).
 \]
Note that if $1-r \le \Re(z) \le 1+r$ for some $r \in (0, \frac{1}{10})$, then $|G_4(z,q)| $ is bounded by
 \begin{align*}
 2^{\omega(q)} \tau_4(q) \prod_{p \mid q}  (1+p^{-\Re(z)})^3  (1+p^{\Re(z)-1}) \le 2^{\omega(q)} \tau_4(q)  (2^{\omega(q)})^3  \prod_{p \mid q} (2 p^{r})  = (32)^{\omega(q)} \tau_4(q)  q^r. 
 \end{align*}
\end{proof}

We are now in a position to prove Proposition \ref{Hlemma}.

\begin{proof}[Proof of Proposition \ref{Hlemma}]
In this proof, we let $\sigma =\Re(s)$.  From the bound $|c_{q}(r)|  \le (q,r)$ and \eqref{G4boundB} of Lemma \ref{G4formulae}, it follows that 
$ \mathscr{H}(u_1,u_2,u_3,s)$, defined in \eqref{curlyH}, is absolutely convergent in $\Re(s)  > \frac{1}{2} + 2 \delta'$. 
Furthermore, since $c_{q}(r) = \sum_{d \mid (q,r)} d \mu(\tfrac{q}{d})$, we have 
\begin{align*}
 \mathscr{H}(u_1,u_2,u_3,s) &=
     \sum_{r=1}^{\infty}
\sum_{q=1}^{\infty} \frac{c_q(r) G_4 (1+u_1,q) G_4(1+u_2,q)}{q^{2+u_3} r^{2s-u_3}} 
       =  \sum_{q=1}^{\infty} \alpha_{q} \sum_{r=1}^{\infty} \frac{1}{r^c} 
   \sum_{d \mid q, d \mid r} d \mu(\tfrac{q}{d}) ,
\end{align*}
where $\alpha_{q} =  \frac{G_{4}(1+u_1,q)G_{4}(1+u_2,q)  }{q^{2+u_3}}$ and $c=2s-u_3$.
Thus,  
\begin{equation*}
\begin{split}
\mathscr{H}(u_1,u_2,u_3,s) & = 
\sum_{q=1}^{\infty} \alpha_{q} \sum_{d \mid q} d 
    \mu(\tfrac{q}{d}) \sum_{r \ge 1, d \mid r} \frac{1}{r^c}  
   = \sum_{q=1}^{\infty} \alpha_{q} 
  \sum_{d \mid q}  \frac{d \mu(\tfrac{q}{d})}{d^c} \zeta(c)   
   =\zeta(c)    \mathscr{H}^{*}(u_1,u_2,u_3,s),
\end{split}
\end{equation*}
where
\begin{equation*}
   \mathscr{H}^{*}(u_1,u_2,u_3,s)     =   \sum_{q=1}^{\infty} \alpha_{q}  q^{1-c} \sum_{d \mid q} \frac{\mu(d)}{d^{1-c}}.
\end{equation*}
For any prime $p$ and $j \ge 1$, we have $
    \sum_{d \mid p^j} \frac{\mu(d)}{d^{1-c}}
    = 1- \frac{1}{p^{1-c}}$.  By multiplicativity,
 \begin{align}  \label{ellsum} 
 \begin{split}
  \mathscr{H}^{*}(u_1,u_2,u_3,s) & = \prod_{p} \left(
  1 + \sum_{j=1}^{\infty} 
  \frac{G_{4}(1+u_1,p^j)G_{4}(1+u_2,p^j)}{(p^j)^{2+u_3}}
  (p^j)^{1+u_3-2s} ( 1- p^{2s-1-u_3})
  \right)  \\
&   = \prod_{p} \left( 
  1 + \sum_{j=1}^{\infty} G_{4}(1+u_1,p^j)G_{4}(1+u_2,p^j)
        \cdot \frac{  1-p^{2s-1-u_3} }{ (p^j)^{1+2s}}
  \right).
\end{split}
\end{align}
We aim to simplify the above expression with the brackets. At this point, it will be convenient to introduce the following notation.  Let 
\begin{equation}
  \label{variables}
U=p^{-1},  \ W = p^{-1-2s}, \ X_i=p^{-u_i}, 
\end{equation}
for $i=1,2,3$, where $|u_i| \le \delta' $ as required in \eqref{cond-ui}. 
Observe that 
\begin{align*}
  |W| & \le p^{-1-2 \sigma}, \\
   p^{-\delta'} \le  |X_i| & \le p^{\delta'} 
\end{align*}
for  $i=1,2,3$.  Also, we shall set
\begin{equation}
  \label{Tjfj}
  T_j  = G_{4}(1+u_1,p^j)G_{4}(1+u_2,p^j) \text{ and }  f_j  =  \frac{  1-p^{2s-1-u_3} }{ (p^j)^{1+2s}}
\end{equation}
for $j \in \mathbb{Z}_{\ge 0}$
Note that we have 
\begin{equation*}
   \mathscr{H}^{*}(u_1,u_2,u_3,s) = \prod_{p} \left(1 + \sum_{j=1}^{\infty} T_j f_j \right). 
\end{equation*}
We aim to simplify this further. 
It follows from  \eqref{variables} that for $j \in \mathbb{N}$,
\begin{equation}
  \label{fjformula}
  f_j =   W^j - X_{3} U^2 W^{j-1}. 
\end{equation}
Also, by  \eqref{G4up} and \eqref{G4up2}, for $i=1,2$,
\begin{align}
  \label{G4uip}
   & G_{4}(1+u_i,p) = \frac{1}{1-U} \sum_{j=0}^{3} \alpha_j(X_i) U^j, \\
    & G_{4}(1+u_i,p^2) = \frac{1}{1-U}   \sum_{j=0}^{3}  \beta_j(X_i) U^j, \nonumber
\end{align}
where 
\begin{equation*}
 \alpha_0(X) = 4-X^{-1}, \, \alpha_1(X) = -6X, \, \alpha_2(X) = 4X^2, \, \alpha_3(X) = -X^3, 
\end{equation*}
\begin{equation*}
 \beta_0(X) = 10-4X^{-1}, \, \beta_1(X) = -20X+6, \, \beta_2(X) = 15X^2-4X, \, \beta_3(X) = -4X^3+X^2. 
\end{equation*}
Note that we have the bounds 
\begin{equation}
  \label{alphabetabds}
  \alpha_0(X_i),\beta_0(X_i) \ll p^{\delta'} \text{ and } \alpha_j(X_i), \beta_j(X_i) \ll p^{j \delta'}  
\end{equation}
for $j=1,2,3.$
By  \eqref{Tjfj} and \eqref{G4uip}, we know
\begin{equation*}
\begin{split}
  T_1 & =  \left( 1-U \right)^{-2} \left( \sum_{j=0}^{3} \alpha_j(X_1) U^j \right) \left( \sum_{j'=0}^{3} \alpha_{j'}(X_2) U^{j'} \right) 
   =   \left( 1-U \right)^{-2} \sum_{k=0}^{6} A_k U^k= \sum_{k=0}^{\infty} \tilde{A}_k U^k,
\end{split}
\end{equation*}
where 
\begin{equation}
  \label{AktildeAk}
A_k = \sum_{\substack{j+j'=k \\ 0 \le j,j' \le 3}} \alpha_j(X_1) \alpha_{j'}(X_2) 
\text{ and }
  \tilde{A}_k = \sum_{\substack{ i+j=k \\  0 \le j \le 6 }} (i+1) A_j.
\end{equation}
It follows from \eqref{alphabetabds} and \eqref{AktildeAk} that 
\begin{equation*}
  \tilde{A}_0 \ll p^{2 \delta'}, \, \tilde{A}_j \ll p^{(j+1) \delta'} \text{ for } j=1, \ldots, 6,\text{ and } \tilde{A}_j \ll (j+1)p^{7 \delta'} \text{ for } j \ge 7. 
\end{equation*}
We now have 
\begin{equation*}
  T_1 f_1 = \left( \sum_{k=0}^{\infty} \tilde{A}_k U^k \right) (W-X_3 U^2). 
\end{equation*}
Expanding this out, we find that 
\begin{equation*}
  T_1 f_1 =  \tilde{A}_0 W + O(p^{3\delta'-2}) + O \left( \left(  \sum_{k=1}^{\infty} |\tilde{A}_k| U^k  \right) (p^{-1-2 \sigma} + p^{-2+\delta'})  \right).
\end{equation*}
Observe that 
\begin{equation*}
  \sum_{k=1}^{\infty} |\tilde{A}_k| U^k \ll \sum_{k=1}^{6} p^{(k+1)\delta'-k} + p^{14\delta' -7}  \ll p^{2 \delta' -1}, 
\end{equation*} 
as long as $\delta' < \frac{1}{2}$,  and thus 
\begin{equation}\label{T1f1-exp}
   T_1 f_1 =  \tilde{A}_0 W + O(p^{3\delta-2} +p^{2 \delta'-2-2 \sigma} +p^{3\delta'-3})
   =  \tilde{A}_0 W + O(p^{3\delta'-2} +p^{2 \delta'-2-2 \sigma} ).
\end{equation} 
  Since 
\begin{equation*}
    \tilde{A}_0 = A_0 = \alpha_0(X_1) \alpha_0(X_2) = (4-X_1^{-1}) (4-X_{2}^{-1}) = 16 -4X_{1}^{-1} - 4X_{2}^{-1} + (X_1 X_2)^{-1},
\end{equation*} 
and the error term in \eqref{T1f1-exp} is bounded by $O(p^{3 \delta'+\vartheta(\sigma)})$, 
it follows that 
\[
  1+T_1 f_1 =  1+ \frac{16}{p^{1+2s}}-
   \frac{4}{p^{1+2s-u_1}}
  - \frac{4}{p^{1+2s-u_2}}+
   \frac{1}{p^{1+2s-u_1-u_2}} + O(p^{3 \delta'+\vartheta(\sigma)}).
\]
Therefore,  for $\sigma>\frac{1}{2} + 2\delta'$, we can factor out some zeta factors from \eqref{ellsum},  which leads to
\[
     \mathscr{H}^{*}(u_1,u_2,u_3,s) =  \frac{\zeta(1+2s)^{16} \zeta(1+2s-u_1-u_2) }{ \zeta(1+2s-u_1)^4 \zeta(1+2s-u_2)^4 }
      \mathscr{I}(u_1,u_2,u_3,s) ,
\]
where 
\begin{equation*}
\begin{split}
    \mathscr{I}(u_1,u_2,u_3,s) = \prod_{p}  \mathscr{I}_p(u_1,u_2,u_3,s),
\end{split}
\end{equation*}
and
\begin{equation}
\begin{split}
  \label{Ipdefn}
  \mathscr{I}_p(u_1,u_2,u_3,s)& = \left(1-\frac{1}{p^{1+2s}} \right)^{16}  \left(1-\frac{1}{p^{1+2s-u_1-u_2}} \right)
   \left(1-\frac{1}{p^{1+2s-u_1}} \right)^{-4}  \left(1-\frac{1}{p^{1+2s-u_2}} \right)^{-4}  \\
   & \times  \left( 
  1 + \sum_{j=1}^{\infty} G_{4}(1+u_1,p^j)G_{4}(1+u_2,p^j)
        \cdot \frac{  1-p^{2s-1-u_3} }{ (p^j)^{1+2s}}
  \right).
\end{split}
\end{equation}
We note that 
\begin{equation}
\begin{split}
 \prod_p \mathscr{I}_p(0,0,0,0)& = \prod_p \left(1-\frac{1}{p} \right)^{9}    \times  \left( 
  1 + \sum_{j=1}^{\infty} G_{4}(1,p^j)^2
        \cdot \frac{  1-p^{-1} }{ (p^j)}
  \right) = a_4
\end{split}
\label{a4D}
\end{equation}
and that this is shown on p. 595 of \cite{CG} starting at equation (43) and in the text below it. 
We now aim to provide an analytic continuation of   $\mathscr{I}(u_1,u_2,u_3,s)$. Observe that 
\begin{equation}
   \label{Iplocal}
    \mathscr{I}_p(u_1,u_2,u_3,s) =  \left( 1+\sum_{j=1}^{\infty} T_j f_j \right) \Pi,
\end{equation} 
where
\begin{equation}
   \label{Pi}
    \Pi = (1 - W)^{16}(1- X_1^{-1} X_2^{-1} W)(1-X_1^{-1} W)^{-4} (1-X_2^{-1}  W)^{-4}.
\end{equation}
To simplify the local factor \eqref{Iplocal}, we first write $\Pi  = \sum_{j=0}^{\infty} a_j W^j$
with 
\begin{equation*}
\begin{split}
 a_0 & =1, \\ 
  a_1 & =  4X_1^{-1} +4 X_2^{-1} - (X_1 X_2)^{-1} - 16,  \\
 a_2 & =  10/X_2^2 + 4(4/X_1 - 1/(X_1 X_2) - 16)/X_2 + 10/X_1^2 + 4(-1/(X_ 1 X_2) - 16)/X_1 + 16/(X_1 X_2)\\
 & + 120.
\end{split}
\end{equation*}
Observe that $a_1 = -\tilde{A}_0$.  
It follows from \eqref{Pi} that $ |a_j| \ll j^{7} p^{(j+1) \delta'}$ and thus 
\begin{align*}
\sum_{j=3}^{\infty} a_j W^j \ll \sum_{j=3}^{\infty} j^7 p^{(j+1) \delta'} \left( \frac{1}{p^{1+2 \sigma}} \right)^j
 \ll p^{\delta'}  \left( \frac{p^{\delta'}}{p^{1+2 \sigma}} \right)^3 
 =\frac{p^{4 \delta'}}{p^{3+6 \sigma}} \ll p^{4 \delta' +\vartheta(\sigma)}. 
\end{align*}
Therefore, we have 
\begin{equation*}
    \Pi =   (1+a_1 W + a_2 W^2)+ O ( p^{4 \delta' +\vartheta(\sigma)} ).
\end{equation*} 
To bound the terms with  $j \ge 3$ to \eqref{Iplocal}, we make use of 
$|T_j| \le (32)^2 \tau_4(p^j)^2 (p^j)^{2 \delta'}$
and \eqref{fjformula}  to obtain 
\begin{align*}
\begin{split}
   \sum_{j=3}^{\infty}  T_j f_j    &  \ll 
  \sum_{j=3}^{\infty}  \tau_{4}(p^j)^2 p^{j 2 \delta'} \left(  \left( \frac{1}{p^{1+2 \sigma}} \right)^{j} +   \frac{p^{\delta'}}{p^2} 
   \left( \frac{1}{p^{1+2 \sigma}} \right)^{j-1} \right)  \\
&\ll  p^{6 \delta' } \left(  \left( \frac{1}{p^{1+2 \sigma}} \right)^{3} +   \frac{p^{\delta'}}{p^2} 
   \left( \frac{1}{p^{1+2 \sigma}} \right)^{2} \right)  
       \le  p^{7 \delta'} \left(   \frac{1}{p^{3+6 \sigma}}   +  
    \frac{1}{p^{4+4 \sigma}} \right)  
     \ll p^{7 \delta'} p^{\vartheta(\sigma)}.
\end{split}
\end{align*}
We now analyse $T_2 f_2$. By \eqref{Tjfj} and \eqref{G4up2}, we get 
\begin{equation*}
\begin{split}
  T_2 & =  \left( 1-U \right)^{-2} \left( \sum_{j=0}^{3} \beta_j(X_1) U^j \right) \left( \sum_{j'=0}^{3} \beta_{j'}(X_2) U^{j'} \right) 
   =   \left( 1-U \right)^{-2} \sum_{k=0}^{6} B_k U^k =  \sum_{k=0}^{\infty} \tilde{B}_k U^k,
\end{split}
\end{equation*}
where 
\begin{equation*}
B_k = \sum_{\substack{j+j'=k \\ 0 \le j, j' \le 3}} \beta_j(X_1) \beta_{j'}(X_2)  \text{ and }  \tilde{B}_k = \sum_{\substack{ i+j=k \\  0 \le j \le 6 }} (i+1) B_j.
\end{equation*}
With these observations in hand, we find
\begin{equation*}
  B_0 \ll p^{2 \delta'}, \, B_j \ll p^{(j+1) \delta'} \text{ for } j=1, \ldots, 6,
\end{equation*}
\begin{equation*}
  \tilde{B}_0 \ll p^{2 \delta'}, \,  \tilde{B}_j \ll p^{(j+1) \delta'} \text{ for } j=1, \ldots, 6, \text{ and }  \tilde{B}_j \ll p^{7 \delta'} \text{ for } (j+1) \ge 7. 
\end{equation*}
We then arrive at 
\begin{align*}
  T_2 f_2  & = \left( \sum_{k=0}^{\infty} \tilde{B}_k U^k \right) (W^2-X_3 U^2 W) = \tilde{B}_0 W^2  + O(p^{3 \delta'-3 -2\sigma}) \\
&  + O \left( \left(  \sum_{k=1}^{\infty} |\tilde{B}_k| U^k  \right) (p^{-2-4 \sigma} + p^{\delta'-3-2\sigma})  \right), 
\end{align*} 
which gives
$$
  T_2 f_2 = \tilde{B}_0 W^2  + O(p^{3 \delta'-3 -2\sigma}+p^{2\delta'-3-4 \sigma} + p^{3\delta'-4-2\sigma}) )  
  = \tilde{B}_0 W^2  + O(p^{3 \delta'+\vartheta(\sigma)}).
$$
%
%
Putting everything together, we find from \eqref{Iplocal} that
\begin{align*}
\begin{split}
     \mathscr{I}_p(u_1,u_2,u_3,s) 
 & = \left( 1+ \tilde{A}_0 W +\tilde{B}_0 W^2
  +O(p^{7 \delta'} p^{\vartheta(\sigma)} )
  \right) \left(1+a_1 W + a_2 W^2 +   O(p^{4 \delta'+\vartheta(\sigma)}) \right) \\
  & = 1 + (\tilde{B}_0 +\tilde{A}_0 a_1 + a_2) W^2 + O(p^{7 \delta'} p^{\vartheta(\sigma)} ).
\end{split}
\end{align*}
Using Maple, one may check  $\tilde{B}_0 +\tilde{A}_0 a_1 + a_2 =Y_p$,
where $Y_p$ is defined in \eqref{Yp}. From which, we see $ \mathscr{I}(u_1,u_2,u_3,s)$ is holomorphic in $\Re(s) > -\frac{1}{4}+ \delta'$ and $\mathscr{H}(u_1,u_2,u_3,s)$ 
has a meromorphic continuation to $\Re(s) > -\frac{1}{4}+\delta'$ with the exception of the poles listed in 
 \eqref{Hpoles}.
\end{proof}

\subsection{Lemmata on $ \mathcal{G}_{s,t}^{(v,b)}(1,1)$ and  $\HH_s^{(u,a)}(1,1)$} We assume $\Re(s) = \e_3$ and recall from \eqref{Gbvs} that  
$$
       \mathcal{G}_{s,t}^{(v,b)}(1,1) =
    \frac{\partial^{v}}{\partial z_1^v}  \frac{\partial^{b}}{\partial z_2^b}
    \left.   \left( 
\frac{\Gamma(z_2-s-\frac{1}{2}+\mi t ) }{ \Gamma(s-z_1+\frac{3}{2}+\mi t )  } 
+ 
\frac{\Gamma(z_1-s-\frac{1}{2}-\mi t)}{\Gamma(s-z_2+\frac{3}{2}- \mi t) }
   \right)   \right|_{z_1=z_2=1}.
$$
Making the variable change $z_2=1-\be$ and $z_1=1-\al$ for the first term and making the variable change
$z_1=1-\be'$ and $z_2=1-\al'$ for the second term, we obtain
\begin{align} \label{changeofvar}
 \begin{split}
   & \mathcal{G}_{s,t}^{(v,b)}(1,1) \\
    & = (-1)^{b+v}
   \frac{\partial^{b}}{\partial \be^b}
       \frac{\partial^{v}}{\partial \al^v}
   \left.   
\frac{\Gamma(\frac{1}{2}-\be-s+\mi t ) }{ \Gamma(\frac{1}{2}+\al +s+\mi t )  }  \right|_{\al=\be=0}  
+   (-1)^{b+v}
 \frac{\partial^{b}}{\partial (\al')^b}
       \frac{\partial^{v}}{\partial (\be')^v} 
         \left. 
\frac{\Gamma(\frac{1}{2}-\be'-s-\mi t)}{\Gamma(\frac{1}{2}+\al' +s-\mi t) }
    \right|_{\al'=\be'=0}.
    \end{split}
\end{align} 

We will need the following technical lemma to proceed.  The proof can be found in \cite[Lemma 4.6]{HN}.
\begin{lem} \label{Stirling}Let $|a_{i_1}|, |b_{i_2}| \leq \delta$. Let $\eta_1 \in (0,\tfrac{1}{2})$ and $A>0$. Assume 
$$
\Re(a_{i_1}+s_1) \in [0,A] \text{ and } \Re(b_{i_2}+s_2) \in [0, \tfrac{1}{2}-\eta_1]\cup[\tfrac12+\eta_1,\tfrac32-\eta_1].
$$
(i)  Assume $\Re(s_1+s_2+a_{i_1}+b_{i_2})\leq1$. When $|\Im(s_1)|\leq t+1$ and $|\Im(s_2)|\leq t+1$, we have
 \begin{align*}
 \begin{split}
 \frac{\Gamma(\frac12-b_{i_2}-s_{2}\pm \mi t)}{\Gamma(\frac12+a_{i_1}+s_{1}\pm \mi t)}
 &=t^{-(s_{1}+s_{2}+a_{i_1}+b_{i_2})}\exp\left(\mp \mi\frac{\pi}{2}(s_{1}+s_{2}+a_{i_1}+b_{i_2})\right)\\
 &\times\left(1+O\left(\frac{1+|s_{1}|^2+|s_{2}|^2}{t}\right)\right).
  \end{split}
 \end{align*}
(ii) When $|\Im(s_1)|\geq t+1$ or $|\Im(s_2)|\geq t+1$,  we have 
\begin{equation*}
\frac{\Gamma(\frac12-b_{i_2}-s_{2}+\mi t )}{\Gamma(\frac12+a_{i_1}+s_{1}+\mi t )}\ll \frac{\Im(s_1)^2+\Im(s_2)^2}{t^2}e^{\frac{\pi}{2}|\Im(s_1+s_2)|}.
\end{equation*}
(iii) Assume $\Re(s_1+s_2+a_{i_1}+b_{i_2})\geq1$. When $|\Im(s_1)|\leq t+1$ and $|\Im(s_2)|\leq t+1$, we have
\begin{equation*}
\frac{\Gamma(\frac12-b_{i_2}-s_{2}\pm \mi t)}{\Gamma(\frac12+a_{i_1}+s_{1}\pm \mi t)}\ll \frac{\Im(s_1)^2+\Im(s_2)^2 }{t^{2}}e^{\pm\frac{\pi}{2}\left(\Im(s_1+s_2)\right)}.
\end{equation*}
\end{lem}
Using Lemma \ref{Stirling}, we can express  $ \mathcal{G}_{s,t}^{(v,b)}(1,1)$ as a combinatoric sum as follows.  
\begin{lem}
   \label{FancyG} Let $ \mathcal{G}_{s,t}^{(v,b)}(1,1)$ be given as in  \eqref{Gbvs}.  Assume $\Re(s) = \e_3$, $0<\delta < \frac{1}{10}$, and $0<\delta < \e_3 <0.15$. \\
   (i) Let $|\Im(s)| \leq t+ 1$. We have
\begin{align*}
    \mathcal{G}_{s,t}^{(v,b)}(1,1)
&= (-1)^{b+v}
   \sum_{r_2'= 0}^b \sum_{r_1' = 0}^{v}  (-1)^{r_1' + r_2'} {v \choose r_1'} {b \choose r_2'} (\log t)^{r_1'+r_2'}  t^{-2s}  \left(  \mi\frac{\pi}{2}\right)^{v+b-r_1'-r_2'} \\
  & \times
 \left( (-1)^{v+b-r_1'-r_2'}\exp\left( -\mi\pi s\right)  + \exp\left( \mi\pi s\right) \right) \\
& + 
O \left( \delta^{-b-v}  t^{-2 \Re(s)+2\delta}    \exp\left( \pi |\Im(s)| +\pi \delta\right)  \left(\frac{1+2|s|^2}{t}\right)  \right).
\end{align*} 
(ii) Let $|\Im(s)| > t+ 1$. We have
\[
 \mathcal{G}_{s,t}^{(v,b)}(1,1) \ll \delta^{-b-v}  \frac{\Im(s)^2}{t^2}e^{\pi|\Im(s)|}.
\]
\end{lem}
\begin{proof}
Recall \eqref{changeofvar}. It follows from Lemma \ref{Stirling}(i) that for $|\Im(s)| \leq t+1$,
\begin{align}
\frac{\Gamma(\frac{1}{2}-\be-s+\mi t ) }{ \Gamma(\frac{1}{2}+\al +s+\mi t )  }  = 
t^{-(2s+\al+\be)}\exp\left(- \mi\frac{\pi}{2}(2s+\al+\be)\right) \left(1+O\left(\frac{1+2|s|^2}{t}\right)\right).
\label{minus-t}
\end{align}
Observe that the $v$-th derivative, with respect to $\al$, of the main term above is 
\begin{align*}
 & \frac{\partial^{v}}{\partial \al^v}  \left( t^{-(2s+\al+\be)}\exp\left(- \mi\frac{\pi}{2}(2s+\al+\be)\right) \right)\\
&= \sum_{r_1' = 0}^{v}  (-1)^{r_1'} {v \choose r_1'} (\log t)^{r_1'}  t^{-(2s+\al+\be)} 
 \exp\left(- \mi\frac{\pi}{2}(2s+\al+\be)\right)  \left( - \mi\frac{\pi}{2}\right)^{v-r_1'}.
\end{align*}
By a direct computation, we see that the $b$-th derivative, with respect to $\be$, of the above expression is 
\begin{align}\label{minus t main}
 \begin{split}
 & \frac{\partial^{b}}{\partial \be^b} \frac{\partial^{v}}{\partial \al^v}  \left( t^{-(2s+\al+\be)}\exp\left(- \mi\frac{\pi}{2}(2s+\al+\be)\right) \right) \\
 &= \sum_{r_2'= 0}^b \sum_{r_1' = 0}^{v}  (-1)^{r_1' + r_2'}  {v \choose r_1'} {b \choose r_2'}(\log t)^{r_1'+r_2'}   t^{-(2s+\al+\be)}  \left( - \mi\frac{\pi}{2}\right)^{v+b-r_1'-r_2'}
\exp\left(- \mi\frac{\pi}{2}(2s+\al+\be)\right). 
  \end{split}
\end{align}

Now, we handle the error term in \eqref{minus-t}. Note that
\[
t^{-(2s+\al+\be)}    \exp\left(- \mi\frac{\pi}{2}(2s+\al+\be)\right)  O\left(\frac{1+2|s|^2}{t}\right)
\]
is holomorphic  in the  region $|\alpha|, |\beta|\leq \delta$. By the Cauchy integral formula, we have
\begin{align*}
&\left. \frac{\partial^{b}}{\partial \be^b} \frac{\partial^{v}}{\partial \al^v}  \left( t^{-(2s+\al+\be)}    \exp\left(- \mi\frac{\pi}{2}(2s+\al+\be)\right)  O\left(\frac{1+2|s|^2}{t}\right) \right)\right|_{\al = \be = 0}\\
& = \frac{b!v!}{(2\pi \mi)^2}  \int_{C(0,\delta)} \int_{C(0,\delta)} 
\left( t^{-(2s+\al+\be)}    \exp\left(- \mi\frac{\pi}{2}(2s+\al+\be)\right)  O\left(\frac{1+2|s|^2}{t}\right) \right)  \frac{1}{\be^{b+1} \al^{v+1}} \, d\al \,  d\be\\
&\ll \delta^{-b-v}  t^{-2 \Re (s)+2\delta}    \exp\left( \pi \Im(s) +\pi \delta\right)  \left(\frac{1+2|s|^2}{t}\right) ,
\end{align*}
where $C(0,\delta)$ is the circle centred at $0$ with radius $\delta$. This estimate, together with \eqref{minus t main}, gives 
\begin{align}
 \begin{split}
& \left. \frac{\partial^{b}}{\partial \be^b} \frac{\partial^{v}}{\partial \al^v} 
\frac{\Gamma(\frac{1}{2}-\be-s+\mi t ) }{ \Gamma(\frac{1}{2}+\al +s+\mi t )  } \right|_{\al = \be = 0}\\
 &=\sum_{r_2'= 0}^b \sum_{r_1' = 0}^{v}  (-1)^{r_1' + r_2'} {v \choose r_1'} {b \choose r_2'} (\log t)^{r_1'+r_2'}  t^{-2s}  \left( - \mi\frac{\pi}{2}\right)^{v+b-r_1'-r_2'} 
\exp\left(- \mi\pi s\right)   \\
& + 
O \left( \delta^{-b-v}  t^{-2 \Re(s)+2\delta}    \exp\left( \pi \Im(s) +\pi\delta\right)  \left(\frac{1+2|s|^2}{t}\right)  \right).
\label{equ-minus-1}
 \end{split}
\end{align}

Similarly,  by Lemma \ref{Stirling}(i), we have
\begin{equation}
 \label{Stirling2}
\frac{\Gamma(\frac{1}{2}-\be'-s-\mi t)}{\Gamma(\frac{1}{2}+\al' +s-\mi t) }
  = 
t^{-(2s+\al'+\be')}\exp\left( \mi\frac{\pi}{2}(2s+\al'+\be')\right) \left(1+O\left(\frac{1+2|s|^2}{t}\right)\right).
\end{equation}
It can be derived that
\begin{align*}
 & \frac{\partial^{b}}{\partial \al'^b}  \frac{\partial^{v}}{\partial \be'^v}  \left( t^{-(2s+\al'+\be')}\exp\left( \mi\frac{\pi}{2}(2s+\al'+\be')\right) \right)\\
&=\sum_{r_2' = 0}^{b}  \sum_{r_1' = 0}^{v}  (-1)^{r_1'+r_2'} (\log t)^{r_1' + r_2'}  t^{-(2s+\al'+\be')} \left(  \mi\frac{\pi}{2}\right)^{v+b-r_1'-r_2'}
 \exp\left( \mi\frac{\pi}{2}(2s+\al'+\be')\right)  .
\end{align*}
Again, applying the Cauchy integral formula  gives 
\begin{align*}
&\left. \frac{\partial^{b}}{\partial \al'^b}  \frac{\partial^{v}}{\partial \be'^v}   \left( t^{-(2s+\al'+\be')}    \exp\left( \mi\frac{\pi}{2}(2s+\al'+\be')\right)  O\left(\frac{1+2|s|^2}{t}\right) \right) \right|_{\al' =\be' = 0}\\
&\ll \delta^{-b-v}  t^{-2 \Re (s)+2\delta}    \exp\left( -\pi \Im(s) +\pi \delta\right)  \left(\frac{1+2|s|^2}{t}\right) ,
\end{align*}
and thus
\begin{align}
 \begin{split}
&\left. \frac{\partial^{b}}{\partial \al'^b}  \frac{\partial^{v}}{\partial \be'^v}     \left( \frac{\Gamma(\frac{1}{2}-\be'-s-\mi t)}{\Gamma(\frac{1}{2}+\al' +s-\mi t) }
   \right)  \right|_{\al' =\be' = 0} \\
    &=\sum_{r_2'= 0}^b \sum_{r_1' = 0}^{v}  (-1)^{r_1' + r_2'}{v \choose r_1'} {b \choose r_2'} (\log t)^{r_1'+r_2'}  t^{-2s}  \left(  \mi\frac{\pi}{2}\right)^{v+b-r_1'-r_2'} 
\exp\left( \mi\pi s\right) \\
& + 
O \left( \delta^{-b-v}  t^{-2 \Re(s)+2\delta}    \exp\left( -\pi \Im(s) +\pi\delta\right)  \left(\frac{1+2|s|^2}{t}\right)  \right).
\label{equ-minus-2}
  \end{split}
\end{align}
Finally, combining  \eqref{equ-minus-1} and \eqref{equ-minus-2}, we complete the proof of the first part of the Lemma.

Now we prove the second part. By the Cauchy integral formula, it follows that for  $|\Im(s)|> t+1$, 
\begin{align*}
\left. \frac{\partial^{b}}{\partial \be^b} \frac{\partial^{v}}{\partial \al^v}  \left( \frac{\Gamma(\frac{1}{2}-\be-s+\mi t ) }{ \Gamma(\frac{1}{2}+\al +s+\mi t )  }  \right)\right|_{\al = \be = 0}
 &= \frac{b!v!}{(2\pi \mi)^2}  \int_{C(0,\delta)} \int_{C(0,\delta)} 
 \frac{\Gamma(\frac{1}{2}-\be-s+\mi t ) }{ \Gamma(\frac{1}{2}+\al +s+\mi t )  }   \frac{1}{\be^{b+1} \al^{v+1}} \, d\al \,  d\be.
\end{align*}
Note that by Lemma \ref{Stirling},
\begin{align*}
\left| \frac{\Gamma(\frac12-\be-s+\mi t )}{\Gamma(\frac12+ \al+s +\mi t )} \right| \ll \frac{\Im(s)^2}{t^2}e^{\pi|\Im(s)|}.
\end{align*}
Therefore, 
\begin{align*}
\left. \frac{\partial^{b}}{\partial \be^b} \frac{\partial^{v}}{\partial \al^v}  \left( \frac{\Gamma(\frac{1}{2}-\be-s+\mi t ) }{ \Gamma(\frac{1}{2}+\al +s+\mi t )  }  \right)\right|_{\al = \be = 0}
 &\ll \delta^{-b-v}  \frac{\Im(s)^2}{t^2}e^{\pi|\Im(s)|}.
 \end{align*}
 In a similar way, we can derive 
 \[
 \left. \frac{\partial^{b}}{\partial \al'^b}  \frac{\partial^{v}}{\partial \be'^v}     \left( \frac{\Gamma(\frac{1}{2}-\be'-s-\mi t)}{\Gamma(\frac{1}{2}+\al' +s-\mi t) }
   \right)  \right|_{\al' =\be' = 0}\ll  \delta^{-b-v} \frac{\Im(s)^2}{t^2}e^{\pi|\Im(s)|} .
 \]
This completes the proof.
\end{proof}

%
%

\begin{lem}
\label{HHHFinal}
Let $\HH_s^{(u,a)}(1,1)$ be given as in \eqref{Hua}. Assume $\Re(s) = \e_3$. We have
\begin{align*}
&\HH_s^{(u,a)}(1,1)\\
& = {\zeta (1+2s) ^{2}}\sum_{i_1+i_2=u} \binom{u}{i_1}   \sum_{j_1+j_2=a} \binom{a}{j_1}  
     \sum_{x_0 =0}^{i_2 + j_2} \sum_{\substack{x_1 + x_2 + x_3 =i_1 \\x_1,x_2,x_3 \geq 0}}   \sum_{\substack{y_1 + y_2 + y_3 =j_1 \\y_1,y_2,y_3 \geq 0}}
      {i_2 +j_2 \choose x_0} (-1)^{x_0+x_1 + y_1}\\
     &  \times \zeta^{(x_0)} (2s )     
      { i_1 \choose x_1,x_2,x_3} { j_1 \choose y_1,y_2,y_3}   g_{x_2} (s)  g_{y_2}(s)   \zeta^{(x_1+y_1)} ( 1+ 2s)
 \mathscr{I}^{(x_3,y_3,i_2+j_2-x_0,0)}(0,0, 0,s),
\end{align*}
where 
$\mathscr{I}(0,0, 0,s)$ is defined in \eqref{def-scrI}, and 
\begin{equation}\label{def-gk}
g_k (s) := \zeta (1+ 2s)^7 \left. \frac{\partial^k}{\partial {z^k}}  \left( \frac{1}{\zeta(1 + 2s - z)^4} \right) \right|_{z=0}.
\end{equation}
Note that
 \begin{equation}\label{equ:g-function-0123}
 \begin{split}
g_0 (s)  &= {\zeta (1+ 2s)^3} ,\\
g_1 (s) & ={4 \zeta'(1+2s)\zeta(1+2s)^2},\\
g_2(s) & ={20 \zeta'(1+2s)^2 \zeta(1+2s)} - {4\zeta^{(2)}(1+2s)\zeta(1+2s)^2},\\
g_3(s) &= {120 \zeta'(1+2s)^3 } -{60 \zeta'(1+2s) \zeta^{(2)}(1+2s)\zeta(1+2s)} + {4\zeta^{(3)}(1+2s)\zeta(1+2s)^2}.
\end{split}
\end{equation}
The $g_j(s)$ are holomorphic everywhere with the exception of poles at $s=0$.
\end{lem}

\begin{proof}
Recall the relation between $\HH_s^{(u,a)}(1,1)$ and $\mathscr{H}(u_1,u_2,u_3,s)$ from \eqref{Hua} and the remark below it.  We only need to handle  $\mathscr{H}(u_1,u_2,u_3,s)$. By Proposition \ref{Hlemma}, for $\Re(s) = \e_3$, it follows that 
\begin{align*}
\frac{\partial^{i_2 +j_2}}{\partial {u_3}^{i_2 +j_2}} \mathscr{H}(u_1,u_2,u_3,s) 
&=  \frac{\zeta (1+2s) ^{16}\zeta ( 1+ 2s - u_1 -u_2)}{\zeta(1 + 2s - u_1)^4 \zeta(1 + 2s - u_2)^4}   
     \sum_{x_0 =0}^{i_2 + j_2} {i_2 +j_2 \choose x_0} (-1)^{x_0} \zeta^{(x_0)} (2s - u_3)\\
     &\times  \mathscr{I}^{(0,0,i_2 +j_2-x_0,0)}(u_1,u_2,u_3,s).
\end{align*}
By a direct calculation, we see that $\frac{\partial^{i_1}}{\partial {u_1}^{i_1}}   \frac{\partial^{i_2 +j_2}}{\partial {u_3}^{i_2 +j_2}}   \mathscr{H}(u_1,u_2,u_3,s)$ equals
\begin{align*}
 &   \frac{\zeta (1+2s) ^{16}}{ \zeta(1 + 2s - u_2)^4}   
     \sum_{x_0 =0}^{i_2 + j_2} {i_2 +j_2 \choose x_0} (-1)^{x_0} \zeta^{(x_0)} (2s-u_3 )    \sum_{\substack{x_1 + x_2 + x_3 =i_1 \\x_1,x_2,x_3 \geq 0}} { i_1 \choose x_1,x_2,x_3}\\
     &\times  (-1)^{x_1}\zeta^{(x_1)} ( 1+ 2s - u_1 -u_2) \\ 
     &  \times\frac{\partial^{x_2}}{\partial {u_1}^{x_2}} \left( \frac{1}{\zeta(1 + 2s - u_1)^4} \right) 
 \mathscr{I}^{(x_3,0,i_2+j_2-x_0,0)}(u_1,u_2, u_3 ,s),
\end{align*}
and thus  $\frac{\partial^{j_1}}{\partial {u_2}^{j_1}} \frac{\partial^{i_1}}{\partial {u_1}^{i_1}}   \frac{\partial^{i_2 +j_2}}{\partial {u_3}^{i_2 +j_2}}   \mathscr{H}(u_1,u_2,u_3,s)$ is equal to
\begin{align*}
&{\zeta (1+2s) ^{16}}
     \sum_{x_0 =0}^{i_2 + j_2} {i_2 +j_2 \choose x_0} (-1)^{x_0} \zeta^{(x_0)} (2s -u_3) \sum_{\substack{x_1 + x_2 + x_3 =i_1 \\x_1,x_2,x_3 \geq 0}} { i_1 \choose x_1,x_2,x_3} (-1)^{x_1}  \\
     &\times\frac{\partial^{x_2}}{\partial {u_1}^{x_2}} \left( \frac{1}{\zeta(1 + 2s - u_1)^4} \right) 
     \\
 &\times  
 \sum_{\substack{y_1 + y_2 + y_3 =j_1 \\y_1,y_2,y_3 \geq 0}} { j_1 \choose y_1,y_2,y_3} (-1)^{y_1} 
 \zeta^{(x_1+y_1)} ( 1+ 2s - u_1 -u_2) \frac{\partial^{y_2}}{\partial {u_2}^{y_2}} \left( \frac{1}{\zeta(1 + 2s - u_2)^4} \right) \\
&\times \mathscr{I}^{(x_3,y_3,i_2+j_2-x_0,0)}(u_1,u_2, u_3 ,s).
\end{align*}
Taking $u_1=u_2=u_3 =0$, we obtain
\begin{align*}
 \begin{split}
&  \left. \frac{\partial^{j_1}}{\partial {u_2}^{j_1}} \frac{\partial^{i_1}}{\partial {u_1}^{i_1}}   \frac{\partial^{i_2 +j_2}}{\partial {u_3}^{i_2 +j_2}} \mathscr{H}(u_1,u_2,u_3,s) \right |_{u_1=u_2=u_3=0}  \\
&={\zeta (1+2s) ^{2}}
     \sum_{x_0 =0}^{i_2 + j_2} \sum_{\substack{x_1 + x_2 + x_3 =i_1 \\x_1,x_2,x_3 \geq 0}}   \sum_{\substack{y_1 + y_2 + y_3 =j_1 \\y_1,y_2,y_3 \geq 0}}
      {i_2 +j_2\choose x_0} (-1)^{x_0+x_1 + y_1}\zeta^{(x_0)} (2s )     
      { i_1 \choose x_1,x_2,x_3} {j_1 \choose y_1,y_2,y_3}   \\
 & \times g_{x_2} (s)  g_{y_2}(s)   \zeta^{(x_1+y_1)} ( 1+ 2s)
 \mathscr{I}^{(x_3,y_3,i_2+j_2-x_0,0)}(0,0, 0,s),
 \end{split}
\end{align*}
where  $g_k (s)$ is defined as in  \eqref{def-gk}. Together with \eqref{Hua}, this completes the proof.
\end{proof}

\section{Appendix 1: Additive divisor conjecture for $\tau_k$ and $\tau_{\ell}$} \label{ADC}

This appendix is based on the ideas of \cite{DFI} and follows closely their notation and presentation. 
Let 
$f(x,y)$ be a smooth function compactly supported on $[X,2X] \times [Y,2Y]$, and let $\phi$ 
be an even smooth compactly supported function $\phi$
which satisfies $\phi(0)=1$.  Note that
\[
   D_{f;k,\ell} (r) = \sum_{m-n=r} \tau_k (m)  \tau_\ell (n)  f(m,n) \phi(m-n-r).
\]
We set $F(x,y)=f(x,y)\phi(x-y-r)$, and  we define $\delta(u)=1$ if $u=0$ and $\delta(u)=0$ otherwise. 
We have 
\[
  D_{f;k,\ell} (r) = \sum_{m,n \ge 1}  \tau_k (m)  \tau_\ell (n) F(m,n) \delta(m-n-r).
\]
Since $  \delta(n) = \sum_{q=1}^{\infty} \chiqs e ( \tfrac{dn}{q}) 
\Delta_q(n),$
 it follows that
\[
    D_{f;k,\ell} (r)  = \sum_{m,n \ge 1}  \tau_k (m)  \tau_\ell (n) F(m,n)
    \sum_{q=1}^{\infty} \ \chiq e ( \tfrac{d}{q}(m-n-r) ) 
\Delta_q(m-n-r).
\]
Set $E(x,y)=F(x,y)\Delta_q(x-y-r)$ and thus 
\begin{align}
    D_{f;k,\ell} (r) =   \sum_{q=1}^{\infty} \ \chiq e ( \tfrac{-dr}{q}) 
  \sum_{m \ge 1}   \tau_k (m)  e ( \tfrac{md}{q})
  \sum_{n \ge 1}     \tau_\ell (n)   e( \tfrac{-nd}{q}) E(m,n).
  \label{equ:1}
\end{align}
Recall the Mellin transform of the smooth function $E(x,y)$ is 
\[
\tilde{E} (z_1,z_2) := \int_0^{\infty}\int_0^{\infty} E(x, y) x^{z_1 -1} y^{z_2 -1} \,dx \,dy.
\]
By the inverse Mellin transform, we have 
\[
E(x,y) =\frac{1}{(2\pi \mi)^2} \int_{(c_1)} \int_{(c_2)} \tilde{E} (z_1,z_2)  x^{-z_1} y^{-z_2}  \,dz_2 \,d z_1,
\]
where $c_1,c_2 >0$.
Inserting this into \eqref{equ:1} gives 
\begin{align*}
    D_{f;k,\ell} (r)&=   \sum_{q=1}^{\infty} \ \chiq e ( \tfrac{-dr}{q}) \frac{1}{(2\pi \mi)^2} \int_{(c_1)} \int_{(c_2)} \tilde{E} (z_1,z_2)
  \sum_{m \ge 1}  \frac{ \tau_k (m)  e ( \tfrac{md}{q})}{m^{z_1}}
  \sum_{n \ge 1}    \frac{ \tau_\ell (n)   e ( \tfrac{-nd}{q})}{n^{z_2}} \,dz_2 \,d z_1.
\end{align*}
Define $
\mathcal{D}_k (s,\frac{d}{q} ):= \sum_{n=1}^{\infty}  \tau_k (n)  e ( \tfrac{nd}{q} )  n^{-s}$.
Then we have
\begin{align*}
   D_{f;k,\ell} (r) &=   \sum_{q=1}^{\infty} \ \chiq e ( \tfrac{-dr}{q}) \frac{1}{(2\pi \mi)^2} \int_{(c_1)} \int_{(c_2)} \tilde{E} (z_1,z_2)
\mathcal{D}_k(z_1,\tfrac{d}{q}) \mathcal{D}
_\ell (z_2,\tfrac{-d}{q}) \,dz_2 \,d z_1.
\end{align*}
Note that (see \cite{CG}) since $
\mathcal{D}_k  (s, \frac{a}{q} ) \sim q^{-s} \zeta^k (s) G_k (s,q)$, we expect that
\begin{align*}
  & D_{f;k,\ell} (r) \\
  & \sim   \sum_{q=1}^{\infty} \ \chiq e ( \tfrac{-dr}{q}) \frac{1}{(2\pi \mi)^2} \int_{(c_1)} \int_{(c_2)} \tilde{E} (z_1,z_2)
q^{-z_1} \zeta^k (z_1) G_k (z_1,q)   q^{-z_2} \zeta^\ell (z_2) G_{\ell} (z_2,q)   \,dz_2 \,d z_1.
\end{align*}
We next  simplify $\tilde{E} (z_1,z_2)$. Note that 
$$
\int_\mathbb{R} E(x, y)  y^{z_2 -1}  \,dy 
= \int_\mathbb{R} F(x,y) \Delta_q (x-y-r)  y^{z_2 -1}  dy =\int_\mathbb{R} F(x,x-u-r) \Delta_q (u)  (x-u-r)^{z_2 -1}  \,du 
$$
and thus
$$
\int_\mathbb{R} E(x, y)  y^{z_2 -1}  \,dy  \sim  F(x,x-r) (x-r)^{z_2-1}
=   f(x,x-r) (x-r)^{z_2-1} 
$$
since the behaviour of $\Delta_q(u)$ is similar to the Dirac delta function. 
Hence, we have
\[
\tilde{E} (z_1,z_2) \sim \int_0^{\infty} f(x,x-r) x^{z_1 -1 } (x-r)^{z_2 -1} dx.
\]
From the definition of the Ramanujan sum, it follows that $D_{f;k,\ell} (r)$ is equal to
\begin{align*}
      \frac{1}{(2\pi \mi)^2} \int_{(c_1)} \int_{(c_2)}  \zeta^k (z_1)\zeta^\ell (z_2) 
\sum_{q=1}^{\infty}  \frac{c_q (r) G_k (z_1,q)   G_\ell (z_2,q)}{ q^{z_1+z_2}}    \int_0^{\infty} f(x,x-r) x^{z_1 -1 } (x-r)^{z_2 -1} \,dx  \,dz_2 \,d z_1.
\end{align*}
Since $\zeta^k(z_1)$ and $\zeta^k(z_2)$ have poles at $z_1=1$ and $z_2=1$, respectively, we expect that 
$D_{f;k,\ell} (r)$ is asymptotic to
 \begin{align*}
   \frac{1}{(2\pi \mi)^2} \int_{\mathcal{B}_1} \int_{\mathcal{B}_2}  \zeta^k (z_1)\zeta^\ell (z_2) 
\sum_{q=1}^{\infty}  \frac{c_q (r) G_k (z_1,q)   G_\ell (z_2,q)}{ q^{z_1+z_2}}    \int_0^{\infty} f(x,x-r) x^{z_1 -1 } (x-r)^{z_2 -1} \,dx  \,dz_2 \,d z_1,
\end{align*}
where $\mathcal{B}_j = \{ z_j \in \mathbb{C} \ | \ |z_j-1| < r_j \}$ for $j=1,2$.  We believe that this last expression is the main term in the additive  divisor conjecture for $\tau_k$ and $\tau_{\ell}$.

\begin{conj}[$k$-$\ell$ additive divisor conjecture] \label{kldivconj}
There exists $C >0$ for which 
the following  holds.  Let $\varepsilon_0$ and $\varepsilon'$ be arbitrarily small positive constants.  Let $P > 1$, and let $X,Y > \frac{1}{2}$ satisfy $Y \asymp X$.  Let $f$
be a smooth function satisfying \eqref{fsupport} and \eqref{fcond}.
Then, in those cases where $X$ is sufficiently large (in absolute terms), one has
  \begin{align*}
      D_{f;k,\ell}(r)
    &    =     \frac{1}{(2\pi \mi)^2} \int_{\mathcal{B}_1} \int_{\mathcal{B}_2}  \zeta^k (z_1)\zeta^{\ell} (z_2)  
\sum_{q=1}^{\infty}  \frac{c_q (r) G_k (z_1,q)   G_{\ell} (z_2,q)}{ q^{z_1+z_2}}   \\
& \times \int_0^{\infty} f(x,x-r) x^{z_1 -1 } (x-r)^{z_2 -1} \, dx  \,dz_2 \,d z_1 
 + O ( P^{C} X^{\frac{1}{2}+\varepsilon_0} ), 
  \end{align*}
 uniformly for $1 \le |r| \ll X^{1-\varepsilon'}$, where for $i=1,2$, $\mathcal{B}_i = \{ z_i \in \mathbb{C} \ | \ |z_i-1| = r_i \} \subset \mathbb{C}$ 
 are circles, centred at $1$, of radii $r_i \in (0, \frac{1}{10})$,  and $c_q(r) =  \chiqs e ( \tfrac{-dr}{q} )$ is the Ramanujan sum. 
\end{conj}

\section{Appendix 2: Proof of Lemma \ref{fpartials}}\label{appendix2}
For $\e_1\in (0,\frac{1}{2}]$, $M \ll U^{2+\varepsilon_2}$, $N \asymp M$, $0 \ne r \ll \frac{M}{T_0} T^{\varepsilon_1}$,  and $(x,y) \in [M,2M] \times [N,2N]$, we claim
\begin{equation}
   \label{fpartialsbd}
  x^m y^n f_{r}^{(m,n)}(x,y) \ll  T^{4 \varepsilon_1} P^{n},
   \text{ where }
     P =T^{1+ \e_1 } T_0^{-1}.
\end{equation}
\begin{proof}
By \eqref{fMN}, the definition of $f_r$, we can write  $f_r(x,y) = W( \frac{x}{M} )
W ( \frac{y}{N} ) \phi(x,y)$, 
where
\[
 \phi(x,y) = 
\frac{1}{2 \pi \mi} \int_{(\varepsilon_1)} \frac{G(s)}{s} 
\left( \frac{1}{\pi^4 xy} \right)^s 
\frac{1}{T} \int_{-\infty}^{\infty} \left( 
1+\frac{r}{y} \right)^{-\mi t} g(s,t) \left( \frac{U}{t}\right)^{4s}w(t) \,dt \, ds
\]
for $x,y>0$ (and is $0$ otherwise). 
It suffices to prove that  for $x \asymp M$ and $y \asymp N$,
\begin{equation}
    \label{phipartialsbd}
      x^m y^n \phi^{(m,n)}(x,y) \ll  T^{4\e_1} P^{n}. 
\end{equation}
This is since by the generalized product rule and \eqref{phipartialsbd}, we have
\begin{equation*}
\begin{split}
  \left|f_{r}^{(m,n)}(x,y)\right| 
  & =  \left| \sum_{i_1+i_2=m} \binom{m}{i_1} W^{(i_1)}\left( \frac{x}{M} \right) M^{-i_1}
  \sum_{j_1+j_2=n}  \binom{n}{j_1} W^{(j_1)}\left( \frac{y}{N} \right) N^{-j_1} \phi^{(i_2,j_2)}(x,y) \right| \\
  & \le  \sum_{i_1+i_2=m} 2^m O_{i_1}(1) \left( \frac{x}{2} \right)^{-i_1} 
  \sum_{j_1+j_2=n} 2^n O_{j_1}(1)  \left( \frac{y}{2} \right)^{-j_1} \cdot x^{-i_2} y^{-j_2} \cdot
  O_{i_2,j_2}(T^{4\e_1} P^{j_2}) 
 \\                                                                                                                                                                                                              & =                                                                                                                                                                                                              \left(  \sum_{i_1=0}^{m} \sum_{j_1=0}^{n} O_{i_1,j_1,m,n}(1) \cdot  P^{-j_1} \right)
 T^{4\e_1}  P^{n}    x^{-m} y^{-n},
\end{split}
\end{equation*} 
where we used the fact that $W(u)=0$ for $u \ge 2$. By \eqref{cond3}, for all sufficiently large $T$, we know $P \ge 1$ and thus obtain \eqref{fpartialsbd}.

Now, we shall prove \eqref{phipartialsbd}. We first write 
\begin{equation}
  \label{phipartials}
  \begin{split}
  &\phi^{(m,n)}(x,y) \\
  &=
  \frac{1}{2 \pi \mi} \int_{(\varepsilon_1)} \frac{G(s)}{s} 
\left( \frac{1}{\pi^4 } \right)^s 
\frac{1}{T} \int_{-\infty}^{\infty} 
   \frac{\partial^m}{\partial x^m} \frac{\partial^n}{\partial y^n}  \left(x^{-s} y^{-s} \left( 
1+\frac{r}{y} \right)^{-\mi t} \right)  g(s,t)  \left(\frac{U}{t} \right)^{4s}\omega(t) \,dt \, ds. 
\end{split}
\end{equation}
As shown in \cite[pp. 56-57]{Ng}, when $\Re(s) = \varepsilon_1$, $x \asymp M$, $y \asymp N$, $t \asymp T$, $1 \le |r|  \ll \frac{M}{T_0} T^{\varepsilon_1} =o(M)$, and $P = ( \frac{T}{T_0} )T^{ \e_1} 
\ge 1$ (by \eqref{cond3}), one has
\begin{equation}
  \label{phi-inner-term}
   \frac{\partial^m}{\partial x^m} \frac{\partial^n}{\partial y^n}  \left(x^{-s} y^{-s} \left( 
1+\frac{r}{y} \right)^{-\mi t} \right)  
 \ll M^{-2 \varepsilon_1} (1+|s|)^{m+n} P^n x^{-m} y^{-n}. 
\end{equation}
Note that $\left|\frac{G(s)}{s}  \right| \le \frac{|G(s)|}{\varepsilon_1}$ for $\Re(s)=
\varepsilon_1$, and 
\[
  |g(s,t)| \ll \left( \frac{t}{2} \right)^{4 \varepsilon_1} (1 + O(|s|^2+1))
  \ll (1+|s|)^2  t^{4 \varepsilon_1}
\] 
 for $\Re(s) = \varepsilon_1$ and $ t \asymp T > 1$ (by Lemma \ref{lem:bd-g}(i)). Using \eqref{phipartials}, \eqref{phi-inner-term}, and these two bounds, combined with \eqref{cond2} and  \eqref{cond3}, we derive
$$
      x^m y^n \phi^{(m,n)}(x,y)   \ll      P^n  M^{-2 \varepsilon_1}
      \int_{(\varepsilon_1)} |G(s)|  \left( \frac{1}{T} 
      \int_{-\infty}^{\infty}  (1+|s|)^{m+n+2}  T^{4 \varepsilon_1} |\omega(t)| dt \right)   |ds| 
      \ll \left( \frac{T^4}{M^2} \right)^{\varepsilon_1} P^n 
$$
for $x \asymp M$, $y \asymp N$ and $m,n \in \mathbb{N} \cup \{ 0 \}$.
Finally, as $M \gg \frac{T_0}{T^{\varepsilon_1}} |r| \ge \frac{T_0}{T^{\varepsilon_1}} 
\ge T^{\frac{1}{2} -\e_1}$ (by \eqref{cond3}) and thus $\frac{T^4}{M^2} \ll T^{3+ 2\e_1} \le T^{4} $ whenever $\e_1\in(0, \frac{1}{2}]$, we complete the proof of Lemma \ref{fpartials}. 
\end{proof}


\begin{thebibliography}{IW}


\bibitem{Ar}
F. Aryan, {\it Binary and quadratic divisor problems}, Int. J. Number Theory 13 (2017), no. 6, 1457-1471.




\bibitem{Bl}
V. Blomer, {\it Shifted convolution sums and subconvexity bounds for automorphic $L$-functions},  Int. Math. Res. Not.  2004,  no. 73, 3905-3926. 



\bibitem{Ch} 
V. Chandee, {\it On the correlation of shifted values of the Riemann zeta function}, Q. J. Math. 62 (2011), no. 3, 545-572. 

\bibitem{CL}
 V. Chandee and X. Li, {\it The eighth moment of Dirichlet L-functions}, Adv. Math. 259 (2014), 339-375. 



\bibitem{CFKRS}
J.B. Conrey, D.W. Farmer, J.P. Keating, M.O. Rubinstein, and N.C. Snaith, {\it Integral moments of $L$-functions},  Proc. London Math. Soc. (3)  91  (2005),  no. 1, 33-104.

  

\bibitem{CGh} 
J.B. Conrey and A. Ghosh, {\it A conjecture for the sixth power moment of the Riemann zeta-function}, Internat. Math. Res. Notices 15 (1998), 775-780.



\bibitem{CG}
J.B. Conrey and S.M. Gonek, {\it High moments of the Riemann zeta-function}, Duke Math. J. 107 (2001), 577-604. 



A  Alaca,  S  Alaca,  KS  Williams),  pp.  75-85.  Berlin,  Germany:
Springer. 


  

\bibitem{DI}
J.M. Deshouillers and H. Iwaniec, {\it An additive divisor problem}. J. London Math. Soc. (2) 26 (1982), no. 1, 1-14.

\bibitem{DFI}
W. Duke, J.B. Friedlander, and H. Iwaniec, 
{\it A quadratic divisor problem}, 
Invent. Math. 115 (1994), no. 2, 209-217. 

\bibitem{GG}
D.A. Goldston and S.M. Gonek, {\it Mean value theorems for long Dirichlet polynomials and tails of Dirichlet series},  Acta Arith.  84  (1998),  no. 2, 155-192. 



\bibitem{HN}
A. Hamieh and N. Ng, 
{\it Long mean value of Dirichlet polynomials with  higher divisor coefficients}, preprint, 
{\tt https://arxiv.org/abs/2105.03525}. 

\bibitem{H}
G. Harcos, {\it An additive problem in the Fourier coefficients of cusp forms}, 
Math. Ann. 326 (2003), no. 2, 347-365. 

\bibitem{HL} 
G.H. Hardy and J.E. Littlewood, {\it Contributions to the theory of the Riemann zeta function and
the theory of the distribution of primes}, Acta. Math. 41 (1918), 119-196.  

\bibitem{Ha}
A. Harper,  {\it 
Sharp conditional bounds for moments of the Riemann zeta function}, preprint,
{\tt https://arxiv.org/abs/1305.4618}. 

\bibitem{HB}
D.R. Heath-Brown, {\it The fourth power moment of the Riemann zeta function},  
Proc. London Math. Soc. 38  (1979), no. 3, 385-422.

\bibitem{HB81}
D.R. Heath-Brown, {\it Fractional moments of the Riemann zeta-function}, J. London Math. Soc., 24 (1981),
no. 2, 65-78.


\bibitem{HY}
C.P. Hughes and M.P. Young, {\it
 The twisted fourth moment of the Riemann zeta function}, J. Reine Angew. Math. 641 (2010), 203-236.


\bibitem{In}
A.E. Ingham, {\it Mean-value theorems in the theory of the Riemann zeta function}, Proc. Lond. Math.
Soc. 27 (1926), 273-300.


\bibitem{Iv} 
 A. Ivi\'{c},  {\it The Riemann zeta-function. Theory and applications}, Reprint of the 1985 original [Wiley, New York; MR0792089], Dover Publications, Inc., Mineola, NY, 2003. 


\bibitem{Iv2}
A. Ivi\'{c}, {\it 
On the ternary additive divisor problem and the sixth moment of the zeta-function}, in Sieve methods, exponential sums, and their applications in number theory (Cardiff, 1995), 205--243,
London Math. Soc. Lecture Note Ser., 237, Cambridge Univ. Press, Cambridge, 1997. 

\bibitem{Iv3} A. Ivi\'{c}, {\it The general additive divisor problem and moments of the zeta-function}, in New Trends in Probability and Statistics, Vol. 4: Analytic and Probabalistic Methods in Number Theory (Palanga, Lithuania, 1996), VSP, Utrecht, Netherlands, 1997, 69-89.
Lectures on mean values of the Riemann zeta function

%


\bibitem{KS}
J.P. Keating and N.C. Snaith, {\it Random matrix theory and $\zeta(\tfrac{1}{2}+\mi t )$}, Comm. Math. Phys. 214 (2000), 57-89. 




\bibitem{Mo} Y. Motohashi, {\it Spectral theory of the Riemann zeta-function}, Cambridge Tracts in Mathematics, 127. Cambridge University Press, Cambridge, 1997. x+228 pp. 


\bibitem{Ng0} N. Ng, \textit{Large gaps between the zeros of the Riemann zeta function}, J. Number Theory 128 (2008), 509-556.

\bibitem{Ng}
N. Ng,  {\it The sixth moment of the Riemann zeta function and ternary additive divisor sums}, 
Discrete Analysis, 2021:6, 60 pages. 

\bibitem{NSW}
N. Ng, Q. Shen, and P.-J. Wong, {\it Shifted moments of the Riemann zeta function}, preprint.


\bibitem{RaSo} M. {Radziwi\l\l} and K. Soundararajan, {\it Continuous lower bounds for moments of zeta and L-functions}, Mathematika 59 (2013), no. 1, 119-128.

\bibitem{Ram78} K. Ramachandra,  {\it Some remarks on the mean value of the Riemann zeta-function and other Dirichlet series. I.}, Hardy–Ramanujan J. (1978), 1-15. 

\bibitem{Ram95} K. Ramachandra,  {\it  On the mean-value and omega-theorems for the Riemann zeta-function}, Tata Institute of Fundamental Research Lectures on Mathematics and Physics, 85, Springer-Verlag, 1995.


\bibitem{Sh}
 Q. Shen, {\it  The fourth moment of quadratic Dirichlet L-functions}, Math. Z. 298 (2021), no. 1-2, 713-745. 


\bibitem{So}
K. Soundararajan, 
{\it Moments of the Riemann zeta function}, Ann. of Math. (2) 170 (2009), no. 2, 981-993.

\bibitem{SY}
K. Soundararajan and M.P. Young, {\it The second moment of quadratic twists of modular $L$-functions}, J. Eur. Math. Soc. (JEMS) 12 (2010), no. 5, 1097-1116. 

\bibitem{Ti}
E.C. Titchmarsh, {\it The theory of the Riemann zeta function, second
edition}, Oxford University Press, New York, 1986.




\end{thebibliography}
\end{document}